\documentclass[a4paper,11pt]{amsart}

\let\oldtocsection=\tocsection
\let\oldtocsubsection=\tocsubsection
\let\oldtocsubsubsection=\tocsubsubsection
 
\renewcommand{\tocsection}[2]{\hspace{0em}\oldtocsection{#1}{#2}\bfseries}
\renewcommand{\tocsubsection}[2]{\hspace{1.8em}\oldtocsubsection{#1}{#2}}
\renewcommand{\tocsubsubsection}[2]{\hspace{4.4em}\oldtocsubsubsection{#1}{#2}}

\makeatletter
\renewcommand\subsection{\@startsection{subsection}{2}%
  \z@{-.5\linespacing\@plus-.7\linespacing}{.5\linespacing}%
  {\normalfont\scshape}}
\renewcommand\subsubsection{\@startsection{subsubsection}{3}%
  \z@{.5\linespacing\@plus.7\linespacing}{.5\linespacing}%
  {\normalfont\scshape}}

\makeatother

\usepackage[framemethod=tikz]{mdframed}

\usepackage[utf8]{inputenc}
\usepackage[T1]{fontenc}
\usepackage{tikz-cd}
\usepackage{paralist}
\usepackage{mathrsfs}
\usepackage{amssymb}
\usepackage{amsmath}
\usepackage{amsfonts}
\usepackage{enumitem}

\newtheorem{theorem}{Theorem}[section]

\newtheorem{lemma}[theorem]{Lemma}

\newtheorem{remark}[theorem]{Remark}
\newtheorem{Fact}[theorem]{Fact}

\newtheorem{observation}[theorem]{Observation}

\theoremstyle{definition}
\newtheorem{definition}[theorem]{Definition}

\usepackage[pdftex]{hyperref}
\hypersetup{
    colorlinks=true, 
    linktoc=all,     
    linkcolor=blue,  
    allcolors=blue,
}

\newcommand{\dom}[1]{\ensuremath{\mathrm{dom}}(#1)}

\newcommand{\power}{\ensuremath{\mathscr{P}}}

\newcommand{\set}[2]{\ensuremath{\{#1 \,|\, #2 \}}}
\newcommand{\seq}[2]{\ensuremath{\langle #1 \,|\, #2 \rangle}}
\newcommand{\restr}[2]{\ensuremath{#1 \! \upharpoonright \! #2}}

\newcommand{\Iff}{\Leftrightarrow}
\newcommand{\el}{\prec}

\newcommand{\sub}{\subseteq}
\newcommand{\super}{\supseteq}

\newcommand{\bb}{\mathbb}

\newcommand{\beq}{\begin{equation}}
\newcommand{\eeq}{\end{equation}}
\newcommand{\brm}{\begin{remark}\begin{rm}}
\newcommand{\erm}{\end{rm}\end{remark}}
\newcommand{\mx}{\mathrm}
\newcommand{\bce}{\begin{compactenum}}
\newcommand{\ece}{\end{compactenum}}

\newcommand{\cof}{\mathrm{cof}}
\newcommand{\cf}{\mathrm{cf}}

\newcommand{\Add}{\mathrm{Add}}

\newcommand{\R}{\bb{R}}
\newcommand{\Q}{\bb{Q}}
\newcommand{\Z}{\bb{Z}}
\renewcommand{\P}{\bb{P}}

\newcommand{\T}{\bb{T}}
\newcommand{\M}{\bb{M}}
\newcommand{\x}{\times}

\newcommand{\Coll}{\mathrm{Coll}}

\newcommand{\TP}{{\sf TP}}

\newcommand{\ZFC}{{\sf ZFC}}
\newcommand{\ZF}{{\sf ZF}}
\newcommand{\CH}{{\sf CH}}
\newcommand{\GCH}{{\sf GCH}}
\newcommand{\SCH}{{\sf SCH}}
\newcommand{\AC}{{\sf AC}}

\newcommand{\SR}{{\sf SR}}

\newcommand{\PD}{{\sf PD}}

\newcommand{\AP}{{\sf AP}}
\newcommand{\CT}{{\sf CT}}

\newcommand{\BPI}{{\sf BPI}}

\newcommand{\DC}{{\sf DC}}
\newcommand{\AD}{{\sf AD}}
\newcommand{\MM}{{\sf MM}}
\newcommand{\PFA}{{\sf PFA}}
\newcommand{\RC}{{\sf RC}}

\newcommand{\GMPii}{{\sf GMP}_{\omega_2}}
\newcommand{\GMP}{{\sf GMP}}
\newcommand{\GMPgen}[1]{{\sf GMP}_{#1}}

\newcommand{\ISP}{{\sf ISP}}
\newcommand{\ITP}{{\sf ITP}}
\newcommand{\SP}{{\sf SP}}
\newcommand{\ISPii}{{\sf ISP}_{\omega_2}}
\newcommand{\TPii}{{\sf TP}_{\omega_2}}
\newcommand{\SRii}{{\sf SR}(\omega_2)}

\newcommand{\ISPgen}[1]{{\sf ISP}_{#1}}
\newcommand{\ITPgen}[1]{{\sf ITP}_{#1}}
\newcommand{\TPgen}[1]{{\sf TP}_{#1}}
\newcommand{\SPgen}[1]{{\sf SP}_{#1}}

\newcommand{\wAGPY}{{\sf wAGP}_{\mathcal Y}}
\newcommand{\BC}{{\sf BC}}
\newcommand{\DSS}{{\sf DSS}}

\newcommand{\IGMP}{{\sf IGMP}}
\newcommand{\GM}{{\sf GM}}
\newcommand{\SGM}{{\sf SGM}}
\newcommand{\SATP}{{\sf SATP}}

\newcommand{\BA}{{\sf BA}(\omega_1)}
\newcommand{\BAii}{{\sf BA}(\omega_2)}
\newcommand{\BAg}{{\sf BA}}

\newcommand{\KH}{{\sf KH}(\omega_1)}

\newcommand{\KHgen}{{\sf KH}}
\newcommand{\nonKHgen}{{\sf KH}}

\newcommand{\SSR}{{\sf SSR}}

\newcommand{\CSR}{{\sf CSR}}

\newcommand{\wKH}{{\sf wKH}}
\newcommand{\nonwKH}{\neg {\sf wKH}}

\newcommand{\SH}{{\sf SH}(\omega_1)}
\newcommand{\SHii}{{\sf SH}(\omega_2)}

\newcommand{\SHg}{{\sf SH}}

\newcommand{\WP}{{\sf WC}(\omega_1)}
\newcommand{\nonWP}{\neg {\sf WC}(\omega_1)}
\newcommand{\WPg}{{\sf WC}}

\newcommand{\KP}{{\sf KC}}
\newcommand{\nonKP}{\neg {\sf KC}}

\newcommand{\MA}{{\sf MA}(\omega_1)}

\newcommand{\PT}{\mathrm{PT}}
\newcommand{\NPT}{\mathrm{NPT}}

\newcommand{\Chr}{\mathrm{Chr}}
\newcommand{\Color}{\mathrm{Col}}

\newcommand{\dense}{\omega_1\mbox{-dense}}
\newcommand{\ZFCc}{\ZFC + \ISPii + \FRP}
\newcommand{\ZFCcR}{\ZFC +  \RC + 2^\omega = \omega_2}

\newcommand{\SCFA}{{\sf SCFA}}
\newcommand{\FRP}{{\sf FRP}}
\newcommand{\GRP}{{\sf GRP}}
\newcommand{\RP}{{\sf RP}}
\newcommand{\WRP}{{\sf WRP}}

\begin{document}

\title[Compactness for small cardinals in mathematics \ldots]{Compactness for small cardinals in mathematics: principles, consequences, and limitations}

\author{Radek Honzik}
\address[Honzik]{
Charles University, Department of Logic,
Celtna{\' a} 20, Prague~1, 
116 42, Czech Republic
}
\email{radek.honzik@ff.cuni.cz}
\urladdr{logika.ff.cuni.cz/radek}

\thanks{
The author has been supported by Czech Science Foundation (GA{\v C}R) grant ``The role of set theory in modern mathematics'' (24-12141S)}

\begin{abstract}
We discuss some well-known compactness principles for uncountable structures of small regular sizes ($\omega_n$ for $2 \le n<\omega$, $\aleph_{\omega+1}$, $\aleph_{\omega^2+1}$, etc.), consistent from weakly compact (the size-restricted versions) or strongly compact or supercompact cardinals (the unrestricted versions). We divide the principles into \emph{logical principles}, which are related to cofinal branches in trees and more general structures (various \emph{tree properties}), and \emph{mathematical principles}, which directly postulate compactness for structures like groups, graphs, or topological spaces (for instance, countable chromatic and color compactness of graphs, compactness of abelian groups, $\Delta$-reflection, Fodor-type reflection principle, and Rado's Conjecture).

We also focus on \emph{indestructibility}, or \emph{preservation}, of these principles in forcing extensions. While preservation adds a degree of robustness to such principles, it also limits their provable consequences.  For example, several well-known mathematical problems decided by $V = L$ and by forcing axioms, in the opposite ways, i.e.\ Suslin Hypothesis, Whitehead's Conjecture, Kaplansky's Conjecture, and Baumagartner's Axiom,  are independent from some of the strongest forms of compactness at $\omega_2$. This is a refined version of Solovay's theorem that large cardinals are preserved by  small forcings and hence cannot decide many natural problems in mathematics.  Additionally, we observe that Rado's Conjecture plus $2^\omega = \omega_2$ is consistent with the negative solutions of some of these conjectures (as they hold in $V = L$), verifying that they hold in suitable Mitchell models.

Finally, we comment on whether the compactness principles under discussion are good candidates for axioms. We consider their consequences and the existence or non-existence of convincing unifications (such as Martin's Maximum or Rado's Conjecture). This part is a modest follow-up to the articles by 
Foreman ``Generic large cardinals: new axioms for mathematics?'' (1998) and Feferman et al.\ ``Does mathematics need new axioms?'' (2000).
\end{abstract}

	\maketitle

\tableofcontents

\section{Introduction}

There are many natural concepts in mathematics formulated in terms of compactness: given an infinite cardinal $\kappa$ and a structure $A$ of size $\kappa$, is it the case that a given property $\varphi$ holds in $A$ if and only if $\varphi$ holds in all substructures of $A$ of size $<\kappa$?   Consider the following examples for a cardinal $\kappa \ge \omega_2$:

\begin{enumerate}
\item Suppose $P$ of size $\kappa$ is a partially ordered set such that all suborders of size $< \kappa$ can be decomposed into countably many chains. Does it follow that $P$ can be decomposed into countably many chains?
\item Suppose $T$ is a tree of size $\kappa$ and every subtree of size $<\kappa$ can be decomposed into countably many antichains. Does it follow that $T$ can be decomposed into countably many antichains?
\item Suppose $G$ is a graph of size $\kappa$ and all its subgraphs of size $<\kappa$ have a countable chromatic number. Does it follow that $G$ has a countable chromatic number?
\item Suppose $A$ is an abelian group of size $\kappa$ and all its subgroups of size $<\kappa$ are free. Does it follow that $A$ is free?
\end{enumerate}

Since $\kappa$ is uncountable, the properties $\varphi$ in these examples are not first-order and therefore are not entailed by compactness of the usual first-order logic (denoted $L_{\omega,\omega}$). However, in most of the cases (and in all examples mentioned in the previous paragraph), the given property $\varphi$ is expressible in an infinitary logic $L_{\kappa,\kappa}$, which allows formulas of length $<\kappa$ with $<\kappa$ many quantifiers and connectives. If $\kappa$ is compact for $L_{\kappa,\kappa}$  and theories of size $\kappa$---we call such $\kappa$ \emph{weakly compact},\footnote{If there is no restriction on the size of theories,  $\kappa$ is called \emph{strongly compact} and the examples (1)--(4) are true for structures of unlimited size with respect to substructures of size $<\kappa$.}---then all four questions above are answered positively.  However, the usefulness of this form of compactness is limited by the fact that every weakly compact cardinal $\kappa$ is necessarily inaccessible and hence quite far away from the size of usual objects in mathematics.

One way of bringing the consequences of weak compactness down to small cardinals is to consider only specific principles which might consistently hold at accessible cardinals. For instance, as we will review, (2) can consistently hold at $\kappa = \omega_2$ (a consequence of Rado's Conjecture, see Section \ref{sec:graphs}) and (4) can hold at $\kappa = \aleph_{\omega^2+1}$ (see Section \ref{sec:ab}). However, for (1) (a local version of  Galvin's Conjecture) and (3) it is still open whether they can hold below a weakly compact cardinal (see Section \ref{sec:graphs}). These examples illustrate that it is unclear, a priori, which principles can consistently hold for small cardinals, and whether there are uniform methods to discover them.

We will survey recent development in this area, with a broader goal in mind of discussing whether compactness principles are good candidates for axioms in mathematics. This goal is an updated version of the original program proposed by G{\"o}del in \cite{G:famous} to look for consequences of large cardinal axioms in order to decide independent statements like the Continuum Hypothesis, $\CH$. By Solovay's observation that a large cardinal $\kappa$ is preserved by all forcings of size $<\kappa$ (see \cite{LS:m}), G{\"o}del's program necessarily fails for independent statements whose truth can be changed by small forcings. These include $\CH$ and many other, for instance all the principles we discuss in Section \ref{sec:preserve} like Suslin Hypothesis, Whitehead's Conjecture or Baumgartner's axiom.\footnote{Some important problems related to the reals are decided by large cardinals, though: If there is a supercompact cardinal (in fact, infinitely many Woodin cardinals with a measurable cardinal above them are enough), Projective Determinacy, $\PD$, holds and consequently all definable subsets of the reals are Lebesgue-measurable and have other regularity properties. We will observe in Section \ref{sec:con} that by combining results of Weiss and Steel, $\PD$ is already implied by a compactness principle $\ITP_{\omega_2}$ related to generalized trees of height $\omega_2$.}

An updated version of G{\"o}del's program, which considers  compactness principles at small cardinals, circumvents Solovay's observation and as such may have a greater chance of deciding these statements. For example, compactness principles related to trees at $\omega_2$ (see Section \ref{sec:lists}) imply the negation of $\CH$, even though they do not imply an upper bound for $2^\omega$. One can broaden G{\"o}del's program even further, and include \emph{forcing axioms} as candidates for new axioms. Axioms like $\PFA$, Proper Forcing Axiom, and $\MM$, Martin's Maximum, were shown to be extremely powerful and capable of deciding almost all traditional independent problems in mathematics (which are usually decided by $V = L$ in the opposite way).

Even though forcing axioms and compactness principles are sometimes treated as two distinct concepts, they share structural similarity because they can both  be formulated in terms of the existence of certain non-principal ultrafilters on infinite Boolean algebras. While compactness of $L_{\kappa,\kappa}$ generalizes the Boolean Prime Ideal Theorem, $\BPI$, and asserts the existence of non-trivial ultrafilters on Boolean algebras which are closed under countable intersections, forcing axioms generalize the Baire Category Theorem, $\BC$, by requiring that for all Boolean algebras in a certain class, there are ultrafilters which meet any given list of $\omega_1$-many dense open subsets.  It is remarkable that for an appropriately chosen class of Boolean algebras (derived from proper and semi-proper forcings), the latter concept related to $\BC$ at $\omega_2$ implies many compactness principles originating from the compactness of $L_{\kappa,\kappa}$, and moreover provides solutions of many  problems in mathematics seemingly unrelated to the existence of ultrafilters on Boolean algebras. This is a powerful extension of the method of forcing which from consistency results shifted to provable consequences of a single axiom. However, this generalization of $\BC$ seems to be at the moment tightly connected with cardinals $\omega_1$ (number of dense open sets) and $\omega_2$ (size of the continuum), leaving many problems outside its scope.\footnote{Martin's Axiom, $\mathsf{MA}$, for ccc partial orders can be asserted for more than $\omega_1$-many dense sets if $2^\omega > \omega_2$. However, $\PFA$ already implies $2^\omega = \omega_2$ and thus only $\omega_1$-dense open sets can be met. In general, there are provable restrictions for forcing axioms on higher cardinals; see for instance \cite{MR689980}, \cite{Shelah:FA} and \cite{todorčević2020inconsistentforcingaxiomomega2} for more details.} 

The inherent limitation of forcing axioms to the cardinal $2^\omega = \omega_2$ suggests that some other principles---such as compactness---might be considered to decide properties of larger structures. The investigation of compactness has the additional benefit of identifying principles which go beyond forcing axioms:  some compactness principles are provably false at $\omega_2$ (for instance compactness for abelian groups or chromatic compactness of graphs, see Theorem \ref{th:fullc} and Section \ref{sec:graphs}), or incompatible with forcing axioms (for instance Rado's Conjecture, see Section \ref{sec:graphs}).

An important aspect of discussion of compactness principles is their \emph{indestructibility} or \emph{preservation} with respect to various forcing notions. As we will review, most compactness principles originating from the compactness of $L_{\kappa,\kappa}$ are preserved by large classes of forcing notions. This might be interpreted positively from the philosophical perspective as lending a degree of robustness and stability to these principles. However, it also prevents them from deciding independent statements whose truth can be changed by forcings from these classes. This in a sense recreates Solovay's restriction for the compactness principles at small cardinals: statements independent from large cardinals such as Suslin Hypothesis or Whitehead's Conjecture we mentioned above remain independent from compactness principles such as Fodor-type Reflection Principle, $\FRP$, and Ineffable Slender Tree Property, $\ISPii$, as well (see Section \ref{sec:sep}).\footnote{It is worth observing that forcing axioms, i.e.\ generalizations of Baire's Category Theorem, behave differently in this respect: forcing axioms are always destroyed by adding just a single Cohen real (which adds an $\omega_1$-Suslin tree). However, fragility is not limited to forcing axioms and starts to appear as principles grow in strength: for instance, Weak Reflection Principle and Rado's Conjecture are destroyed by adding a single new real (see Section \ref{sec:graphs}).} With G{\"o}del's program in mind, these considerations suggest that compactness principles which are easier to destroy appear to be better candidates for new axioms.

The article is structured as follows. 

In Section \ref{sec:omega} we briefly review consequences of compactness of the usual  first-order logic $L_{\omega,\omega}$, and  single out specific compactness principles which we will discuss in the generalized setting of $L_{\kappa,\kappa}$ for an uncountable $\kappa$.

In Section \ref{sec:unctble} we discuss logical compactness principles. We chose to call them ``logical'' because they characterize, modulo inaccessibility, the compactness of various infinitary logics. We will briefly discuss infinitary logics in Section \ref{sec:inacc}, and turn to discussing compactness principles associated with them in Section \ref{sec:lists}. In Section \ref{sec:con} we list known consequences of these principles.

In Section \ref{sec:math} we focus on ``mathematical'' compactness principles, which directly postulate compactness for specific  structures like graphs, algebras, or topological spaces. An important difference from the logical principles is their dependence on stationary reflection: non-reflecting stationary subsets of $\kappa$ often suffice to construct incompact mathematical structures of size $\kappa$, though, importantly, not incompact trees which appear in the logical principles. Stationary reflection is relatively weak in terms of consequences, but if it is generalized to stationary subsets of $[\kappa]^\theta$, for some $\theta \le \kappa$,\footnote{For infinite cardinals $\theta \le \kappa$, $[\kappa]^\theta$ denotes the set of all subsets of $\kappa$ of size $\theta$. The case $[\kappa]^\omega$ is the most important one in many contexts (such as for the forcing axioms). The notation $\power_\theta(\kappa)$ is used to denote the set of all subsets of $\kappa$ of size $<\theta$. (There is no reason to have two different notations, but we follow the prevalent notational conventions.)} it often becomes a sufficient condition for compactness of many mathematical structures (we will specifically discuss the \emph{Fodor-type Reflection Principle, $\FRP$}, a consequence of Martin's Maximum, and the \emph{$\Delta$-reflection}).

In Section \ref{sec:canonical} we describe standard constructions for collapsing large cardinals which yield models with compactness principles.   

In Section \ref{sec:preserve}  we survey existing preservation results. We first review absolute theorems in Section \ref{sec:abs} and then model-related results in Section \ref{sec:model}, which are connected to the standard models mentioned in Section \ref{sec:canonical}.  In Section \ref{sec:sep} we illustrate  applications of preservation theorems by showing that many well-known consequences of Martin's Maximum, $\MM$, such as the Suslin Hypothesis, Whitehead's Conjecture, Baumgartner's Axiom and Kaplansky's Conjecture, and the maximal value of cardinal invariants, are all independent from a theory which contains a very strong fragment of the compactness-type consequences of $\MM$ (for example, the Fodor-type Reflection Principle or the strong tree properties at $\omega_2$). However, a lack of indestructibility alone does not automatically guarantee more consequences. We will discuss Rado's Conjecture---a principle incompatible with forcing axioms, which is destroyed by adding a single new real. We will show that the negation of the Suslin Hypothesis, the negation of Baumgartner's axiom and the negation of Whitehead's Conjecture are consistent with Rado's Conjecture $+$ $2^\omega = \omega_2$, but we will leave open whether the positive versions are consistent as well.\footnote{We define the positive versions of these conjectures to hold under $\PFA$. Thus Rado's Conjecture + $2^\omega = \omega_2$ does not decide these statements the way $\PFA$ does. For the consistency of the positive versions, new models of $\RC$, in addition to the Levy and Mitchell collapses, need to be found.}

In the final section, Section \ref{sec:eval}, we mention that the existence or non-existence of natural unifications among various compactness principles may serve as a good criterion for adopting them as new axioms. This discussion is a modest and limited follow-up to the articles of Foreman \cite{MR1648052} and Feferman et al.\ \cite{4}.

\brm
This expository article is intended for a wide audience interested in the applications of set-theoretical concepts and methods in general mathematics. The author acknowledges the support of grant \emph{The role of set theory in modern mathematics} (Czech Science Foundation, GAČR 24-12141S).\erm

\textbf{Acknowledgements.} The author wishes to thank (alphabetically) to Saka{\'e} Fuchino, Assaf Rinot, and Corey Switzer for motivating discussions and valuable feedback and suggestions regarding the article.

\subsection{Preliminaries and notation}\label{sec:note}

We define all compactness principles which appear in the main body of the text and which we discuss in some detail. We relax this convention for footnotes and sections where we discuss consequences of various principles (like in Section \ref{sec:con}), but we always give references to articles with more details and definitions.

Our notation is standard, as in Kanamori \cite{KANbook} and Jech \cite{JECHbook}. For more details on  large cardinals, combinatorial principles and forcing axioms, we refer the reader to monographs  \cite{KANbook} and  \cite{JECHbook}, the handbook volumes for set theory \cite{handbook123} and the handbook of set-theoretic topology \cite{topology}. The book \cite{modules} by Eklof and Mekler contains many additional set-theoretic results motivated by research into almost free abelian groups and modules.

There seems to be no comprehensive survey of compactness principles in the strict sense of the word (which is one of the reasons why we have written up this one), but Cummings \cite{CUM:comb} provides a clearly written survey of some of the notions related to compactness.

Finally, let us mention some specific conventions which we use: 

\begin{itemize}
\item We write $\omega_\alpha$ to denote infinite cardinals, with $\omega$ denoting the least infinite cardinal (the set of natural numbers).
\item All abbreviations of combinatorial principles are typeset using sans serif font for ease of reading, e.g.\ $\TP(\omega_2)$ for the tree property at $\omega_2$, $\ZFC$ for the Zermelo-Fraenkel set theory with the Axiom Choice ($\AC$), $\GCH$ for the Generalized Continuum Hypothesis, etc.
\item By the Singular Cardinal Hypothesis, $\SCH$, we will always mean the assertion that for every singular strong limit cardinal $\kappa$, $2^\kappa = \kappa^+$.\footnote{$\SCH$ is often viewed as a compactness principle provable in $\ZFC$ for uncountable cofinalities by Silver's theorem: if $\kappa$ is a singular strong limit cardinal of uncountable cofinality and $2^\mu = \mu^+$ for all $\mu <\kappa$ (in fact stationarily many such $\mu$ suffice), then $2^\kappa = \kappa^+$. This contrasts with the countable cofinality which is known to behave differently (Shelah's \emph{pcf theory} extends many of these results to countable cofinality, Shelah \cite{MR1318912} for more details). Shelah later extended Silver's theorem to a more general form of compactness related to abelian groups and other structures of singular size (including countable cofinalities) and proved in $\ZFC$ his \emph{singular compactness theorem}, see \cite{Shelah:ab}. See also Remark \ref{rm:R3}.}
\end{itemize}

\section{First-order compactness} \label{sec:omega}

Let us recall the compactness theorem for the first-order logic:

\begin{itemize}
\item ($\CT, \BPI$) Compactness theorem for the first-order logic: Given a first-order theory $T$ in an arbitrarily large language, $T$ has a model if and only if every finite subtheory of $T$ has a model. 
\end{itemize}

It was soon observed that $\CT$---a theorem about a specific logic---has many combinatorial equivalents or consequences which refer to  well-known mathematical structures like graphs, algebras or topological spaces. In this reformulation, the compactness theorem asserts that first-order properties of an \emph{infinite} structure are determined by the properties of all of its \emph{finite} substructures. Let us state some important examples which are paradigmatic for the generalization to an uncountable $\kappa$.

\brm
It is known that the compactness theorem $\CT$ is provable in $\ZFC$ but cannot be proved in $\ZF$. It is also known that $\ZF + \CT$ does not prove $\AC$ (not even its weakenings like the principle of Dependent Choices).  In the strict sense of the word, a ``compactness principle'' for $L_{\omega,\omega}$ should mean a principle derivable from $\CT$, without the use of (a form of) $\AC$. While this distinction may be justified for $L_{\omega,\omega}$ in the context of $\ZF$, it is of lesser importance  for compactness of $L_{\kappa,\kappa}$, $\kappa > \omega$, in the context of $\ZFC$:\footnote{$\AC$ adds some genuinely new consequences over $\CT$, but it can also be seen as a ``constructive version'' of $\CT$: The transfinite recursion theorem (provable in $\ZF$) together with $\AC$ provides explicit constructions of objects whose existence is postulated by $\CT$  without saying how they should be constructed (such as a recursive construction of an ultrafilter extending the Frechet filter on natural numbers using a well-ordering of $\power(\omega)$, or a construction of a completion of a theory using a well-ordering of its language). However, $\AC$  does not have similar benefits for $L_{\kappa,\kappa}$ because constructions using the transfinite recursion may not retain infinitary properties at  stages  of small  cofinalities (like $\sigma$-completeness of filters at stages of countable cofinality when a construction of a non-principal $\sigma$-complete ultrafilter is attempted). Compactness principles for $L_{\kappa,\kappa}$ therefore postulate the existence of the desired objects, but without a uniform principle for their construction (it is an open question whether there exists one, see also Remark \ref{rm:kAC}).} For instance, while Ramsey theorem on the size of homogenous sets in infinite graphs does not follow from $\CT$, its generalization to an inaccessible $\kappa$ is equivalent to a generalized version of K{\"o}nig's Lemma (the tree property) and also to the compactness of $L_{\kappa,\kappa}$ (see Theorem \ref{th:full}).
\erm

\brm \label{rm:size}
In order to compare the strength of  principles derivable from $\CT$ and $\AC$, we work in $\ZF$. This implies that assumptions must be stated more carefully: For instance, even if  $\ZF$ does not to prove that every infinite graph has either an infinite clique or an infinite independent set (Ramsey theorem), it does prove that every \emph{countable}  graph has either an infinite clique or an infinite independent set (if the domain of the graph is well-ordered, the usual proof works in $\ZF$).
\erm

The book \cite{Rubin:AC} by Howard and Rubin contains an extensive list of principles equivalent to $\AC$ and $\CT$ (Form 14 in \cite{Rubin:AC}) and lists many other principles which follow from $\AC$, with detailed references and results on their relative strengths. References for all statements below can be found in \cite{Rubin:AC}. See also Jech \cite{Jech:AC} for an extended discussion and proofs. 

The following are equivalent:

\begin{itemize}
\item The compactness theorem $\CT$.

\item The compactness theorem for propositional logic. The completeness theorems for propositional and first-order logic.

\item Ultrafilter theorem: Every filter over an infinite set $S$ can be extended into an ultrafilter.

\item Boolean Prime Ideal Theorem, $\BPI$: Every Boolean algebra has a prime (= maximal) ideal.

\item Consistency principle: For every binary mess $M$ on a set $S$ there is a function $f$ on $S$ which is consistent with $M$, introduced by Jech \cite{jech:tree}. This principle is studied today in the context of cofinal and ineffable branches in $(\kappa,\lambda)$-lists, see Section \ref{sec:lists} for more details.

\item Tychonoff's theorem for compact Hausdorff spaces: Every product of compact Hausdorff spaces is compact.

\item Every commutative ring with unit has a prime ideal (an ideal $I$ is prime if $ab \in I$ implies $a \in I$ or $b \in I$).

\item Finite chromatic numbers: If $G$ is a graph and there exists a natural number $n \ge 3$ such that  every finite subgraph of $G$ is  $n$-colorable, then $G$ itself is $n$-colorable.
\end{itemize}

Some other well-known principles follow from $\BPI$, but are strictly weaker. Let us state some examples which are relevant for us:

\begin{itemize}

\item  K{\"o}nig's Lemma that every infinite tree with finite levels has an infinite branch (Form 10 in \cite{Rubin:AC}).

\item Dilworth's decomposition theorem: If $P$ is an infinite partial order and there exists a natural number $n$ such that every antichain  in $P$ has size at most $n$, then $P$ can be decomposed into at most $n$ many chains. See \cite{T:Dil} for more details.\footnote{Dilworth's theorem can also be stated as follows: If $P$ is an infinite partial order and every finite suborder of $P$ can be decomposed into at most $n$ many chains, so can be the whole $P$. Mirsky's theorem asserts the dual principle formulated for antichains; interestingly, while Dilworth's theorem requires compactness for its proof, Mirsky's theorem can be proved in $\ZF$. Rado's Conjecture can  be seen as a generalization of Mirsky's theorem for certain tree orders, and Galvin's Conjecture as a generalization of Dilworth's theorem (it is known that Galvin's Conjecture implies Rado's Conjecture).} This principle is generalized as Galvin's and Rado's Conjectures, see Section \ref{sec:graphs}.

\item The existence of a Lebesgue non-measurable subset of the unit interval.

\end{itemize}

Some other principles follow from $\AC$ but do not follow from $\BPI$. Relevant for us are for instance the following (all strictly weaker than $\AC$):

\begin{itemize}

\item Baire category theorem, $\BC$, for compact Hausdorff spaces, which is equivalent to the axiom of Dependent Choices, $\DC$, and hence logically independent over $\ZF$ with respect to $\BPI$. An equivalent reformulation (modulo $\BPI$) of this principle for partial orders is Rasiowa--Sikorski lemma that for every partial order $\P$ and a list of countably many dense open sets, there is filter which meets all of them.

\item Ramsey theorem that every infinite graph contains either an infinite clique of an infinite independent set (Form 17 in \cite{Rubin:AC})). See also \cite{Blass:RT} that it does not imply the $\AC$ or $\BPI$, and does not follow from $\BPI$.

\item  Theorems that every subgroup of a free (abelian) group is free (attributed to Nielsen--Schreier for non-abelian free groups and to Dedekind for abelian groups). See \cite{H:sub} and \cite{KL:NS} for proofs that these theorems do not imply $\AC$ or $\BPI$, and do not follow from $\BPI$.

\end{itemize}

It is instructive to compare these principles with full equivalents of $\AC$, formulated in terms of algebraic and topological structures. For instance, the following are equivalent to $\AC$:

\begin{itemize}
\item In every vector space, every generating set contains a basis.
\item Every graph has a chromatic number.
\item Tychonoff's theorem for compact spaces: Every product of compact spaces is compact.
\item For every abelian group $G$ and its subgroup $H$, there exists a set of representatives for the cosets in the quotient $G/H$.
\end{itemize}

As we already mentioned in Remark \ref{rm:size},  compactness principles are formulated in $\ZF$ with reference to infinite structures and not for specific cardinalities, but the first infinite cardinal $\omega$ plays a special role in $\CT$: first-order properties of substructures of size $<\omega$ determine the whole infinite structure. When condering generalizations of $\CT$  we always assume the Axiom of Choice, so this distinction transforms into substructures of size $<\kappa$ reflecting up to the whole structure of size $\ge \kappa$ for some \emph{uncountable} cardinal $\kappa$. 

In principle, any consequence of $\CT$ (or $\AC$) can be considered for a generalization. However, some compactness principles at uncountable cardinals do not have a clear analogue on $\omega$ (for instance the notion of a free abelian group or stationary reflection), so it is more appropriate to start with some infinitary logic, study its compactness consequence and then try to apply them at small cardinals. We review some basic facts related to infinitary logics in Section \ref{sec:unctble}.

\brm \label{rm:kAC}
The various compactness principles in $\ZF$ form a complex hierarchy with $\AC$ at the top. In $\ZFC$, all these principles become provable equivalent and the hierarchy collapses. It is open  whether there is a similar ultimate (not outright inconsistent, i.e.\ consistent modulo some established large cardinals) compactness principle for an uncountable $\kappa$ which would imply all (or many) other.\footnote{\label{ft:kAC} Martin's Maximum, $\MM$, and its strengthenings like $\MM^{++}$,  imply many compactness principles at $\omega_2$ and can be considered as candidates for such principles at $\omega_2$, see Viale \cite{MR4713473} for a clearly written exposition and further references. For an inaccessible $\kappa$, a non-trivial embedding with critical point $\kappa$, $j: V\to V$, was introduced by Reinhardt and briefly considered as the ultimate large cardinal principle before being proved by Kunen \cite{KUNEN:v} to be inconsistent with $\ZFC$ (consistency with $\ZF$ is still open; see \cite{doi:10.1142/S0219061324500132} for recent developments with regard to $\ZF$).} We briefly discuss this unification problem---as a criterion for new axioms---in final Section \ref{sec:uni}.
\erm

\section{Compactness in logic}\label{sec:unctble}

We first consider compactness principles called \emph{(strong) tree properties} which are tightly connected with compactness of infinitary logics $L_{\kappa,\kappa}$. In fact, they fully characterize weakly compact, strongly compact and supercompact cardinals of $\kappa$ once the assumption of inaccessibility of $\kappa$ is added. For this reason we call them ``logical compactness principles'' to differentiate them from principles discussed in Section \ref{sec:math} which refer to specific mathematical structures such as graphs, groups, or topological spaces.

\subsection{Infinitary logics}\label{sec:inacc}

Suppose $\mathscr{L}$ is a logic in a broad sense (according to the examples below).\footnote{See Kanamori's book \cite{KANbook} for more details about large cardinals mentioned below and Keisler's book \cite{K:inf} or the handbook by Barwise and Feferman \cite{MR3728293} for more information about infinitary logics.}

\begin{definition}\label{def:sk}
We say that $\kappa$ is $\mathscr{L}$-compact if and only if for every set of sentences $A$ in $\mathscr{L}$, $A$ has a model if and only if all subsets $B \sub A$ with $|B|<\kappa$ have a model.
\end{definition}

Suppose $\kappa$ is an infinite regular cardinal. The  logic denoted $L_{\kappa,\kappa}$ allows formulas of length $<\kappa$ with $<\kappa$ many quantifiers, conjunctions and disjunctions (the size of the vocabulary of $L_{\kappa,\kappa}$ can in principle arbitrarily big). In this notation, $L_{\omega,\omega}$ denotes the usual first-order logic, and $\omega$ is $L_{\omega,\omega}$-compact.  If $\kappa > \omega$, it is for instance possible to express in $L_{\kappa,\kappa}$ the property of being a well-ordering, of having any fixed cardinality below $\kappa$ or the notion of a separable  topological space (see for instance Dickmann \cite{Dickmann} for more examples).

\begin{definition} A cardinal $\kappa>\omega$ is called \emph{strongly compact} if and only if $\kappa$ is $L_{\kappa,\kappa}$-compact over an arbitrarily large language, equivalently, without a limit on the size of the set of sentences $A$ in Definition \ref{def:sk}. If we limit the size of $A$ in Definition \ref{def:sk} to $|A|\le \kappa$ and $\kappa$ is $L_{\kappa,\kappa}$-compact in this weaker sense, we say that $\kappa$ is \emph{weakly compact}.
\end{definition}

\brm
A fine point is whether one allows infinitely many alternating quantifiers in the syntax of $L_{\kappa,\kappa}$. Karp \cite{Karp} apparently defines the infinitary formulas by recursion along $\omega$ with (finitely many) infinite blocks of the \emph{same} quantifiers of length $<\kappa$ allowed. Dickmann \cite{Dickmann}, Jech \cite{JECHbook} or Kanamori \cite{KANbook} do not explicitly refer to this point. However, they do claim that the satisfaction of $L_{\kappa,\kappa}$ formulas in first-order structures is defined ``naturally'', which we interpret to mean that infinitely many alternating quantifiers are not allowed.\footnote{\label{middle} The satisfaction of formulas with infinitely many alternating quantifiers is defined in terms of Gale--Stewart game semantics and winning strategies, and hence non-trivially extends the usual Tarski-style definition of satisfaction. Moreover, its properties depend on additional set-theoretic assumptions: for instance if $\psi$ denotes the formula $(\forall x_0)(\exists x_1) \cdots \varphi(\seq{x_n}{n<\omega})$, for a quantifier free $\varphi$, and the satisfaction of $\neg \psi$ is reasonably defined to be equivalent to the satisfaction of $(\exists x_0)(\forall x_1) \cdots \neg \varphi(\seq{x_n}{n<\omega})$, then the Law of the Excluded Middle may fail for the formula $\psi \vee \neg \psi$: It is easy to check that if $\varphi$ is of the form $\seq{x_n}{n<\omega} \in A$, for $A \sub \omega^\omega$, then $\psi \vee \neg \psi$ is equivalent to $A$ being determined as a subset of the Baire space $\omega^\omega$ (under the countable Axiom of Choice for $\omega^\omega$). It is known that $\AC$ implies that not all $A \sub \omega^\omega$ are determined, while $\AD$ postulates that all $A \sub \omega^\omega$ are determined (see \cite{KANbook}).} In any case, for compactness-related topics, the restricted version of infinitary formulas is entirely sufficient because the notion of compactness of $L_{\kappa,\kappa}$ reduces to compactness of its propositional fragment.\footnote{It is easy to check that compactness of the propositional fragment of $L_{\kappa,\kappa}$ (sometimes denoted $L_{\kappa,0}$) implies that every $\kappa$-complete filter can be extended into a $\kappa$-complete ultrafilter, thus it fully characterizes the standard notion of strong compactness of $\kappa$.} To have an unambiguous definition which yields the usual characterization of strong compactness, and moreover avoids the complexities related to  the Law of the Excluded Middle mentioned in Footnote \ref{middle}, we view $L_{\kappa,\kappa}$ as not allowing infinitely many alternating quantifiers.\footnote{Formulas with infinitely many alternating quantifiers are studied in set theory in the specific context of determinacy of games (e.g., the Axiom of Determinacy) even though a connection to $L_{\kappa,\kappa}$ is usually not explicitly mentioned. Additionally, some topos theorists and categorical logicians study such formulas in the context of infinitary proof systems (often under the name \emph{heterogenous} infinitary logics (see for instance \cite{Esp})).}
\erm

The logic $L_{\kappa,\kappa}$ can be strengthened to higher-order infinitary logics $L_{\kappa,\kappa}^n$ for $n<\omega$. Magidor showed in \cite{MAG:logic} that compactness of $L_{\kappa,\kappa}^n$ for $n>2$ reduces to $n = 2$ and that $\kappa$ is $L^2_{\kappa,\kappa}$-compact if and only if $\kappa$ is a certain large cardinal called \emph{extendible}. In particular, the least extendible cardinal $\kappa$ is the least cardinal for which the usual second order logic $L^2_{\omega,\omega}$ is compact. 

Makowski showed in \cite{MAK:logic} shows that the \emph{Vop{\v e}nka cardinal} characterizes in a certain sense the compactness of all finitely generated logics. 

Thus, weakly compact, strongly compact, extendible and Vop{\v e}nka cardinals are all definable through compactness of certain logics. 
Moreover, the notion of compactness from Definition \ref{def:sk} can be generalized by reference to omitting types, yielding the notion of \emph{compactness for omitting types}, which extends the logical characterization to include more large cardinals. 

Boney \cite{MR4093885} analysed this concept and proved $L_{\kappa,\kappa}$-compactness for omitting types characterizations of several other large cardinals like supercompact or $n$-huge cardinals, which are consistency-wise sufficient for all examples we will consider in this article (but even rank-to-rank large cardinals can be characterized in this way in some second-order logics if required). The article \cite{MR4093885} was  extended by a follow-up article \cite{MR4855327} Boney et al.\ which adds compactness-type characterization  for practically all large cardinals (for instance Woodin cardinals or subtle cardinals).\footnote{The articles \cite{MR4093885} and \cite{MR4855327} are carefully written and contain comprehensive up-to-date bibliography related to infinitary logics.} These results make the logical approach to compactness completely general and provide an intuitive justification for large cardinals. 

This being said, the study of large cardinals is usually carried out via combinatorial characterizations expressible in the first-order set theory $\ZFC$ which are more easily applicable to mathematical concepts. This is what we will do as well.

\brm
There is one combinatorial characterization which stands apart as the most universal one, characterizing the majority of large cardinals -- i.e.\ the existence of $\mu$-complete non-principal \emph{ultrafilters}, for $\mu \ge \omega$, over some underlying set $X$.  Postulating the existence of $\mu$-complete ultrafilters over $X$ for appropriate $\mu > \omega$ and $X$ typically yields a straightforward proof of compactness of various logics, for instance the strong compactness of $L_{\kappa,\kappa}$. However,  the existence of countably complete non-principal ultrafilters on $\kappa$ implies that $\kappa$ must be a measurable cardinal or above a measurable cardinal, which makes this principle inconsistent on small cardinals and hence of limited interest for this article. 
\erm

\subsection{Tree-like characterizations of compactness} \label{sec:lists}

Weakly compact, strongly compact and supercompact cardinals $\kappa$ can be defined by compactness properties related to trees or more general tree-like systems. These characterizations have the important benefit of explicitly factoring out the inaccessibility of $\kappa$ and isolating combinatorial principles which can be naturally formulated for small cardinals, such as $\omega_2, \aleph_{\omega+1}$, or $\aleph_{\omega^2+1}$, as well.

In sections \ref{sec:slender},  \ref{sec:gmp} and \ref{sec:con} we will review  basic definitions and results and discuss them also in the context of modern development and connections with the Guessing Model Principle, first considered as a consequence of $\PFA$ and studied by Viale and Weiss (see for instance \cite{VW:PFA}). 

\subsubsection{Thin and slender lists}\label{sec:slender}

Recall the following well-known characterization of weak compactness which first appeared in Erd{\H o}s and Tarski \cite{ET:wc}.

\begin{definition}\label{def:TP}
We say that a regular uncountable cardinal $\kappa$ satisfies \emph{the tree property}, and we write $\TP(\kappa)$, if and only if every $\kappa$-tree (i.e.\ a tree of height $\kappa$ with all levels of $T$ having size $<\kappa$)  has a cofinal branch.
\end{definition}

\begin{Fact}[Erd{\H o}s--Tarski \cite{ET:wc}]
A uncountable cardinal $\kappa$ is weakly compact iff $\kappa$ is inaccessible and $\TP(\kappa)$ holds.
\end{Fact}

This characterization of weak compactness was extended by Jech to strong compactness in \cite{jech:tree} and Magidor to supercompactness \cite{MAG:super}. Since both these characterizations are formulated using the same combinatorial context, we will discuss them together. Recall that $\kappa$ is strongly compact if for every $\lambda \ge \kappa$ there is a fine ($\kappa$-complete) ultrafilter on $\power_\kappa(\lambda)$ and it is supercompact if there a normal ($\kappa$-complete) ultrafilter on $\power_\kappa(\lambda)$.\footnote{See Kanamori \cite{KANbook} for more details. For the notation, we use $\power_\kappa(\lambda)$ to denote the set of all subsets of $\lambda$ of size $<\kappa$.} For an equivalent characterization in terms of trees, it is necessary to find an appropriate two-dimensional generalization of a $\kappa$-tree. Jech defined a two-dimensional system and called it a $(\kappa,\lambda)$-mess. We will use a more recent terminology of  Weiss \cite{Weiss:gentree} and refer to these objects as $(\kappa,\lambda)$-lists. We will distinguish two types of lists, which are equivalent for an inaccessible $\kappa$, but are different for successor cardinals.

\begin{definition}\label{def:list}
Suppose $\kappa \le \lambda$ are cardinals, with $\kappa$ regular uncountable. We call sequence $\seq{d_x}{x \in \power_\kappa(\lambda)}$ a \emph{$(\kappa,\lambda)$-list} if $d_x \sub x$ for every $x \in \power_\kappa(\lambda)$. We say that a $(\kappa,\lambda)$-list $\seq{d_x}{x \in \power_\kappa(\lambda)}$ is

\begin{itemize}

\item \emph{thin} if there is a closed unbounded set $C \sub \power_\kappa(\lambda)$ such that $|\set{d_x \cap y}{y \sub x}|<\kappa$ for every $y \in C$.

\item \emph{$\mu$-slender} for some uncountable $\mu \le \kappa$ if for all sufficiently large $\theta$ there is a club $C \sub \power_\kappa(H(\theta))$\footnote{$H(\theta)$ denotes the set of all sets whose transitive closure has size $<\theta$.} such that for all $M \in C$ and all $y \in M \cap \power_\mu(\lambda)$, $d_{M \cap \lambda} \cap y \in M$.
\end{itemize}

\end{definition}

Note that every $\kappa$-slender list is $\mu$-slender for every $\mu\le \kappa$ and that the family of all $\omega_1$-slender lists is the most extensive (and usually considered as the default option for slender lists unless said otherwise). It is straightforward to show that that every thin list is $\kappa$-slender (see for instance Weiss \cite[Proposition 2.2]{Weiss:gentree}).

As defined, $(\kappa,\lambda)$-lists are not trees in the usual sense, but can be reformulated to be quite similar to trees (see for instance Lambie-Hanson and Stejskalov{\'a} \cite[Definition 1]{ChS:guess} for $\Lambda$-trees). Such reformulations have the benefit of retaining some of the intuition related to trees in this more general setting. To upheld this similarity, certain coherent families of elements of lists are called branches:

\begin{definition} Let $D = \seq{d_x}{x \in \power_\kappa(\lambda)}$ be a \emph{$(\kappa,\lambda)$-list} and $d \sub \lambda$.
\begin{itemize}
\item We say that $d$ is a \emph{cofinal branch} of $D$ if for all $x \in \power_\kappa(\lambda)$ there is $z_x \super x$ such that $d \cap x = d_{z_x} \cap x$.

\item We say that $d$ is an \emph{ineffable branch} if the set $\set{x \in \power_\kappa(\lambda)}{d \cap x = d_x}$ is stationary.
\end{itemize}
\end{definition}

See the clear summary of notions related to closed unbounded and stationary subsets of $\power_\kappa(\lambda)$ in \cite[Section 3]{ChS:guess}. 

The existence of cofinal or ineffable branches in thin and slender lists leads to multiple compactness principles, as first defined in \cite{Weiss:gentree}:

\begin{definition}\label{def:tree-principles} Let $\mu \le \kappa\le \lambda$ be cardinals with $\kappa$ regular uncountable:
\begin{itemize}

\item We say that the \emph{$(\kappa,\lambda)$-tree property} holds and write $\TP(\kappa,\lambda)$ if every \emph{thin} $(\kappa,\lambda)$-list has a \emph{cofinal} branch.

\item We say that the \emph{ineffable $(\kappa,\lambda)$-tree property} holds and write $\ITP(\kappa,\lambda)$ if every \emph{thin} $(\kappa,\lambda)$-list has an \emph{ineffable} branch.

\item We say that the \emph{$(\mu,\kappa,\lambda)$-slender tree property} holds and write  $\SP(\mu, \kappa,\lambda)$ if every \emph{$\mu$-slender} $(\kappa,\lambda)$-list has a \emph{cofinal} branch. We write $\SP(\kappa,\lambda)$ for the strongest principle $\SP(\omega_1,\kappa,\lambda)$.

\item We say that the \emph{ineffable $(\mu,\kappa,\lambda)$-slender tree property} holds and write $\ISP(\mu, \kappa,\lambda)$ if every \emph{$\mu$-slender} $(\kappa,\lambda)$-list has an \emph{ineffable} branch. We write $\ISP(\kappa,\lambda)$ for the strongest principle $\ISP(\omega_1,\kappa,\lambda)$.
\end{itemize}
\end{definition}

In this notation, $\TP(\kappa,\kappa)$ is equivalent to the usual tree property at $\kappa$ which we already denote by $\TP(\kappa)$.

\begin{definition}\label{def:gentree}
To simplify the notation further, we write $\ISPgen{\kappa}$,  $\SPgen{\kappa}$ and $\ITPgen{\kappa}$, $\TPgen{\kappa}$ if $\ISP(\omega_1,\kappa,\lambda)$, $\SP(\omega_1,\kappa,\lambda)$ and $\ITP(\kappa,\lambda)$, $\TP(\kappa,\lambda)$, respectively,  hold for every $\lambda \ge \kappa$. 
\end{definition}

If $\kappa$ is inaccessible, slender lists are by definition also thin.  It follows that for an inaccessible $\kappa$, $\TPgen{\kappa}$ is equivalent to $\SPgen{\kappa}$ and $\ITPgen{\kappa}$ is equivalent to $\ISPgen{\kappa}$, and characterize strong compactness and supercompactness, respectively:

\begin{Fact}[Jech \cite{jech:tree}, Magidor \cite{MAG:super}]
Suppose $\kappa$ is a regular uncountable cardinal. Then:
\bce[(i)]
\item $\kappa$ is strongly compact iff $\kappa$ is inaccessible and $\TPgen{\kappa}$ (equivalently $\SPgen{\kappa}$) holds.
\item $\kappa$ is supercompact iff $\kappa$ is inaccessible and $\ITPgen{\kappa}$ (equivalently $\ISPgen{\kappa}$) holds.
\ece
\end{Fact}

Compactness principles on successor cardinals may reveal some distinctions which are not apparent on inaccessibles: at successor cardinals, the formulations with slender and thin lists are no longer equivalent, with the principles referring to slender lists being substantially stronger. Let us give an example illustrating the strength of slender lists at successor cardinals and also some details regarding the role of $\mu$ in the definition of slender lists. This and similar examples indicate that the ``$\omega_1$-slender'' list is the right concept for compactness related to $(\kappa,\lambda)$-lists, with the distinction between ineffable and cofinal branches of lesser importance.\footnote{We will briefly review the consistency and consequences of the tree properties on successors in Section \ref{sec:con}. Let us take the consistency at, e.g., $\omega_2$ as given for the moment.}  
We will use for this example a compactness principle interesting in its own right.

\begin{definition}
If $\kappa$ is a regular cardinal, we say that $(T,<)$ is a \emph{weak $\kappa$-Kurepa tree} if $T$ is a tree of height $\kappa$ and size $\le \kappa$ which has at least $\kappa^+$-many cofinal branches. We say that the \emph{weak Kurepa hypothesis} holds at $\kappa$, denoted $\wKH(\kappa)$,  if and only if there is a weak $\kappa$-Kurepa tree.
\end{definition}

Cox and Krueger showed in \cite{CK:ind} that $\nonwKH(\omega_1)$ follows from $\ISPii$. Lambie-Hanson and Stejskalov{\'a} generalized this and some other results in \cite{arithmetic_paper} by proving them from a weaker principle (and simultaneously proving stronger results) related to cofinal branches in slender lists (see Section \ref{sec:con} for more details). We will illustrate their results by showing that $\SP(\omega_1,\omega_2,\omega_2)$ implies $\nonwKH(\omega_1)$:

\begin{lemma}\label{lm:sp}
$\SP(\omega_1,\omega_2,\omega_2)$ implies $\neg \wKH(\omega_1)$.
\end{lemma}

\begin{proof}
Suppose for contradiction that $(T,<_T)$ is an $\omega_1$-Kurepa tree which we identify with a subset of $2^{<\omega_1}$. Let $\seq{b_\alpha}{\alpha<\omega_2}$ be an injective enumeration of $\omega_2$-many cofinal branches (we identify cofinal branches in $T$ with subsets of $\omega_1$). Let us define a slender list $D$ as follows: for every $x \in \power_{\omega_2}(\omega_2)$ such that $\omega_1 \sub x$, let $\gamma_x$ be the least ordinal below $\omega_2$ not in $x$ and set $d_x = b_{\gamma_x}$. For all $x$ with $\omega_1 \not \sub x$, set  $d_x = \emptyset$.

To show that $D$ is slender, we need to find a club of $M$ such that whenever $z \in \power_{\omega_1}(\omega_2) \cap M$, $d_{M \cap \omega_2} \cap z \in M$. This holds for all $M$ such that $\omega_1, T \sub M$ (and such $M$ clearly form a club): If $d_{M \cap \omega_2}  =\emptyset$, then we are done. If $d_{M \cap \omega_2} = b_\alpha$, for some $\alpha < \omega_2$, then $b_\alpha \cap z$ is included in some $t \in T$, $t \sub b_\alpha$, and can be defined in $M$ because both $z, t$ are in $M$.

Let $d$ be a cofinal branch. Fix any $x$ with $\omega_1 \sub x$.  There is $z_x \super x$ such that $$d \cap x = d_{z_x} \cap x.$$ Let $\gamma$ be the least ordinal not in $z_x$ so that $d_{z_x} = b_\gamma$. Now choose any $y \super \omega_1 \cup \{\gamma\}$. Then there is $z_y \super y$ such that 
$$d \cap y = d_{z_y} \cap y.$$ Since $\gamma \in z_y$, $d_{z_y}$ is a cofinal branch $b_\delta$ (where $\delta$ is the least ordinal not in $z_y$) distinct from $b_\gamma$. This a contradiction because $d \cap \omega_1$ is fixed and cannot be equal to both $b_\gamma$ and $b_\delta$.
\end{proof}

The $\omega_1$-slenderness is essential: $\ISP(\omega_2,\omega_2,\lambda)$ for all $\lambda \ge \omega_2$ is consistent with the existence of (non-wide) $\omega_1$-Kurepa trees (see \cite[Theorem 53]{ChS:guess}). Since $\ISP(\omega_2,\omega_2,\lambda)$ implies $\ITP(\omega_2,\lambda)$, which in turn implies $\TP_{\omega_2}$, none of the other principles implies $\neg \wKH(\omega_1)$ either.

\subsubsection{The Guessing Model Principle}\label{sec:gmp}

Suppose $\kappa \ge \omega_2$ is a regular cardinal. Viale and Weiss isolated in \cite{VW:PFA, Viale:Laver} a model-theoretic principle equivalent to $\ISP_\kappa$ and called it the \emph{Guessing Model Principle}, $\GMP_\kappa$. $\GMP_\kappa$ makes it possible to derive consequences of $\ISP_\kappa$ in a model-theoretic way, making the arguments similar to those using elementary embeddings.

In the interest of completeness we will review this principle because it has been increasingly used to derive consequences of $\ISP_\kappa$. See for instance Viale \cite{MR2959666} for the failure of squares, Krueger \cite{MR4021101} for $\SCH$, articles by Lambie-Hanson and Stejskalov{\'a} \cite{ChS:guess, arithmetic_paper} for applications related to Kurepa trees and combinatorics,  Honzik et al.\ \cite{HLS:gmp} for a proof using guessing models that $\ISPii$ is preserved by Cohen reals over all models of $\ISPii$ (in particular over models of $\PFA$), and Mohammadpour and Veli{\v c}kovi{\'c} \cite{MR4290492} who obtain $\ISP_{\omega_2}$ and $\ISP_{\omega_3}$ simultaneously as a consequence a variant of the guessing model principle (see the end of Section \ref{sec:col} for more information about global patterns of compactness principles).

\begin{definition}[$\GMP$]\label{guessing_model_def}
Let $\mu < \theta$ be uncountable cardinals, $\theta$ regular, and let $M \el H(\theta)$.
  \begin{compactenum}[(i)]
    \item Given a set $x \in M$, and a subset $d \subseteq x$, we say that 
    \begin{enumerate}
      \item $d$ is \emph{$(\mu, M)$-approximated} if, for every $z \in M \cap \power_{\mu}(M)$, 
      we have $d \cap z \in M$;
      \item $d$ is \emph{$M$-guessed} if there is $e \in M$ such that $d \cap M = e \cap M$.
    \end{enumerate}
    \item For $x \in M$, $M$ is a \emph{$\mu$-guessing model for $x$} if
    every $(\mu, M)$-approximated subset of $x$ is $M$-guessed.
    \item $M$ is a \emph{$\mu$-guessing model} if, for every $x \in M$, it is a $\mu$-guessing 
    model for $x$.
  \end{compactenum}
  Let $\mu \le \kappa\le\theta$ be uncountable cardinals with $\kappa$ and $\lambda$ regular. We denote by $\GMP(\mu, \kappa,\theta)$ the assertion that the set of $M \in \power_\kappa (H(\theta))$ such that $M$ is a 
  $\mu$-guessing model is stationary in $\power_\kappa (H(\theta))$. 

In keeping with our notation for $\ISPgen{\kappa}$, we will write $\GMPgen{\kappa}$ if $\GMP(\omega_1,\kappa,\theta)$ holds for all regular $\theta\ge \kappa$.
\end{definition}

Viale and Weiss proved in \cite{VW:PFA} that $\GMPii$ and $\ISPii$ are equivalent and are consequences of $\PFA$. Lambie-Hanson and Stejskalov{\'a} explicitly proved the generalization that \begin{equation} \forall \lambda \ge \kappa \; \ISP(\mu,\kappa,\lambda) \Iff \forall \theta \ge \kappa \mbox{ regular} \; \GMP(\mu,\kappa,\theta) \end{equation} for all regular uncountable cardinals $\mu\le \kappa$ (see \cite[Corollary 14]{ChS:guess}).\footnote{It is possible to show a local equivalence of these principles but there is a certain asymmetry due to the fact that one of the principles refers to cardinals $\lambda$ and the other one to $H(\theta)$  (see Corollary 14 and the paragraph below it in \cite{ChS:guess}).}

In order to illustrate the use of guessing models with $\mu >\omega_1$ we prove the tree property $\TP(\omega_2)$, i.e.\ that there are no $\omega_2$-Aronszajn trees.

\begin{lemma}\label{lm:tp}
$\GMP(\omega_2,\omega_2,\omega_3)$ implies $\TP(\omega_2)$.
\end{lemma}

\begin{proof}
Suppose $M\el H(\omega_3)$ is an $\omega_2$-guessing model of size $\omega_1$ and $T \in M$. We can assume $\delta = M \cap \omega_2$ is a limit ordinal greater than $\omega_1$ (because the set of all such $M$ is a club). To argue for $\TP(\omega_2)$, it suffices to notice that if $z \in M \cap \power_{\omega_2}(\omega_2)$, then $z$ must be bounded below $\delta$: Otherwise there would be in $M$ a bijection from $\omega_1$ onto $z$ cofinal in $\delta$ (by elementarity, if $z \in M$ and $\omega_1 \in M$, then there must be some $f: \omega_1 \to z$ in $M$ as well). This would mean that $M$ thinks that $\omega_2$ has cofinality $\omega_1$, which is impossible. Let $t$ be any node in $T$ on level $\delta$, and $d$ the set of its $T$-predecessors. By what we said above, $d$ is $(\omega_2,M)$-approximated and there must be some $e \in M$ such that $e \cap M = d \cap M$. By elementarity $e$ is a cofinal branch in $T$.
\end{proof}

\subsubsection{Consistency and consequences}\label{sec:con}

In Lemmas \ref{lm:sp} and \ref{lm:tp}, we illustrated the use of compactness related to trees and lists at $\omega_2$. Let us say a few words here regarding the forcings which show the consistency of  such principles at $\omega_2$, or more generally at double successors of regular cardinals. See Section \ref{sec:canonical} for more details on  models of compactness principles.

Compactness at $\omega_2$ is historically the most researched case due to the connections with $\PFA$. It is also sufficiently representative for the class of double successors of regular cardinals which tend to behave the same way with regard to compactness principles.\footnote{\label{R1} There are known exceptions to this heuristics: (i) an important exception is the countable-support iteration of proper forcings of a supercompact length which forces $\PFA$. Since a generalization of properness to higher cardinals is missing, the case of $\omega_2$ is special in the context of forcing axioms, and (ii) the consistency strength of the Suslin Hypothesis at $\kappa^{++}$ for $2^{<\kappa} = \kappa$ may be different for $\kappa = \omega$ and $\kappa > \omega$ according to the known results (see Footnote \ref{ft:SH} for more details).}  The first result in this direction was obtained by Mitchell \cite{MIT:tree} who showed that $\TP(\omega_2)$, and in general $\TP(\kappa^{++})$ for a regular $\kappa$, is consistent from a weakly compact cardinal. Mitchell defined a forcing notion in \cite{MIT:tree} which turns a large cardinal $\lambda$ into a double successor $\kappa^{++}$ of a regular cardinal $\kappa<\lambda$ and which has since become a standard tool for obtaining compactness at double successors of regular cardinals. The terms ``Mitchell forcing'' (or ``Mitchell collapse'') and ``Mitchell model'' are used in the literature to denote various variants and generalizations of the original Mitchell forcing and the associated generic models. See Abraham \cite{ABR:tree} for a detailed exposition of the classical Mitchell forcing and Krueger \cite{KR:DSS} for a description based on the concept of a mixed-support iteration. Later on,  other forcings were used to produce models with compactness at small cardinals, for instance the iteration of the usual Sacks forcing at $\omega$ of weakly compact length produces a model with $\TP(\omega_2)$, see Kanamori \cite{KANAMORIperfect}; this result can be  generalized to other forcings with fusion, see Stejskalov{\'a} \cite{S:g} and Honzik and Verner \cite{HV:grig} for Grigorieff forcing or Friedman et al.\ \cite{FHZ:fusion} for a more general set-up.

In addition to the tree property, the same forcings yield compactness principles like stationary reflection, failure of the approachability property (see Unger \cite{U:ap}), the negation of the weak Kurepa Hypothesis (see Honzik and Stejskalov{\'a} \cite{HS:ureg} for a detailed proof) and the strong tree properties (Weiss \cite{Weiss:gentree}, Viale and Weiss \cite{VW:PFA}, and Fontanella \cite{Font:tree}). We note that the principle $\SPgen{\omega_2}$ is not fully understood yet:  while \cite{Weiss:gentree} claims that strongly compact cardinals suffice to have $\SP_{\omega_2}$ in the Mitchell model, no details are given for the supposed proof, and in fact it is open whether even the weaker principle $\SP(\omega_2,\omega_2,\lambda)$ holds in the Mitchell model (see Lambie-Hanson and Stejskalov{\'a} \cite{ChS:guess} who formulate a weakening of $\SP_{\omega_2}$ and prove that it holds in the Mitchell model starting with a strongly compact cardinal).

More complicated forcings are required for the tree property to hold at the successor or double successor of a singular cardinal. See Section \ref{sec:canonical} for some more details for these cases.

Let us summarize the known consequences of the tree properties and state some related open questions. We will divide the consequences into three types:

\medskip

\textbf{Other compactness principles and $\AD^{L(\R)}$.}

\begin{itemize}
\item Weiss proved in \cite{Weiss:gentree} that $\ITP_\kappa$ implies $\neg \square(E^\lambda_{<\kappa},\kappa)$, where $E^\lambda_{<\kappa}$ is the set of ordinals below $\lambda$ of cofinality $<\kappa$, for all $\lambda$ with $\cf(\lambda) \ge \kappa$.\footnote{\label{R2} Weiss \cite{Weiss:gentree} uses a different notation. We use a more common notation $\square_{\lambda}(E,\kappa)$, where the first parameter in the brackets is the domain of the principle, the second one is the width and the subscript is an (optional) order-type restriction (the sequences are required to have length $\le \lambda$). The original Jensen's notation $\square_\lambda$ is thus equal to $\square_\lambda(\lambda^+,1)$; more generally, Schimmerling's weak square $\square_{\lambda,\kappa}$ is $\square_\lambda(\lambda^+,\kappa)$, the weak square $\square^*_\lambda$ is $\square_{\lambda}(\lambda^+,\lambda)$, and Todor{\v c}evi{\'c}'s $\square(\lambda)$ is $\square(\lambda,1)$. See Section \ref{sec:sq} for the definition of $\square(\lambda)$.} In particular, if $\kappa \le \lambda$ and $\lambda$ is a singular strong limit cardinal, then $\ITP_\kappa$ implies $\neg \square(E^{\lambda^+}_{<\kappa},\kappa)$, and hence $\neg\square_{\lambda}$. By Steel's theorem \cite[Theorem 0.1]{MR2194247} that the failure of a square at a strong limit singular cardinal implies $\AD^{L(\R)}$,  $\ITP_\kappa$ for any $\kappa \ge \omega_2$ entails $\AD^{L(\R)}$. It follows that Projective Determinacy and  regularity properties of the reals like Lebesgue measurability of definable subsets (all consequences of $\AD^{(L(\R))}$) follow not only from $\PFA$, but in fact from a ``compactness'' part of $\PFA$, specifically  $\ITP_{\omega_2}$. In fact, the failure of (weak) squares is often considered a compactness principle in itself because it follows from sufficiently large cardinals, and also from $\ITP_\kappa$ as we saw, or  from some forms of simultaneous stationary reflection.\footnote{\label{R3} Failures of squares are not the main focus of this article, but we will squeeze in some comments here. It is know that the weak square $\square^*_\lambda$ is equivalent to the existence of a special $\lambda^+$-Aronszajn tree for any infinite $\lambda$; in fact this characterization can be extended to all regular (not only successor) cardinals, see Krueger \cite{JohnKrueger2013}. The consistency strength of the failure of  square $\square_\lambda$ for regular $\lambda$ is equivalent to a Mahlo cardinal. The failure of $\square_\lambda$ for a strong limit singular cardinal has a much higher consistency strength since it implies $\AD^{L(\R)}$ as we mentioned above.  See Hayut \cite[Corollary 7]{MR3959249} for equivalences between chromatic compactness of graphs, failures of squares and simultaneous stationary reflection (this article is related to large cardinals), Lambie-Hanson et al.\ \cite{paper64} for connections between squares, forcing axioms and indecomposable ultrafilters and Sakai \cite{MR3017275} for connections between Chang's Conjecture and squares. It is worth mentioning that simultaneous stationary reflection is compatible with a weak instance of square $\square_\lambda(\lambda^+,\cf(\lambda))$, see Cummings et al. \cite[Theorem 10.1]{cfm}. This contrasts with indestructible forms of stationary reflection which kill the weakest square  $\square_\lambda(\lambda^+,\lambda)$, see Fuchs and Rinot \cite{paper33} (see also Section \ref{sec:sr}) for more details about simultaneous stationary reflection.}
\item If $\kappa$ is regular, then $\ISP_{\kappa^+}$ implies the failure of the approachability property, $\neg \AP(\kappa^+)$. Note that $\ITP_{\kappa^+}$ is not sufficient for proving $\neg \AP(\kappa^+)$ (folklore and Cummings et al.\ \cite{8fold}).
\item Cox and Krueger showed in \cite{CK:ind} that $\ISP_{\kappa^+}$ implies $\nonwKH(\kappa)$. Their result for $\nonwKH(\kappa)$ was improved by  Lambie-Hanson and Stejskalov{\'a} in \cite{arithmetic_paper} who  developed a model-theoretic characterization of the principle $\SP$ called $\wAGPY$ (with several parameters) and showed that with appropriate parameters implies $\nonwKH(\kappa)$. See Lemma \ref{lm:sp} which illustrates this result by showing that $\SP(\omega_1,\omega_2,\omega_2)$ implies $\nonwKH(\omega_1)$. Note that $\ITP_{\kappa^+}$ does not suffice for $\nonwKH(\kappa)$, the notion of slenderness is essential here: see the paragraph below Lemma \ref{lm:sp} for more details.

\item Tree properties are in general independent from stationary reflection. See Cummings et al.\ \cite{8fold} who show that $\SR(\omega_2)$ is independent from $\TP(\omega_2)$ and $\neg \AP(\omega_2)$ (with all eight possibilities consistent). It is also known that $\PFA$ does not imply $\SR(\omega_2)$ by Beaudoin \cite{MR1099782}, so \emph{a fortiori}, $\ISPii$ does not imply $\SR(\omega_2)$.  In fact, at some successors of singulars the tree properties and stationary reflection may be incompatible: Magidor conjectured that the tree property $\TP(\aleph_{\omega+1})$ together with $\aleph_\omega$ strong limit implies $\SCH$ at $\aleph_\omega$ and $\neg \SR(\aleph_{\omega+1})$.\footnote{\label{R4} The case of $\aleph_{\omega+1}$ with $\SCH$ at $\aleph_\omega$ is very specific. For instance $\SR(\aleph_{\omega+1})$ implies the incompactness principle $\AP(\aleph_{\omega+1})$ by Shelah \cite{Sh:108} (for a proof see \cite[Corollary 3.41]{MR2768694} where this principle is denoted $\AP_{\aleph_\omega}$), but does not imply $\SCH$ at $\aleph_\omega$, see Poveda et al.\ \cite{paper43} and Ben-Neria et al.\ \cite{MR4725661}. Note that the combination of the failure of $\SCH$ at $\aleph_\omega$ with $\SR(\aleph_{\omega+1})$ is optimal because stationary reflection for subsets of $[\aleph_{\omega+1}]^\omega$ already implies $\SCH$ at $\aleph_\omega$ by Shelah \cite{MR2369124}.} This lack of connection, and perhaps even incompatibility, between stationary reflection and the tree properties explains the lack of consequences of logical principles for compactness of mathematical structures (see Remark \ref{rm:lack} and Section \ref{sec:uni} for more discussion).
\end{itemize}

\textbf{Cardinal arithmetics} 

\begin{itemize}
\item For all $\kappa \ge \omega_2$, $\ISP_\kappa$ implies $\SCH$ above $\kappa$, by Krueger \cite{MR4021101}. Lambie-Hanson and Stejskalov{\'a} improved this result by showing that already the principle $\wAGPY$ at $\kappa$  with appropriate parameters implies Shelah's Strong Hypothesis above $\kappa$, which is known to imply $\SCH$.\footnote{\label{R5} It is worth observing that Shelah’s strong hypothesis is by itself a reflection principle. It is equivalent to the assertion that for every regular cardinal $\kappa$, for every first countable space $X$ of density $\kappa$, if $|X|>\kappa$, then some separable subspace $Y$ of $X$ satisfies $|Y|>\kappa$, see Rinot \cite{paper05}. } Hence having $\ISPii$ (in fact $\wAGPY$ at $\omega_2$) limits the extent of compactness in the universe, forbidding for instance the global principle that the tree property holds on every regular cardinal (to have this there would need to be strong limit cardinals $\kappa$ with $\TP(\kappa^{++})$ which implies the failure of $\SCH$ at $\kappa$).

\item Cummings at al.\ show in \cite{MR4127897}  that $\ITP_{\kappa^+}$ is consistent with the failure of $\SCH$ at a strong limit $\kappa$ ($\kappa$ can be equal to $\aleph_{\omega^2}$). So the principle at $\kappa$ does not enforce $\SCH$ below $\kappa$. However, it is open whether $\ITP_{\kappa}$ or even $\TP_\kappa$ implies $\SCH$ above $\kappa$. It is stated as plausible in \cite{MR4127897} that $\ITP_{\kappa^+}$ and $\ITP_{\kappa^{++}}$ with $\kappa = \aleph_{\omega^2}$ can hold simultaneously, extending the known results for $\TP(\kappa^+)$ and $\TP(\kappa^{++})$ in Sinapova and Unger \cite{SU:TP}.

\item Suppose $\ISP_{\kappa^{++}}$ holds, with $\kappa$ being a regular cardinal. Then $\ISP_{\kappa^{++}}$ implies  $2^\kappa > \kappa^+$, but places no additional requirements on the value of $2^\kappa$. In particular the consequence $2^\omega = \omega_2$ of $\PFA$ is not retained by $\ISPii$. This was first proved by Cox and Krueger \cite{MR3471139}, and it is also a corollary of an indestructibility theorem by Honzik et al.\ \cite{HLS:gmp}. 

\item However, $\ISPii$ does put some restrictions on the value of $2^{\omega_1}$; in fact already the principle $\nonwKH(\omega_1)$ does. More generally, Lambie-Hanson and Stejskalov{\'a} show in \cite{arithmetic_paper} that if $\kappa$ is regular uncountable, $\nonwKH(\kappa)$ and $2^{<\kappa} <\kappa^{+\kappa}$, then $2^\kappa = 2^{<\kappa}$. In particular if $\nonwKH(\omega_1)$ and $2^\omega <\aleph_{\omega_1}$, then $2^{\omega_1} = 2^{\omega}$. However, relative to the existence of a supercompact cardinal, $\nonwKH(\omega_1)$ is consistent with $2^\omega = \aleph_{\omega_1}$ and $2^{\omega_1} > \aleph_{\omega_1+1}$. \end{itemize}

\textbf{Cardinal invariants of the continuum}

\begin{itemize}
\item There is an ``indestructible'' strengthening of $\ISPii$, called the \emph{indestructible guessing model property}, $\IGMP$, introduced by Cox and Krueger in \cite{MR3691206}. They show that it follows from $\PFA$ and is non-trivially stronger than $\ISPii$, for instance it implies the Suslin Hypothesis. They show that $\IGMP$ does not put any bound on the value of $2^\omega$, but they left open the question whether $\cf(2^\omega) = \omega_1$ is consistent with $\IGMP$. This open question was partially answered by Lambie-Hanson and Stejskalov{\'a} in \cite{arithmetic_paper} who showed that in a generic extension by a measure algebra, an indestructible version of $\nonwKH(\omega_1)$ holds and $\cf(2^\omega) = \omega_1$. 

\item Provable consequences of the tree properties at $\omega_2$ for cardinal invariants are not known (except that they imply $2^\omega > \omega_1$, thus making the structure of cardinal invariants non-trivial, at least in principle). For the tree property $\TP(\omega_2)$ and $\nonwKH(\omega_1)$, Honzik and Stejskalov{\'a} showed in \cite{HS:ureg} that various patterns of cardinal invariants are consistent with the principles $\TP(\omega_2), \nonwKH(\omega_1)$ and $\SR(\omega_2)$. A configuration which is left open in \cite{HS:ureg} is whether these principles are consistent with $\omega_1 < \mathfrak{t} = \mathfrak{u} < 2^\omega$. Note in this connection that Cox and Krueger asked in \cite{MR3691206}  whether $\IGMP$ implies  $\mathfrak{t} > \omega_1$. This was answered negatively by Lambie-Hanson and Stejskalov{\'a} in \cite{arithmetic_paper}.

\item Mohammadpour and Veli{\v c}kovi{\'c} introduced in \cite{MR4290492} a two-cardinal strengthening of $\ISPii$ called $\GM^+(\omega_3,\omega_1)$ which is consistent modulo two supercompact cardinals and implies $\ISPii$ and $\ISP_{\omega_3}$, the failure of the weak square principle $\square(\kappa, \omega_2)$ for all $\kappa \ge \omega_2$, and the fact that the restriction of the approachability ideal $I[\omega_2]$ to $\omega_2 \cap \mx{cof}(\omega_1)$ is the non-stationary ideal (which is the strongest possible failure of the approachability  property at $\omega_2$).\footnote{\label{R13}This was first shown consistent by Mitchell in \cite{MR2452816} starting with a greatly Mahlo cardinal. Mitchell's result solved a long-standing open problem of Shelah and introduced several important forcing methods (finite conditions with side conditions for adding clubs in $\omega_2$, strongly generic conditions and strongly proper forcings). See Gilton and Krueger \cite{MR3620064} for more comments and generalizations of Mitchell's proof.} In \cite{MOHAMMADPOUR_VELIČKOVIĆ_2025} they consider an indestructible version of $\GM^+(\omega_3,\omega_1)$, which they call $\SGM^+(\omega_3,\omega_1)$, in an analogy with $\IGMP$ discussed above.

\end{itemize}

\brm \label{rm:lack}
The tree properties at successors of regulars do not in general imply compactness principles for graphs or algebras which we discuss here. For example,  $\ZFC$ proves that abelian compactness (Section \ref{sec:ab}), chromatic compactness of graphs (Section \ref{sec:graphs}) or full stationary reflection all fail at $\omega_2$ (Sections \ref{sec:sr} and \ref{sec:delta}) and in fact, everywhere below $\aleph_\omega$\footnote{With high probability everywhere below $\aleph_{\omega^2}$, but not higher. The $\Delta$-reflection which we discuss in Section \ref{sec:math} can consistently hold at $\aleph_{\omega^2+1}$ and it implies several compactness principles for mathematical structures.}, even though strong tree properties can consistently hold below $\aleph_\omega$. The main reason why tree properties at successors of regulars have limited effect on the compactness of mathematical structures is the existence of non-reflecting stationary sets: they suffice by an inductive argument to construct various incompact mathematical objects, but---importantly---do not suffice to construct thin lists without cofinal branches. At larger cardinals, where compactness of graphs or groups may consistently hold, independence of  mathematical compactness from the tree properties is shown by direct arguments (see for instance Fontanella and Magidor \cite{MF:TPSR} and Fontanella and Hayut \cite{FH:delta}). See also Remark \ref{rm:dtp}.
\erm

\section{Compactness in mathematics}\label{sec:math}

There are natural compactness principles for mathematical structures such as algebras, graphs or topological spaces which assert that certain properties holding in all substructures of size $<\kappa$ necessarily hold in the whole structure of size $\kappa$. Many principles---and all which we will discuss in some detail---can be expressed in the infinitary logic $L_{\kappa,\kappa}$ and hence are true if $\kappa$ is weakly compact. An interesting question is whether the converse holds as well, i.e., whether validity of these  principles for structures of size $\kappa$ already implies that $\kappa$ is weakly compact. While this is the case in $V = L$ (see Theorem \ref{th:full} for all examples we will discuss),\footnote{In this article, the initial formulations of compactness principles at $\kappa$ are usually stated for structures of size $\kappa$. This makes it possible to compare them in $V = L$, where they all turn out to be equivalent to weak compactness of $\kappa$. It is natural to allow unrestricted forms of these principles in which the size of the structures is unlimited, and which correspond to strong compactness of $\kappa$. However, since there are no  models for strongly compact cardinals like $L$ (core models), it is unclear whether there is a single universe where all these unrestricted principles are equivalent.} it is good news---from the perspective of this article---that it is consistent modulo large cardinals that some  principles can hold on small cardinals as well.  

A necessary  condition for compactness of many  mathematical structures of size $\kappa$ is stationary reflection for certain subsets of $\kappa$. For instance, Eklof, Gregory and Shelah (independently) showed that if $\kappa < \lambda $ are regular cardinals and there is a non-free almost-free abelian group of size $\kappa$ and there is a non-reflecting stationary subset $S \sub \lambda \cap \cof(\kappa)$, then there is a non-free almost free abelian group of size $\lambda$ (a ``pump-up lemma'') (see \cite[Theorem 2.3]{modules}). Similarly, Shelah showed \cite[Claim 1.2]{Shelah:inc} that if there is a non-reflecting subset of $\kappa^{++} \cap \cof(\kappa)$, then there is a graph $G$ of size $(\kappa^{++})^\kappa$ with chromatic number $>\kappa$ with all small subgraphs having chromatic number $\le \kappa$. These conditions are often non-trivial, and the research into compactness often leads to new $\ZFC$-theorems for mathematical structures (such as Lemma \ref{lm:MS} for the provable existence of almost-free non-free abelian groups which extends the ``pump-up lemma'').

In order to understand compactness at small cardinals better, it is worth studying whether stationary reflection is also sufficient for compactness, i.e.,whether incompact structures (of the given type) necessarily arise out of non-reflecting stationary sets. As we will review in Section \ref{sec:L}, the answer is \emph{yes} in G{\"o}del's constructible universe  $L$: in fact, if $V = L$, stationary reflection just for the subsets of $E^\kappa_\omega$ (the set of ordinals $<\kappa$ with countable cofinality) is sufficient (since it is equivalent to weak compactness of $\kappa$). If $V \neq L$ (and modulo large cardinals), the usual reflection principle for stationary subsets of ordinals is relatively weak. However, under large cardinal assumptions, there are powerful strengthenings of stationary reflection with mathematical consequences at small successor cardinals: We will review in some detail the $\Delta$-reflection introduced by Magidor and Shelah \cite{MS:free} for compactness at successors of singulars, the (Weak) Reflection Principle, $\mathsf{(W)RP}$, introduced by Foreman et al.\ \cite{FMS:MM}, and the Fodor-type Reflection Principle, $\FRP$, introduced by Fuchino et al.\ \cite{MR2610450}, which can hold already at $\omega_2$.

\brm \label{rm:dtp}
Stationary reflection and logical principles related to trees, which we reviewed in the previous section, are logically independent because all known forms of stationary reflection at $\omega_2$ are consistent with $\CH$, while $\TP(\omega_2)$ implies $\neg \CH$. However, assuming a weak fragment of $\MA$,  some forms of stationary reflection do imply $\TP(\omega_2)$ (see Section \ref{sec:gsr} for more details), and assuming $2^\omega = \omega_2$, Strong Chang's Conjecture, which is a consequence of Rado's Conjecture, implies $\TP_{\omega_2}$, which hints at a possibility of unifying the logical and mathematical principles (see Section \ref{sec:uni} for details and references). At successors of singulars, the principles are independent as well: Fontanella and Magidor showed in \cite{MF:TPSR} that  $\TP(\aleph_{\omega^2+1})$ (and also $\neg \AP(\aleph_{\omega^2+1})$) is independent from  $\Delta_{\aleph_{\omega^2},\aleph_{\omega^2+1}}$ (see Section \ref{sec:delta} for details on $\Delta$-reflection).
\erm

\subsection{Equivalents of weak compactness in $L$}\label{sec:L}

If $V = L$,  many concepts of compactness at $\kappa$ provide the full characterization of weak compactness and hence imply inaccessibility of $\kappa$.  While we are not interested in the axiom $V = L$ \emph{per se}, the analysis of compactness under $V = L$ helps to clarify the interdependencies between the notions and isolates concepts which prevent compactness to occur at small cardinals. For this reason, we first summarize in Theorem \ref{th:full}  characterizations of weak compactness in $L$ which are expressed in the language of various mathematical structures.

As we will see in the next sections, some of these principles are consistent at small cardinals, but not necessarily equivalent.

\begin{theorem}\label{th:full} Suppose $V = L$. Then the following are equivalent for all regular uncountable cardinals:
\bce[(i)]
\item \emph{Logic and trees (Section \ref{sec:unctble}).}

\begin{itemize}
\item $\kappa$ is weakly compact for the infinitary logic $L_{\kappa,\kappa}$ with signature of size $\kappa$.
\item There are no $\kappa$-Aronszajn trees.\footnote{This follows from the stronger result that in $V =L$, there are no $\kappa$-Suslin trees if and only if $\kappa$ is weakly compact (Theorem \ref{th:Ls}).}
\end{itemize}

\item \emph{Squares and stationary reflection (Section \ref{sec:sr}).}

\begin{itemize}
\item All stationary subsets of $E^\kappa_\omega = \set{\alpha < \kappa}{\cf(\alpha) = \omega}$ reflect. See Theorem \ref{th:E}.
\item Stationary reflection $\SR(\kappa)$ holds. 
\item Todor{\v c}evi{\'c}'s square $\square(\kappa)$ does not hold. See Theorem \ref{th:sq}.
\item Fodor-type Reflection principle $\FRP(\kappa)$ holds. See Theorem \ref{th:FRP}.
\item $\kappa$ is $\Delta_{<\kappa,\kappa}$-compact for $\kappa$-sized algebras. See Theorem \ref{th:Ldelta}.
\end{itemize}

\item \emph{Algebras (Section \ref{sec:ab}).}

\begin{itemize}
\item  $\kappa$ is abelian compact for $\kappa$-sized abelian groups. See Theorem \ref{th:Lab}.
\end{itemize}

\item \emph{Graphs (Section \ref{sec:gr}).}

\begin{itemize}
\item $\kappa$ is countably coloring compact for $\kappa$-sized graphs. See Theorem \ref{th:Lcolor}. 
\item \label{full:chrom} $\kappa$ is countably chromatically compact for $\kappa$-sized graphs. See Theorem \ref{th:Lg}.
\item \label{full:RC} Rado's Conjecture for graphs of size $\kappa$ holds. See Theorem \ref{th:RC}.
\item There are no $\kappa$-Suslin trees. See Theorem \ref{th:Ls}.
\end{itemize}

\item \emph{Topological spaces (Section \ref{sec:top}).}

\begin{itemize}
\item \label{full:H}$\kappa$ is collectionwise Hausdorff compact for $\kappa$-sized topological spaces. See Theorem \ref{th:Haus}.
\end{itemize}

\ece
\end{theorem}

There are other principles for $\kappa$ which are equivalent to weak compactness of $\kappa$ in $V = L$, but we cannot list them all for lack of space. In particular, there are principles equivalent to Fodor-type Reflection principle $\FRP$ (for all regular $\kappa \ge \omega_2$, $\FRP(\kappa)$ holds) which can be stated locally to characterize weak compactness of $\kappa$ using Theorem \ref{th:FRP}: for instance, by Fuchino and Rinot \cite{MR2784001}, $\FRP$ is equivalent to the statement  
\begin{quote} Any Boolean algebra is openly generated if and only if it is $\omega_2$-projective.
\end{quote} and by Fuchino et al.\ \cite{more_frp} also to the statement  
\begin{quote}
For any locally countably compact topological space $X$, if all subspaces
of $X$ of cardinality $<\omega_2$ are metrizable,  then $X$ itself is also metrizable.
\end{quote}

\subsection{Squares and stationary reflection} \label{sec:sr}

\subsubsection{Stationary reflection for subsets of ordinals}

If $\kappa$ is regular and $S \sub \kappa$, we say that $S$ \emph{reflects} if there is $\alpha < \kappa$ of uncountable cofinality such that $S \cap \alpha$ is a stationary subset of $\alpha$. We say that $S \sub \kappa$ is \emph{non-reflecting} if $S$ does not reflect.

\begin{definition}\label{def:srplus}\label{def:srfull}
Suppose $\kappa$ is a regular cardinal. We say that \emph{stationary reflection} holds at $\kappa$, $\SR(\kappa)$, if the following hold:
\bce[(i)]

\item If $\kappa$ is regular and limit cardinal, then $\SR(\kappa)$ means that every stationary subset $S \sub \kappa$ reflects.

\item If $\kappa$ is singular, then $\SR(\kappa^+)$ means that every stationary subset $S \sub \kappa^+$ reflects.

\item If $\kappa$ is regular, then $\SR(\kappa^+)$ means that every stationary subset $S \sub \kappa^+ \cap \cof(<\kappa)$ reflects. Note that $\kappa^+ \cap \cof(\kappa)$ is always non-reflecting which prevents full stationary reflection in this case.
\ece
\end{definition}

As we will observe, by Theorem \ref{th:E},  stationary reflection just for subsets of $E^\kappa_\omega = \kappa \cap \cof(\omega)$ characterizes weak compactness in $L$, and hence \emph{a fortiori} so does reflection for all stationary sets. By Magidor's result in \cite{M:sr}, $\SR(\omega_2)$ and $\SR(\aleph_{\omega+1})$ are  consistent,\footnote{\label{R6}Magidor \cite{M:sr} used $\omega$-many supercompact cardinals for the construction. Hayut and Unger \cite{MR4231611} lowered the assumption to a single $\kappa^+$-$\Pi^1_1$-subcompact cardinal which is weaker than a single $\kappa^+$-supercompact cardinal $\kappa$.} so stationary reflection does not characterize weak compactness in general.

By Harrington and Shelah \cite{HS:sr}, $\SR(\omega_2)$ is equiconsistent just with a Mahlo cardinal, while a strengthening of stationary reflection at successor cardinals from Definition \ref{def:ssrplus} to \emph{simultaneous stationary reflection} for two sets, $\SR(\omega_2,2)$, is equiconsistent with a weakly compact cardinal by Magidor \cite{M:sr}. This leads to the following more general definition:

\begin{definition}\label{def:ssrplus}
Suppose $\kappa$ is a regular cardinal and $\theta < \kappa$ is a cardinal. We say that \emph{simultaneous stationary reflection} (for $\theta$-many sets) holds at $\kappa$, denoted $\SR(\kappa,\theta)$, if for every collection of $\theta$-many sets there is $\alpha < \kappa$ of an uncountable cofinality such that all sets from the collection reflect at $\alpha$.\footnote{$\SR(\kappa,\kappa)$ is always inconsistent: the sets $S_\alpha = \kappa \setminus \alpha$, for $\alpha < \kappa$, are stationary (in fact, clubs), but they cannot reflect at a single ordinal.}
\end{definition}

Simultaneous reflection for infinitely many sets $\SR(\kappa,\theta)$ has been studied as a large cardinal property which implies failures of various forms of square (for instance Hayut and Lambie-Hanson show in \cite{MR3730566} that $\square(\kappa,<\lambda)$ kills $\SR(\kappa,<\lambda)$).\footnote{See  Footnotes \ref{R2} and \ref{R3} for some more details on squares, and also Section \ref{sec:sq}.} There is an application of results related to simultaneous stationary reflection in Fuchs and Rinot \cite{paper33} for the Subcomplete Forcing Axiom $\SCFA$: it is shown that it entails $\neg \square^*_\lambda$ for every singular cardinal $\lambda> 2^\omega$ of countable cofinality (see Fuchs \cite{MR4765799} for the necessity of $\lambda> 2^\omega$). $\SCFA$ deserves to be mentioned here because it is compatible with $\CH$, unlike $\MA$. See also Sakai and Switzer \cite[Corollary 3.7]{sakai2025separatingsubversionforcingaxioms} which shows that $\SCFA + 2^\omega = \omega_2$ does not imply $\TP(\omega_2)$, hence $\SCFA + 2^\omega = \omega_2$ is strictly weaker regarding compactness-type consequences than $\PFA$, $\MA + \SSR$, $\MA + \WRP$ (see Section \ref{sec:gsr}), and $\RC + 2^\omega = \omega_2$ (see Section \ref{sec:uni}).

Finally, let us consider another variant which provides an ultimate strengthening of simultaneous reflection:

\begin{definition} \label{def:csrplus}
Suppose $\kappa$ is a regular cardinal. We say that  \emph{club stationary reflection} holds at $\kappa^+$, denoted $\CSR(\kappa^+)$, if for every stationary $S \sub \kappa^+ \cap \cof(<\kappa)$ there is a club $C \sub \kappa^+$ such that $C \cap \cof(\kappa)$ is included in the set of reflection points of $S$ (we say that the set of reflection points contains a $\kappa$-club).
\end{definition}

See Section \ref{sec:preserve} where we discuss preservation of various forms of stationary reflection by  forcing notions.

As we will see, stationary reflection is often necessary for compactness of various mathematical structures, but it is known that it is not strong enough to be a sufficient condition as well. Several concepts have been introduced which provide non-trivial strengthenings of stationary reflection with compactness-type consequences. We will discuss three such principles in some detail: \emph{Reflection Principle}, \emph{Fodor-type Reflection Principle} and \emph{$\Delta$-reflection}, in Sections  \ref{sec:gsr}, \ref{sec:FRP} and \ref{sec:delta}, respectively.

\subsubsection{Non-reflecting stationary sets and squares}\label{sec:sq}

\begin{definition}\label{def:E}
Suppose $\kappa$ is a regular cardinal. We denote by $E(\kappa)$ the assertion that there is a stationary set $S \sub \kappa \cap \cof(\omega)$ (let us denote the set $\kappa \cap \cof(\omega)$ by $E^\kappa_\omega$) which does not reflect, i.e.\ for every $\alpha < \kappa$ of uncountable cofinality, $S\cap \alpha$ is not stationary in $\alpha$.
\end{definition}

\begin{theorem}[Jensen \cite{JENfine}]\label{th:E}
If $V =L$, then an uncountable regular cardinal $\kappa$ is not weakly compact if and only if $E(\kappa)$ holds.
\end{theorem}

\begin{proof}
See a detailed and very well written exposition in Eklof and Mekler \cite{modules} (Chapter VI, Section 3).
\end{proof}

As we will see below, non-reflection stationary subsets can be used to construct incompact abelian groups, graphs and Suslin trees and various other structures. It is known that Jensen's square principle $\square_\kappa$ implies the existence of a non-reflecting stationary subsets of $E^{\kappa^+}_\omega$. Todor{\v c}evi{\'c} defined a square principles $\square(\kappa)$ which makes sense on limit cardinals as well and showed (among other things) that it characterizes weak compactness in $V = L$. Similarly to the principle $E(\kappa)$, $\square(\kappa)$ is often used to construct incompact mathematical structures (for instance chromatically incompact graphs as in Theorem \ref{th:Lg} or collectionwise incompact graphs in Section \ref{sec:top}). 

Let us give the definition of $\square(\kappa)$ for completeness, according to Todor{\v c}evi{\'c} \cite{MR908147}.\footnote{We will not give more details on squares for lack of space here, but see Footnotes \ref{R2} and \ref{R3} for some more information.}

\begin{definition}\label{def:sq}
Let $\kappa$ be a regular uncountable cardinal. We say that $\square(\kappa)$ holds if and only if there exists a sequence $\seq{C_\alpha}{\alpha < \kappa}$ such that:
\bce[(i)]
\item $\forall \alpha  <\kappa$, $C_\alpha \sub \alpha$ is a club subset of $\alpha$.
\item $\forall \alpha < \beta$, if $\alpha$ is a limit point of $C_\beta$, then $C_\alpha = C_\beta \cap \alpha$.
\item There is no club $C \sub \kappa$ such that for every limit point $\alpha$ of $C$, $C_\alpha = C \cap \alpha$.
\ece
\end{definition}

\begin{theorem}[Jensen \cite{JENfine}]\label{th:sq}
If $V =L$, then an uncountable regular cardinal $\kappa$ is not weakly compact if and only if $\square(\kappa)$ holds.
\end{theorem}

For the proof, see Jensen \cite[Theorem 6.1]{JENfine}, which actually obtains a stronger version of $\square(\kappa)$ on all regular non-weakly compact cardinals $\kappa>\omega_1$ in $L$ (the item (iii) is replaced by the stronger condition that there exists a stationary set $S \sub \kappa$ which contains no limit points of $C_\alpha$ for all $\alpha < \kappa$). The version of $\square(\kappa)$ in Definition \ref{def:sq} was used by Todor{\v c}evi{\'c} \cite[(1.10)]{MR908147}, among other things, to obtain a related result  applicable in the context of $V \neq L$ as well: If $\kappa > \omega_1$ is regular and not weakly compact in $L$, then $\square(\kappa)$ holds. Independently, Shelah and Stanley \cite[Proposition 6.1]{LSnon} state the same theorem (with the a slightly weaker assumption that $\kappa$ is regular in $L$ and its cofinality is $>\omega$ in $V$), with a considerably more detailed proof (and with an argument for a missing step in the original Jensen's proof).  Since $\PFA$ implies that $\square(\kappa)$ fails for every regular $\kappa>\omega_1$ by Todor{\v c}evi{\'c} \cite{MR763902}, it follows that $\PFA$ implies that all regular cardinals $\kappa>\omega_1$ are weakly compact in $L$. See the discussion in Rinot \cite[Theorem 1.3]{MR3271280} for more context and historical details.

\subsubsection{Generalized stationary reflection}\label{sec:gsr}

Foreman et al.\ \cite{FMS:MM} introduced a generalized form of stationary reflection called the \emph{Reflection Principle},  $\RP(\kappa)$, for regular $\kappa > \omega_1$ and subsets of $[\kappa]^\omega$, and showed that $\MM$ entails $\RP(\kappa)$ for all $\kappa >\omega_1$. As we mentioned above, this form of reflection is stronger than stationary reflection for subsets of ordinals, and entails compactness of various mathematical structures (see Section \ref{sec:FRP} for the Fodor-type Reflection Principle, a weakening of the Reflection Principles, and its consequences). See Jech \cite[Section 4]{JECHhandbook} or \cite[Section 38]{JECHbook}  for the definitions related to the generalized notion of stationarity in $[\kappa]^\omega$, or more generally in $\power_\kappa(\lambda)$.

\begin{definition}\label{def:rp}
$\RP(\kappa)$ holds if and only for every stationary $S \sub [\kappa]^\omega$ there is $I \sub \kappa$ of size $\omega_1$, $\omega_1\sub I$, with the cofinality of $I$ equal to $\omega_1$ such that  $S \cap [I]^\omega$ is stationary in $[I]^\omega$. We write $\RP$ if $\RP(\kappa)$ holds for every regular $\kappa \ge \omega_2$.
\end{definition}

\brm \label{rm:I}
Definition \ref{def:rp} is equivalent to a principle which asserts that there are stationarily many such $I$ in $[\kappa]^{\omega_1}$ (see for instance Fuchino and Rinot \cite[Lemma 1.1]{MR2784001} for a proof, which is stated for a weaker Fodor-type Reflection Principle which we discuss in the next section). In particular, $I$ in Definition \ref{def:rp}  satisfies without loss of generality the extra condition of containing a club $C$ of order type $\omega_1$ with $\sup(C) = \sup(I)$.
\erm

If we remove the condition on $I$ having the ordertype of $\omega_1$ (and leave other conditions exactly as they are), we obtain an ostensibly weaker notion\footnote{$\WRP(\kappa)$ is the principle first formulated by Foreman et al.\ \cite{FMS:MM} and called there the \emph{Strong Reflection}. It is open in general whether $\WRP(\kappa)$ and $\RP(\kappa)$ are equivalent (see K{\"o}nig et al.\ \cite{MR2320769} for some partial results). There are several other variants of these principles, such as Fleissner's Axiom R, see \cite{MR2610450} for an extensive discussion. See also Krueger \cite{MR2674000} for separations of some other principles related to $I$.} called \emph{the weak reflection property}, $\WRP(\kappa)$.

The principle $\RP(\kappa)$ implies stationary reflection for subsets of $E^\kappa_\omega$, and in particular $\RP(\omega_2)$ implies $\SR(\omega_2)$. Let us give an elementary proof:

\begin{lemma}\label{lm:RP}
$\RP(\kappa)$ implies that every stationary subset of $E^\kappa_\omega$ reflects (at an ordinal of cofinality $\omega_1$).
\end{lemma}

\begin{proof}
Let $S \sub E^\kappa_\omega$ be stationary. Define $S^* = \set{x \in [\kappa]^\omega}{\exists \alpha \in S, \mx{sup}(x) = \alpha}$. It is easy to observe that  $S^*$ is stationary in $[\kappa]^\omega$. In some detail: let $C^*_F$  generated by some $F: [\kappa]^{<\omega} \to \kappa$ and let $C$ be the club in $\kappa$ of ordinals closed under $F$. Since $S$ is stationary, there is some $\alpha \in C \cap S$ of countable cofinality. Since $\alpha$ is closed under $F$, a closure-type argument yields a countable $x$ cofinal in $\alpha$ closed under $F$, and any such $x$ is in $S^* \cap C^*_F$. By $\RP(\kappa)$, there is $I$ of size $\omega_1$, $\omega_1 \sub I$, and cofinality $\omega_1$ where $S^*$ reflects. Let  $\gamma = \mx{sup}(I)$ and let $D$ be a closed unbounded set in $\gamma$ of type $\omega_1$ contained in $I$ (this is possible by Remark \ref{rm:I}). Consider the set of its limit points $\mx{Lim}(D)$; it has the useful property that every $\alpha \in \mx{Lim}(D)$ has the countable cofinality. To finish the proof, we show that $S \cap \gamma$ is stationary in $\gamma$. Since $\mx{Lim}(D)$ is a club in $\gamma$, it suffices to consider club subsets of $\mx{Lim}(D)$. Let $A$ be any club in $\gamma$ contained in $\mx{Lim}(D)$ and define $A^* = \set{x \in [I]^\omega}{\exists \alpha \in A, \mx{sup}(x) = \alpha}$.  Clearly, $A^*$ is a club in $[I]^\omega$. Fix some $x \in S^* \cap A^*$. Then $\mx{sup}(x) \in A \cap S$ and the proof is finished.
\end{proof}

The principle $\WRP$ implies that every poset which preserves stationary subsets of $\omega_1$ is semi-proper, and this consequence was shown by Shelah \cite[Chapter XIII, 1.7]{SHELAHproper} to be equivalent to a strictly weaker reflection principle called the \emph{Semi-proper stationary reflection}, $\SSR$. The principle $\SSR$ has many of the consequences of $\RP$, such as the global failure of squares (see Sakai and Veli\u{c}kovi\'{c} \cite{MR3284479} for more details)\footnote{The notation is sometimes inconsistent in the literature. The article \cite{MR3284479} uses the notation $\SR$ (for \emph{stationary reflection}) to denote the principle which we denote $\WRP$.}. It it worth mentioning that both $\SSR$ and $\WRP(\omega_2)$\footnote{It is open whether $\RC$ implies $\WRP$.} are consequences of Rado's Conjecture, $\RC$, which we discuss in Section \ref{sec:graphs}.  They are both   related to strong tree properties:  Sakai and Veli{\v c}kovi{\'c} \cite[Theorem 3.1]{MR3284479}  showed that $\MA$ for Cohen forcing  with $\WRP$ implies $\ITP_{\omega_2}$, and with $\SSR$ implies $\TP_{\omega_2}$.

It  is known that $\WRP(\omega_2)$ implies $2^\omega \le \omega_2$ (see \cite[Theorem 7.8]{Tdich}). This is of some interest for us because it implies that $\WRP(\omega_2)$ and $\RC$ cannot be indestructible under Cohen forcing, in contrast to $\SR(\omega_2)$ and the Fodor-type Reflection Principle, $\FRP$, which we discuss in next Section \ref{sec:FRP},  which are preserved by all ccc forcings.

\brm \label{rm:MM}
The principle $\RP$, and  similar compactness principles like $\FRP$, are usually considered with stationarity in $[\kappa]^\omega$. Generalizations to stationarity in $[\kappa]^\theta$ for regular $\theta > \omega$ present new problems; most importantly, countable subsets of $\kappa$ can be identified with countable elementary submodels of $H(\kappa)$ if $|H(\kappa)| = \kappa$, and these are closed under all sequences of length $<\omega$. In contrast, uncountable submodels $M$ of $H(\kappa)$ may not be closed under $<|M|$-sequences, and various variants of approachability are considered to deal with these models (see for instance Krueger \cite{MR2674000}). This appears to be the same obstacle--- in a different language---as the one which prevents a straightforward generalization of the Proper Forcing Axiom or Martin's Maximum to higher cardinals.\footnote{Forcing axioms and the structure of stationary sets in $[\kappa]^\theta$ are closely connected. For example, by Woodin \cite[Theorem 2.53]{MR2723878}, for every $\P$, there exists for every collection $D$ of $\omega_1$-many dense sets a filter $F \sub \P$ meeting every set in $D$ if and only if the set $S^{\omega_1}_{\P}$ of all $M$ such that for a sufficiently large $\kappa$, $M \el H(\kappa)$, $\omega_1 \sub M$, $|M| = \omega_1$, $\P \in M$, and there exists a filter $F \sub M$ meeting every dense set in $\P$ which is an element of $M$, is stationary in $[H(\kappa)]^{\omega_1}$ (see Viale \cite{MR3453587} for a generalization to an arbitrary uncountable successor cardinal). By one of the several characterizations of properness, it is consistent (under $\PFA$) that preservation of stationary subsets of $[\kappa]^\omega$ by a forcing $\P$ implies stationarity of $S^{\omega_1}_{\P}$  (recall that $\P$ is proper if and only if $\P$ preserves stationary subsets of $[\kappa]^\omega$ for every regular uncountable $\kappa$). However, the preservation of stationarity of $[\kappa]^\theta$ for uncountable $\theta$ is a more complicated notion, as we briefly mentioned in the main body of the text where we discussed the notion of approachability.}
 However, note that $\Delta$-reflection discussed below does deal with stationarity in $[\kappa]^\theta$ for uncountable $\theta$, so it is in this sense more general than $\RP$ and $\FRP$ (but $\Delta$-reflection is rather specific as it can only hold at successors of singulars, or large cardinals).
\erm

\subsubsection{Fodor-type Reflection Principle}\label{sec:FRP}

The reflection principle $\RP$ has an important weakening called the \emph{Fodor-type Reflection Principle}, $\FRP$, introduced by Fuchino et al.\ \cite{MR2610450}. It is actually equivalent to various other compactness principles (see the bullets below). This universality makes $\FRP$ similar to $\Delta$-reflection which implies these principles as well, but is strictly stronger (see Section \ref{sec:delta} for $\Delta$-reflection).

\begin{definition}\label{def:FRP}
Let $\kappa$ be a regular cardinal $\ge \omega_2$. The \emph{Fodor-type Reflection Principle} for $\kappa$, $\FRP(\kappa)$, is the following statement: For any stationary $S \sub E^\kappa_\omega = \set{\alpha<\kappa}{\cf(\alpha) = \omega}$ and a mapping $g: S \to [\kappa]^{\le \omega}$ there is $I \in [\kappa]^{\omega_1}$ such that
\bce[(i)] 
\item $\cf(I) = \omega_1$,
\item $g(\alpha) \sub I$ for all $\alpha \in I \cap S$,
\item For any regressive $f: S \cap I \to \kappa$ such that $f(\alpha) \in g(\alpha)$ for all $\alpha \in S \cap I$, there is $\xi^* < \kappa$ such that $f^{-1}"\{\xi^*\}$ is stationary in $\mx{sup}(I)$.
\ece
We write $\FRP$ if $\FRP(\kappa)$ holds for every regular $\kappa \ge \omega_2$.
\end{definition}

\brm \label{rm:Ifrp}
Definition \ref{def:FRP} is equivalent to a principle which asserts that there are stationarily many such $I$ in $[\kappa]^{\omega_1}$. In particular, without loss of generality, we can assume $\omega_1 \sub I$ and that $I$ contains a club in $\sup(I)$ (see Fuchino and Rinot \cite[Lemma 1.1]{MR2784001} for details).
\erm

By Remark \ref{rm:Ifrp}, (iii) immediately gives that $S \cap I$ is stationary in $\sup(I)$, and thus $\FRP(\kappa)$ implies that all stationary subsets of $E^\kappa_\omega$ reflect (at an ordinal of cofinality $\omega_1$). By Theorem \ref{th:E}, this immediately yields:

\begin{theorem}\label{th:FRP}
Suppose $V = L$ and $\kappa$ is an uncountable regular cardinal, then the following are equivalent:
\bce[(i)]
\item $\kappa$ is weakly compact.
\item $\FRP(\kappa)$.
\ece
\end{theorem}

By results in Fuchino et al.\ \cite{MR2610450}, $\FRP(\kappa)$ is strictly stronger than stationary reflection for subsets of $E^\kappa_\omega$, and strictly weaker than other natural principles postulating reflection for generalized stationary sets such as $\RP$ or $\WRP$ mentioned above.

By Miyamoto \cite{frp_consistency}, $\FRP(\omega_2)$ is equiconsistent just with the Mahlo cardinal, analogously to the usual stationary reflection $\SR(\omega_2)$. However,  $\FRP$ in its global form is quite strong (all known constructions start with a strongly compact cardinal) because it implies the failure of $\square_\lambda$ for any uncountable $\lambda$ (see Fuchino and Rodrigues \cite{ReflPrinciples}).

$\FRP$ is equivalent to several compactness properties for specific mathematical structures:

\begin{itemize}
\item By Fuchino and Rinot \cite{MR2784001} $\FRP$ is equivalent to the statement that every Boolean algebra is openly generated if and only if it is $\omega_2$-projective.
\item By Fuchino et al.\ \cite[Theorem 3.1]{more_frp} , $\FRP$ is equivalent to countable color compactness of all graphs. See also Section \ref{sec:color}.
\item By \cite[Theorem 4.1]{more_frp}, $\FRP$ is equivalent to collectionwise Hausdorff compactness. See also Section \ref{sec:top}.
\item By \cite{more_frp}, $\FRP$ is equivalent to another topological property, namely ``For any locally countable compact topological space $X$, if all subspaces of $X$ of cardinality $\le \omega_1$ are metrizable, then $X$ is also metrizable''. The fact that this principle is a consequence of $\FRP$ is proved in Fuchino et al.\ \cite{MR2610450}.
\end{itemize}

By \cite[Theorem 3.4]{MR2610450}, $\FRP(\kappa)$ for any regular $\kappa \ge \omega_2$ is preserved by all ccc forcings. By a discussion in Section \ref{sec:sep}, this immediately implies that $\FRP$ (even when combined with $\ISPii$) does not decide some of the well-known independent statements in mathematics.

\subsubsection{$\Delta$-reflection}\label{sec:delta}

In order to prove Theorem \ref{th:fullc} on compactness for abelian groups (which we discuss below), Magidor and Shelah formulated in \cite{MS:free}  a principle postulating a strong form of stationary reflection, called the $\Delta$-reflection. Suppose $\lambda < \kappa$ and $\kappa$ is regular. As we will see in Definition \ref{def:delta}, $\Delta_{\lambda,\kappa}$-reflection is similar to the Reflection Principle and the Fodor-type Reflection Principle, which we discussed above, but it is stronger in the sense that it deals with stationarity of subsets of $[\kappa]^\theta$, where $\theta<\lambda$ can be uncountable. Since in its strongest form, i.e.\ $\Delta_{<\kappa,\kappa}$-reflection, implies full stationary reflection at a successor cardinal $\kappa$, $\kappa$ must be a successor of a singular cardinal for this form of reflection to hold.  

The principle is applicable to a wide class of mathematical structures whose compactness shares certain  properties with ``freeness'' of abelian groups.\footnote{The problem of compactness for abelian groups, originating  itself from a work of Shelah on the Whitehead's Conjecture, motivated the formulation of this principle, see Magidor and Shelah \cite{MS:free} and Section \ref{sec:ab}. See Remark \ref{rm:W} and Section \ref{sec:sep} for more details on Whitehead's Conjecture.}

\begin{definition}[Magidor--Shelah, \cite{MS:free}] \label{def:delta}
Suppose $\lambda < \kappa$ are uncountable cardinals, with $\kappa $ regular. $\Delta_{\lambda,\kappa}$ is the statement that, for every stationary $S \sub \set{\alpha <\kappa}{\cf(\alpha) < \lambda}$ and every algebra $A$ on $\kappa$ with  fewer than $\lambda$ operations, there is a subalgebra $A^*$ of $A$ such that, letting $\delta = \mx{otp}(A^*)$, the following hold:
\bce[(i)]
\item $\delta$ is a regular cardinal;
\item $\delta < \lambda$;
\item $S \cap A^*$ is stationary in $\mx{sup}(A^*)$.
\ece
We say that $\kappa$ has the $\Delta_{<\kappa,\kappa}$-reflection if $\Delta_{\lambda,\kappa}$ holds for every $\lambda < \kappa$. We say that $\kappa$ has the global $\Delta_\kappa$-reflection if $\Delta_{\kappa,\nu}$ holds for all regular $\nu > \kappa$.
\end{definition}

Notice that if $\kappa$ is singular, then $\Delta_{\kappa,\kappa^+}$ implies $\SR(\kappa^+)$.

$\Delta$-reflection provides a characterization of weak compactness in $L$:

\begin{theorem}\label{th:Ldelta}
Suppose $V = L$ and $\kappa$ is an uncountable regular cardinal, then the following are equivalent:
\bce[(i)]
\item $\kappa$ is weakly compact,
\item $\kappa$ has the $\Delta_{<\kappa,\kappa}$-reflection.
\ece
\end{theorem}

\begin{proof}
(i)$\to$(ii) is implicit in Magidor and Shelah \cite{MS:free}, with a detailed proof in Fontanella and Hayut \cite[Proposition 2.4]{FH:delta}. (ii)$\to$(i) follows by Theorem \ref{th:E} and by the observation that $\Delta_{<\kappa,\kappa}$-reflection implies $\neg E(\kappa)$.\footnote{The article \cite{MS:free} allows $\lambda = \kappa$ in the definition of $\Delta$-reflection, but there are also versions with $\lambda< \kappa$, as in Fontanella and Hayut \cite{FH:delta}. The case $\lambda = \kappa$ is only relevant for weakly inaccessible $\kappa$ and hence unnecessary for successor cardinals which we discuss in this article (with the exception of  Theorem \ref{th:Ldelta}, but $\Delta_{<\kappa,\kappa}$ is sufficient here since it implies reflection for stationary subsets of $E^\kappa_\omega$).}
\end{proof}

The $\Delta_{\kappa,\kappa^+}$-reflection for a singular $\kappa$ implies abelian compactness of $\kappa^+$ by  \cite{MS:free}, and several other compactness principles, such as the countable color compactness  (Lambie-Hanson and Rinot \cite[Proposition 2.23]{LHR:numbers}) and  the collectionwise Hausdorff compactness (implicit in Shelah \cite{MR540629} and Magidor and Shelah \cite{MS:free}). Note that $\Delta_{\kappa,\kappa^+}$ holds whenever $\kappa$ is a singular limit with countable cofinality of strongly compact cardinals by \cite[Corollary 2.5]{FH:delta}, but it can also hold at small cardinals such as $\aleph_{\omega^2+1}$ or the successor of the least cardinal fixed point, which is the key ingredient in \cite{MS:free} to obtain a model where these cardinals are abelian compact.

It is worth observing that the countable color compactness and collectionwise Hausdorff compactness can---unlike the $\Delta_{<\kappa,\kappa}$-reflection and abelian compactness---hold at $\omega_2$, as we discussed in the context of $\FRP$ in Section \ref{sec:FRP}.

\subsection{Abelian groups}\label{sec:ab}

There is a natural notion of compactness related to abelian groups which emerged as an important concept during the analysis of Whitehead's Conjecture and the proof of its independence from $\ZFC$ by Shelah \cite{Shelah:abelian, Shelah:ab}.

\begin{definition}\label{def:freewise}
We say that an abelian group of size $\kappa\ge \omega_1$ is \emph{almost-free} if every abelian subgroup of size $<\kappa$ is free.  If every almost-free abelian group of size $\kappa$ is free, we say that $\kappa$ is \emph{abelian compact (for groups of size $\kappa$)}.\footnote{There does not seem to be an established terminology for this concept.} We say that $\kappa$ is \emph{fully abelian compact} if every abelian group of size $\ge \kappa$ is free if and only if every subgroup of size $<\kappa$ is free.
\end{definition}

We observe in Theorem \ref{th:Lab} that abelian compactness characterizes weakly compact cardinals in $L$. It is natural to ask about the possible range of abelian compactness if $V \neq L$. Note that in contrast to  other compactness principles we discuss, abelian compactness does not refer to finite substructures (every free group is infinite), and therefore the notion is strictly speaking meaningful only for uncountable groups. A weak analogy for countable groups is Pontriyagin's theorem that every countable abelian group is free if and only if all its finitely generated subgroups are free.

Here are some initial observations (see Fuchs \cite{L:inf} for more  details related to abelian groups in general and Eklof's and Mekler's excellent monograph \cite{modules} for set-theoretic connections):

\begin{theorem}\label{th:abinit}The following hold:
\bce[(i)]
\item A principles saying that ``a group $G$ is free if and only if  every finitely generated subgroup $H$ of $G$ is free''\footnote{Note that the implication from left to right is provable from  a weak form of $\AC$; see our discussion of this point in Section \ref{sec:omega}.} is false in $\ZFC$.
\item If $\kappa$ is weakly compact, then $\kappa$ is abelian compact.
\item If $\kappa$ is strongly compact, then $\kappa$ is fully abelian compact.
\ece
\end{theorem}

\begin{proof}
(i). Baer and Higman constructed in 1950's a non-free  abelian groups of size $\omega_1$ whose all countable subgroups are free.

(ii) and (iii). This is an easy observation using the properties of weakly compact or strongly compact cardinals. See \cite[Chapter IV, Theorem 3.2]{modules} for more details.
\end{proof}

\begin{theorem}\label{th:Lab} If $V = L$ and $\kappa$ is an uncountable regular cardinal, then the following are equivalent:
\bce[(i)]
\item $\kappa$ is weakly compact.
\item $\kappa$ is abelian compact.
\ece
\end{theorem}

\begin{proof}
The non-trivial direction (ii)$\to$(i) follows by Theorem \ref{th:E} and a construction of an almost-free non-free abelian group using a non-reflecting stationary subset of $E^\kappa_\omega$ (see \cite[Section VII, Theorem 1.4]{modules}).
\end{proof}

This characterization of weak compactness can fail in general: Magidor and Shelah showed in \cite{MS:free} that modulo large cardinals, there is a generic extension in which there can be non-weakly compact inaccessible cardinals, and also successor cardinals, which are abelian compact and even fully abelian-compact.

However, unlike the case of principles related to trees such as $\ISP_\kappa$, there is a considerable amount of $\ZFC$ restrictions regarding the possible extent of abelian compactness on small cardinals. Generalizing Baer's and Higman's result that there always exists an almost-free non-free group of size $\omega_1$, Magidor and Shelah  proved in \cite{MS:free} the following more extensive sufficient condition for the existence of almost-free non-free abelian groups:

\begin{lemma}\label{lm:MS}
Assume there exists an almost-free non-free abelian group of size $\delta$ for some regular $\delta \ge \omega$ (vacuously true for $\delta = \omega$). Let $C_\delta$ the closure of  $\{\delta\}$ under the operations $\lambda \mapsto \lambda^+$ and $(\lambda,\kappa) \mapsto  \lambda^{+\kappa+1}$. Then there is an almost-free non-free abelian group of size $\lambda$ for every $\lambda \in C_\delta$.
\end{lemma}

The proof of this lemma is based on the fact that incompact abelian groups of certain sizes generate larger incompact abelian groups with appropriate non-reflecting stationary sets around.\footnote{For instance, as we already mentioned, Eklof, Gregory and Shelah (independently) showed that if $\kappa < \lambda $ are regular cardinals and there is a non-free almost abelian group of size $\kappa$ and there is a non-reflecting stationary subset $S \sub \lambda \cap \cof(\kappa)$, then there is a non-free almost free abelian group of size $\lambda$ (a ``pump-up lemma''). See \cite{modules}, Theorem 2.3.} 
Conversely, a strong form of stationary reflection, the $\Delta$-reflection which we reviewed in Section \ref{sec:delta}, is a sufficient condition for abelian compactness, by \cite{MS:free}.

Lemma \ref{lm:MS} applied with  $C_\omega$  implies  that there can be no abelian compact regular cardinal below $\aleph_{\omega^2}$ and  that there can be no fully abelian compact regular cardinal below the first fixed point of the $\aleph$ function. From the existence of infinitely many supercompact cardinals \cite{MS:free} produces models which are optimal with respect to the restrictions in Lemma \ref{lm:MS}: 

\begin{theorem} [\cite{MS:free}] \label{th:fullc}
Suppose there are infinitely many supercompact cardinals, then the following are consistent:
\bce[(i)]
\item $\aleph_{\omega^2+1}$ is abelian compact, i.e.\ any abelian group of size $\aleph_{\omega^2+1}$ which is almost-free is also free.
\item Suppose $\kappa$ is the least cardinal such that $\kappa = \aleph_\kappa$. Then $\kappa$ is fully abelian compact, i.e.\  if $G$ is an abelian group of size $\ge \kappa$ such that every subgroup of size $<\kappa$ is free, then $G$ itself is free.
\ece
\end{theorem}

It follows that modulo large cardinals $\ZFC$ cannot prove the existence of an almost-free non-free abelian group above the least cardinal fixed point.

It is of separate interest that Magidor and Shelah reformulated the compactness principle dealing with almost-free abelian groups in terms of a more general concept of \emph{transversals}. 

Recall that $f$ is a \emph{transversal} for a set $A$ if $f$ is an injective choice function on $A$.

\begin{definition}\label{def:PT}
Suppose $\omega_1 \le \lambda < \kappa$ are infinite cardinals. We say that $\PT(\kappa,\lambda)$ holds if for every set $A$ of size $\kappa$ such that every $a \in A$ has size $<\lambda$, if every subset $X \sub A$ of size $<\kappa$ has a transversal, so does the whole set $A$. We denote $\neg \PT(\kappa,\lambda)$ by $\NPT(\kappa,\lambda)$.
\end{definition}

In \cite{Shelah:inc} Shelah showed that $\PT(\kappa,\omega_1)$ is equivalent to the property that every almost-free abelian group of size $\kappa$ is free. 

\brm
As we already mentioned, the proof of Theorem \ref{th:fullc} is carried out in a more general setting of $\Delta$-reflection which we discussed in Section \ref{sec:delta}, which implies other compactness principles apart from abelian compactness, and is therefore of separate interest.
\erm

\brm \label{rm:W}
The notion of abelian compactness is connected to research of Shelah and others into the Whitehead problem in abelian group theory (see Section \ref{sec:sep} for a few more details on the Whitehead problem). For the context of this section, we observe that if $\kappa$ is abelian compact and every Whitehead group of size $<\kappa$ is free, then by abelian compactness of $\kappa$ every Whitehead group of size $\kappa$ must be free as well. This provides a form compactness not only for the property of being free, but also for the more general property of being Whitehead. However, Shelah showed that these two notions of compactness can consistently diverge: by   \cite{MR1998104} it is consistent that $\kappa$ is strongly inaccessible, $\GCH$ holds and $\ast_\kappa$ holds: 
\begin{itemize}
\item[$\ast_\kappa$] Every Whitehead group of size $<\kappa$ is free + every almost-free abelian group of size $\kappa$ is Whitehead + $\kappa$ is \emph{not} abelian compact.
\end{itemize}

By a follow-up result in \cite{MR2886755}, the least $\kappa$ satisfying $\ast_\kappa$ cannot be accessible.
\erm

\brm
The \label{rm:R3} study of the Whitehead problem also motivated Shelah's famous result that singular cardinals are provably in $\ZFC$ compact for a range of principles. In particular singular cardinals are provably abelian compact. This result shows there is a sharp distinction between compactness at regular cardinals (which we study in this article) and singular cardinals. See Shelah \cite{Shelah:ab} for more details.
\erm

\subsection{Graphs and trees} \label{sec:gr}

\subsubsection{Compactness for the coloring number}\label{sec:color}

Suppose $\mathcal G = (G,E)$ is an (undirected) graph and $<$ some fixed well-order on $G$. The neighborhood $N^<_{\mathcal G}(x)$ of a vertex $x \in G$ with respect to $<$ is defined by $N^<_{\mathcal G}(x) = \set{y}{\{x,y\} \in E \mbox{ and }y < x}$. 

\begin{definition}\label{def:color}
The \emph{coloring number} number of $\mathcal G$,  $\Color(\mathcal G) = \chi$, is defined to be the least cardinal $\chi$ such that there is a well-order $<$ on $G$ such that $|N^<_{\mathcal G}(x)|<\chi$ for all $x \in G$.
\end{definition} 

The notion of the coloring number can be seen as a more constructive version of the usual chromatic number of a graph $\mathcal G = (G,E)$: a function $c:G \to \chi$ is called a chromatic coloring of $\mathcal G$ if $\{x,y\} \in E$ implies $c(x) \neq c(y)$ for all $x,y \in G$. 

\begin{definition}\label{def:chrom}
The \emph{chromatic number} of $\mathcal G$, $\Chr(\mathcal{G})$, is defined as the least cardinal $\chi$ such that there is a chromatic coloring with range $\chi$.
\end{definition}

The well-order $<$ in the definition of colorwise compactness provides an explicit construction of a chromatic function with small domain:

\begin{lemma}
Suppose $\mathcal G = (G,E)$ is a graph. Then $\Chr(\mathcal G)  \le \Color(\mathcal G)$.
\end{lemma}

\begin{proof}
Suppose $<$ is a well-order of $G$ which witnesses $\Color(\mathcal G) = \chi$, i.e.\ $|N^<_{\mathcal G}(x)|<\chi$ for all $x \in G$, where $N^<_{\mathcal G}(x) = \set{y}{\{x,y\} \in E \mbox{ and }y < x}$. Let $\seq{v_\alpha}{\alpha < \gamma}$ be the enumeration of $G$ given by $<$, for some $\kappa \le \gamma< \kappa^+$. We will define chromatic coloring $c:G \to \chi$ by induction on $\gamma$. Suppose $\restr{c}{\alpha}$ is defined. Define $c(v_\alpha)$ to be the least $\xi < \chi$ in $$\chi \setminus \set{c(v_\beta)}{v_\beta \in N^{<}_{\mathcal G}(v_\alpha)}.$$ It is clear that $c$ is a chromatic coloring: if $\{x,y\} \in E$ and $y < x$, then $c(y) \neq c(x)$ by the construction.
\end{proof}

The coloring number can be used to define a notion of compactness of graphs. We will give a specific definition of countable coloring compactness in analogy of countable chromatic compactness which will review in the next section (see Definition \ref{def:chcompact}). For more details about the notions of coloring and chromatic numbers with respect to compactness see Lambie-Hanson and Rinot \cite{LHR:numbers}.

\begin{definition}\label{def:ccompact}
Suppose $\kappa$ is a regular uncountable cardinal. We say $\kappa$ is \emph{countably coloring compact (for graphs of size $\kappa$)} if for every graph $\mathcal G = (G,E)$ of size $\kappa$, if every subgraph $\mathcal G'$ of $\mathcal G$ with $|G'| < |G|$ satisfies $\Color(\mathcal G') \le \omega$, then also $\Color(\mathcal G) \le \omega$.
\end{definition}

As we mentioned in Section \ref{sec:FRP}, this definition can be equivalently expressed in terms of the Fodor-type Reflection Principle \cite[Theorem 3.1]{more_frp} (without the restriction on the size of $G$).

Coloring compactness  provides a characterization of weak compactness in $L$:

\begin{theorem}\label{th:Lcolor}
If $V = L$ and $\kappa$ is an uncountable regular cardinal, then the following are equivalent:
\bce[(i)]
\item $\kappa$ is weakly compact,
\item $\kappa$ is countably coloring compact.
\ece
\end{theorem}

\begin{proof}
The non-trivial direction (ii)$\to$(i) follows from a result of Shelah that non-reflecting stationary sets ensured by the principle $E(\kappa)$ from Theorem \ref{th:E} yield countably coloring incompact graphs. See Lambie-Hanson and Rinot \cite[Theorem 2.17]{LHR:numbers} for a detailed proof (in a more general setting).
\end{proof}

\brm \label{rm:M}
Note that $\omega_2$ can be countably coloring compact, and a Mahlo cardinal is the optimal consistency strength. See Miyamoto \cite{frp_consistency} for more details.
\erm

It is known that $\Delta_{<\kappa,\kappa}$-reflection implies countable coloring compactness, see Lambie-Hanson and Rinot \cite[Proposition 2.23]{LHR:numbers} (it is not known whether it implies countable chromatic compactness as well, see next Section \ref{sec:graphs}). In particular, Theorem \ref{th:fullc} implies that modulo large cardinals $\ZFC$ cannot prove that there are coloring incompact graphs above the first cardinal fixed point.

\subsubsection{Chromatic compactness, Rado's Conjecture, and transfer principles}\label{sec:graphs}

Unlike the coloring number of graphs which we discussed above, the more familiar notion of chromatic compactness appears to be harder to analyse. In analogy with Definition \ref{def:ccompact}, we define (following Todor{\v c}evi{\'c} \cite{Tdich}):

\begin{definition}\label{def:chcompact}
Suppose $\kappa$ is a regular uncountable cardinal. We say $\kappa$ is \emph{countably chromatically compact (for graphs of size $\kappa$)} if for every graph $\mathcal G = (G,E)$ of size $\kappa$, if every subgraph $\mathcal G'$ of $\mathcal G$ with $|G'| < |G|$ satisfies $\Chr(\mathcal G') \le \omega$, 
then also $\Chr(\mathcal G) \le \omega$.\footnote{\label{R8} A generalization of this property without the restriction on the size of $G$ is a consequence of $\kappa$ being an $\omega_1$-strongly compact cardinal, see Bagaria and Magidor \cite{MR3226024}. However, it is in general open whether the converse is true, i.e., whether countable chromatic compactness of $\kappa$ (even for an inaccessible $\kappa$) implies that $\kappa$ is weakly compact (for graphs of size $\kappa$) or $\omega_1$-strongly compact (for unlimited size of graphs). Compare with Theorem \ref{th:Lg} formulated for $V = L$.}
\end{definition}

It is easy to observe that if $\kappa$ is weakly compact, then every $\mathcal G$ of size $\kappa$ is countably chromatically compact, and if $\kappa$ is strongly compact, then this holds for all graphs $\ge \kappa$. Countably chromatically incompact graphs can be constructed by means of non-reflecting stationary subsets of $E^\kappa_\omega$, and thus provide a characterization of weak compactness of $\kappa$ in $L$. 

\begin{theorem}[Shelah]\label{th:Lg} If $V = L$, then the following are equivalent for an uncountable regular $\kappa$:
\bce[(i)]
\item $\kappa$ is weakly compact.
\item $\kappa$ is countably chromatically compact.
\ece
\end{theorem}

\begin{proof}
The non-trivial direction follows for instance from Theorem \ref{th:E} and the result of Shelah in \cite{Shelah:graph} which proves that $E(\kappa)$ implies that there is a countably chromatically incompact graph $\mathcal G$ of size $\kappa$.
\end{proof}

However, unlike the countable coloring compactness, it is open whether there can be a countably chromatically compact cardinal $\kappa$ which is not weakly compact. At the moment, all that is known that any such $\kappa$ must be greater or equal to $\beth_\omega$ (see Todor{\v c}evi{\'c} \cite{Tdich} for a survey of the topic).

\brm \label{rm:R2}
There are some provable differences between the chromatic and color compactness. While it is possible for the chromatic compactness  to fail with an arbitrarily large gap, see Shelah \cite{Sh:347} and Rinot \cite{paper12}, the incompactness of the coloring number is limited to a gap of no more than two cardinals, see Lambie-Hanson and Rinot \cite[Corollary 2.18]{LHR:numbers}.
\erm

There are some partial results for smaller classes of graphs. We state a few examples here and refer the reader to more details and examples in Lambie-Hanson and Rinot \cite{LHR:numbers}.

Galvin's and Rado's conjectures were originally formulated for graphs, but can be equivalently stated in the language of partial orders and trees:

\begin{itemize}
\item Galvin's Conjecture; in the language of graphs related to chromatic numbers of incomparability graphs.\footnote{Graphs of the form $(G,E)$ where $G$ is the domain of a partially ordered set $(P,<)$ (or more generally a quasi-ordered set) and $\{x,y\} \in E$ iff $x,y$ are incomparable in $<$.} It is consistent that for any partially ordered set $P$, $P$ can be decomposed into countably many chains (sets of pairwise comparable elements under the ordering of $P$) if and only if every suborder $P'$ of size $\le \omega_1$ can be decomposed into countably many chains.

\item Rado's Conjecture, $\RC$; in the language of graphs related to chromatic numbers of interval graphs.\footnote{Graphs of the form $(G,E)$ where $G$ is the set of intervals (or more generally convex sets) of some linearly ordered set $(L,<)$ and $\{I,J\} \in E$ iff $I \cap J \neq \emptyset$.} It is consistent that for any tree $T$, $T$ can be decomposed into countably many antichains (sets of pairwise incomparable nodes in the tree ordering) if and only if every subtree $T'$ of size $\le \omega_1$ can be decomposed into countably many antichains.\footnote{Notice that Rado's Conjecture is interesting only for trees $T$ of height $\omega_1$ without cofinal branches: Since every level of a tree is an antichain, all trees with a countable height are trivially decomposable into countable many antichains. On the other hand, no tree with uncountable chains (branches) can be decomposed into countably many antichains.}
\end{itemize}

Galvin's Conjecture is still open. Todor{\v c}evi{\'c} showed in \cite{T:Rado}  that $\RC$ is relatively consistent with the existence of a strongly compact cardinal:  if a strongly compact cardinal $\kappa$ is turned into $\omega_2$ using Levy collapse, then $\RC + \CH$ holds in the resulting model. Zhang \cite{MR4094551} elaborated on Todor{\v c}evi{\'c}'s observation that $\RC$ holds in the Mitchell model and showed that $\RC + \TP_{\omega_2}$ is consistent from a strongly compact cardinal.

It is known that $\RC$ contradicts forcing axioms, which makes it rather exceptional because it provides an alternative to $\MM$, as we will briefly discuss in Section \ref{sec:uni}. It implies, among other things, $2^\omega \le \omega_2$, the failure of $\square(\kappa)$ for any regular $\kappa \ge \omega_2$, and hence $\AD^{L(\R)}$, and also the Strong Chang's Conjecture and Weak Reflection Principle $\WRP(\omega_2)$. See \cite{Tdich} for more details and references and Section \ref{sec:uni} for a comparison of  theories $\ZFC + \MM$ and $\ZFC + \RC + 2^\omega = \omega_2$ as two examples of unification of logical and mathematical principles.

Let us define a more detailed notion for Rado's Conjecture.

\begin{definition} \label{def:RC}
We say that Rado's Conjecture holds for a cardinal $\kappa$ and write $\RC(\kappa)$ if every tree $T$ of size $\kappa$ can be decomposed into countably many antichains if and only if every subtree $T'$ of size $< \kappa$ can be decomposed into countably many antichains.\footnote{See Footnote \ref{ft:RC3} for a three-parameter version of $\RC$.}
\end{definition}

\begin{theorem}
[Todor{\v c}evi{\'c}]\label{th:RC} If $V = L$, then the following are equivalent for an uncountable regular $\kappa$:
\bce[(i)]
\item $\kappa$ is weakly compact.
\item $\RC(\kappa)$.
\ece
\end{theorem}

\begin{proof}
The non-trivial direction follows from Todor{\v c}evi{\'c} \cite[Theorem 5.2]{Tdich} for $\theta = \kappa$, which yields the failure of $\square(\kappa)$ from the assumption of $\RC(\kappa)$.
\end{proof}

While the least countably chromatically compact cardinal must be above $\beth_\omega$, Foreman and Laver showed in \cite{MR925267} that $\omega_2$ can be  chromatically compact in a weaker sense: starting with a huge cardinal, they constructed a model where for every graph of size $\omega_2$, if all its subgraphs of size $\omega_1$ have the countable chromatic number, then the chromatic number of the whole graph is at most $\omega_1$. This weaker form of countable chromatic compactness is a part of a more general body of results related \emph{downward transfer principles}  between successor cardinals. Discussing in detail transfer principles such as various forms of Chang's Conjecture is beyond the scope of this article, but the reader can consult the afore-mentioned article of Foreman and Laver \cite{MR925267}, and also Foreman \cite{MR2538021}, Cox \cite{cox2020adjoiningthingswantsurvey} or Eskew and Hayut \cite{MR3748588}. We just note that there are connections between transfer principles and compactness principles in this article: for instance, by results of Torres-P{\'e}rez and Wu,  Chang's Conjecture is equivalent to the tree property at $\omega_2$ if we assume $\neg \CH$ (\cite{MR3431031}) and Strong Chang's conjecture (with $\neg \CH$) implies the strong tree property at $\omega_2$, $\TP_{\omega_2}$ in our notation (\cite{MR3600760}). This is relevant for this article because Strong Chang's conjecture is implied by $\RC$, and hence these result show that $\RC + 2^\omega = \omega_2$ unifies certain logical and mathematical principles (see Section \ref{sec:uni} for more discussion). 

For several of the transfer properties the best lower bound known today is often at the level of huge cardinals (like in Foreman and Shelah \cite{MR925267} mentioned above), which sets these principle apart from the principles derivable from supercompact cardinals\footnote{Some variants of Chang's Conjecture can be obtained from a supercompact cardinal, though. See Eskew and Hayut \cite{MR3748588} for more details.} which we discuss here, but by discussion in Section \ref{sec:inacc} it is still expressible as a specific type of compactness of some logic.

\subsubsection{Suslin Hypothesis}\label{sec:Suslin}

Suppose $\kappa$ is an uncountable regular cardinal. Recall that a $\kappa$-Aronszajn tree is called \emph{$\kappa$-Suslin} if it has no antichains of size $\kappa$. Let us write $\SHg(\kappa)$ for the statement that there are no $\kappa$-Suslin trees.

Every $\kappa$-Suslin tree is $\kappa$-Aronszajn, hence $\TP(\kappa)$ implies $\SHg(\kappa)$. However, Suslin trees are of independent interest for cardinals on which the tree property necessarily fails, but some degree of compactness can be salvaged by having no Suslin trees: whenever $\kappa^{<\kappa} = \kappa$, then there are always $\kappa^+$-Aronszajn trees by Specker's result, yet $\kappa^+$-Suslin trees may not exist.\footnote{\label{ft:SH} It is known that $\SH$ is equiconsistent with $\ZFC$. Above $\omega_1$, the consistency strength increases according to the currently known lower bounds: By Shelah and Stanley \cite{shelah_stanley}, if $\kappa$ is an infinite cardinal and $2^\kappa = \kappa^+$, then $\SHg(\kappa^{++})$ implies that $\kappa^{++}$ is inaccessible in $L$. For an uncountable $\kappa$, Lambie-Hanson and Rinot \cite{paper31} improved the lower bound to a Mahlo cardinal: if $\kappa > \omega$, $2^\kappa = \kappa^+$, then $\SHg(\kappa^{++})$ implies $\kappa^{++}$ is Mahlo in $L$. In the presence of $\GCH$, the consistency strength of $\SHg(\kappa^+)$ for an uncountable $\kappa$ is the weakly compact cardinal ($\kappa^+$ is weakly compact in $L$ in this case), see Rinot \cite{paper24} for more details.} The case of $\kappa = \omega$ is of special interest in connection with the characterization of the real line: if there are no $\omega_1$-Suslin trees, then the real line has a combinatorial characterization through the non-existence of $\omega_1$ many open non-empty pairwise disjoint itervals of the reals. This is the \emph{Suslin hypothesis}, $\SH$, which we also briefly discuss in  Section \ref{sec:sep}. 

In this sense, $\SHg(\kappa)$ is a compactness principle in its own right\footnote{\label{R9} It is worth observing that $\SHg(\kappa)$ is equivalently expressed by stating the Ramsey property for $\kappa$-trees instead of graphs of size $\kappa$ (i.e.\ that every $\kappa$-tree contains a clique or an anti-clique of size $\kappa$). While Ramsey property for graphs  implies inaccessibility of $\kappa$, this is not the case for $\kappa$-trees, underscoring the difference between graphs and trees.}  and one which is compatible with $\GCH$ like other mathematical principles which we discuss here. Moreover, analogously to incompact abelian groups or graphs, the construction of Suslin trees is related to the existence of non-reflecting stationary sets (see Theorem \ref{th:Ls} for more details), in telling contrast to Aronszajn trees in the context of $\neg \GCH$.\footnote{By Cummings et al.\ \cite{8fold}, the principle $E(\omega_2)$ together with $2^\omega = \omega_2$ does not imply there are $\omega_2$-Aronszajn trees.} 

By Jensen's result, Suslin trees characterize weak compactness in $L$:

\begin{theorem}[Jensen \cite{JENfine}]\label{th:Ls} If $V = L$, then the following are equivalent for an uncountable regular $\kappa$:
\bce[(i)]
\item $\kappa$ is weakly compact.
\item There are no $\kappa$-Suslin trees, i.e.\ $\SHg(\kappa)$.
\ece
\end{theorem}

\begin{proof}
 The non-trivial direction was proved by Jensen.  A closer analysis shows that a $\kappa$-Suslin tree may be constructed assuming the diamond holds over a stationary subset of $\kappa$ that does not
reflect in the strong sense of being avoided by some weakly coherent $C$-sequence. For detailed proofs see  Brodsky and Rinot \cite[Theorem A]{paper32} for successor of regulars, \cite[Theorem B]{paper32} for successor of singulars and \cite[Corollary 4.27]{MR4228336} for inaccessibles, all three being instances the proxy principle $\mathrm{P}^\bullet(\kappa,\kappa,\sqsubseteq^*,1,\{\kappa\},\kappa)$ which implies the existence of a $\kappa$-Suslin tree, \cite[Corollary 6.7]{MR4228336}.
\end{proof}

In contrast to $\TP(\kappa)$, it is  open whether the characterization of weak compactness of $\kappa$ via $\SHg(\kappa)$ and inaccessibility holds in an arbitrary universe $V$. The principles $\SHg(\kappa)$ and $\TP(\kappa)$ appear to be conceptually different, with the former being more mathematical and the latter more logical (in the sense of a the classification we use). 

The principle $\SHg(\kappa)$ has often been analysed through a stronger principle called \emph{the special  Aronszajn tree property}, $\SATP(\kappa)$, which states that there are $\kappa$-Aronszajn trees and all $\kappa$-trees are special. Clearly $\SATP(\kappa)$ implies $\SHg(\kappa) + \neg \TP(\kappa)$. It is known that $\MA$ implies $\SATP(\omega_1)$, hence the traditional method of obtaining $\SHg(\omega_1)$ in fact ensures a stronger property. Laver and Shelah showed in \cite{LS:sh} that $\CH$ is consistent with $\SATP(\omega_2)$. Later,  Hayut and Golshani generalized the method of Laver and Shelah in \cite{MR4094555} and obtained a model where $\SATP(\kappa^+)$ holds for every regular cardinal $\kappa$. See also recent results of Adkisson \cite{adkisson2025treepropertiessuccessorssingulars} and Cummings et al.\ \cite{cummings2025treepropertylongintervals}.

\brm
Under $V = L$, there is a related characterization in terms of cofinal branches in $\kappa$-trees for a cardinal larger than a weakly compact, the \emph{ineffable cardinal}: if $V = L$, then $\kappa$ is ineffable if and only if there are no thin $\kappa$-Kurepa trees, see \cite[Chapter 7, Theorem 2.7]{DEVbook} for more details. The non-existence of (thin) Kurepa trees is also a compactness principle, as we already mentioned in connection with the weak Kurepa Hypothesis. Kurepa trees have many applications in model theory, see for instance Vaught \cite{MR210574}, Sinapova and Souldatos \cite{MR4159762} and Po\'{o}r and Shelah \cite{MR4292936}, but more details are beyond the scope of this article.
\erm

\subsection{Topological spaces} \label{sec:top}

Another consequence of $\Delta$-reflection is a notion of compactness for topological spaces, the \emph{collectionwise Hausdorff compactness}. In fact, this compactness principle, and another principle related to metrizability, are equivalent to the principle $\FRP$ we discussed in Section \ref{sec:FRP} (see Fuchino et al.\ \cite{more_frp}).

We refer the reader to Shelah \cite{MR540629}, Fleissner and Shelah \cite{FS:Haus}, and Laberge and Landver \cite{LL:Haus} for more details regarding collectionwise Hausdorff compactness. Let us just state that this concept characterizes weakly compact cardinals in $V = L$: 

\begin{theorem}\label{th:Haus}
If $V = L$, then the following are equivalent for an uncountable regular $\kappa$:
\bce[(i)]
\item $\kappa$ is weakly compact.
\item $\kappa$ is collectionwise Hausdorff compact.
\ece
\end{theorem}

\begin{proof}
It is easy to check that if $\kappa$ is weakly compact, then spaces of size $\kappa$ are collectionwise Hausdorff compact. For the converse direction, by Laberge and Landver \cite[Theorem 1]{LL:Haus}, $\square(\kappa)$ is sufficient for constructing a collectionwise incompact Hausdorff space of size $\kappa$. In fact, by a variant of this argument, $E(\kappa)$ is sufficient as well (see the discussion in \cite{LL:Haus} below Theorem 1).
\end{proof}

By Shelah \cite{MR540629}, collectionwise Hausdorff compactness is consistent on $\omega_2$ in contrast to abelian compactness. It is known that the countable color compactness (without the restriction on the size of the graphs) is also consistent at $\omega_2$, and in fact equivalent to $\FRP$ by Fuchino et al.\ \cite[Theorem 2.8]{more_frp} and Section \ref{sec:FRP} in this article.

\section{Standard models}\label{sec:canonical}\label{sec:col}

It is known that compactness at small cardinals has a large cardinal strength,\footnote{For instance the tree property $\TP(\omega_2)$ implies that $\omega_2$ is a weakly compact cardinal in $L$ and $\ISPii$ forced by a proper standard iteration implies that there must be a supercompact cardinal in the ground model, see Viale and Weiss \cite{VW:PFA}.} hence a natural method for obtaining compactness at small cardinals is to turn a large cardinal $\kappa$ into a small cardinal and argue that some compactness principles associated with large cardinals are preserved by the collapse. Depending on the collapsing forcing, various compactness principles will hold in the final generic extension. In particular, starting with a supercompact $\kappa$, there is a proper iteration with countable support guided by a Laver function naming all proper forcings which forces $\PFA$ or $\MM$. It is also possible to collapse infinitely many cardinals at the same time in order to have compactness at the successor of a singular cardinal.\footnote{This method can be used to obtain stationary reflection at $\aleph_{\omega+1}$ or the principle $\Delta_{\aleph_{\omega^2},\aleph_{\omega^2+1}}$ at $\aleph_{\omega^2+1}$. An important observation behind this construction is that compactness properties which hold at large cardinals often extend to the successors of their singular limits. This was first observed by Magidor and Shelah, see \cite{MAGSH:succ}, who showed that the tree property holds at the successor of the singular limit of infinitely many strongly compact cardinals, and this result extends to $\Delta$-reflection as well (Fontanella and Hayut \cite[Corollary 2.5]{FH:delta}, for the countable cofinality).}

We will briefly consider the most common collapsing forcings and direct readers to specific articles for more details.

The standard forcing iteration $\P_\kappa$ (where $\kappa$ is a large cardinal) defined by Viale and Weiss in \cite{VW:PFA}  is sufficiently representative for this survey:
\begin{compactenum}[(i)]
\item $\P_\kappa$ is a direct limit of an iteration $\seq{(\P_\alpha,\dot{\Q}_\alpha)}{\alpha \le \kappa, \alpha < \kappa}$ which takes direct limits stationarily often, and
\item $\P_\alpha$ has size $<\kappa$ for all $\alpha < \kappa$ (this for instance implies that $\P_\alpha$ preserves the largeness of $\kappa$).
\end{compactenum}

Every standard iteration is $\kappa$-cc and satisfies the $\kappa$-approximation property by \cite[Lemma 5.2]{VW:PFA}.\footnote{It is apparently  open  whether there are non-trivial standard iterations starting with a weakly compact $\kappa$ which, for instance, preserve $\omega_1$, turn $\kappa$ into $\omega_2$ and do \emph{not} force $\SRii$. Note that the results showing logical independence of several compactness principles in Cummings et al.\ \cite{8fold} are obtained by intentionally destroying the desired compactness principles by  forcings applied after  standard iterations.}

Let us mention some well-known examples for reference. Suppose $\GCH$ holds in the ground model and $\kappa$ is at least a weakly compact cardinal for the following examples:

\begin{itemize}

\item  Levy collapse $\Coll(\omega_1,<\kappa)$ with countable conditions forces for instance $\SR(\omega_2)$ and $\neg \KH$ (and Rado's Conjecture $\RC$ if $\kappa$ is strongly compact, see \cite{T:Rado}). It does not force $2^\omega > \omega_1$ and hence forces $\neg \TP(\omega_2)$, $\neg \AP(\omega_2)$, etc.

\item A countable support iteration of length $\kappa$ of the Sacks forcing at $\omega$ forces $\TP(\omega_2)$, $\SR(\omega_2)$ and other principles (see Kanamori \cite{KANAMORIperfect}). More generally, many iterations with countable support of forcings with some form of fusion which add new reals force $\TP(\omega_2)$, $\SR(\omega_2)$ and other principles. See Honzik and Verner \cite{HV:grig} and Stejskalov{\'a} \cite{S:g} for the Grigorieff forcing and Friedman et al.\ \cite{FHZ:fusion} for an abstractly defined class of forcings with fusion.

\item (Mixed support iteration). A forcing due to Mitchell (\cite{MIT:tree}) and its variants. See Abraham \cite{ABR:tree} for a modern presentation of the forcing and Krueger \cite{KR:DSS} for its description using the notion of a mixed support iteration.

\item  Canonical iterations with countable support $\P$ which force $\PFA$ (and other forcing axioms) ($\kappa$ is supercompact for these principles).

\end{itemize}

There are other ways of forcing compactness which are not of the ``standard'' form mentioned above. A well-known example is Neeman's forcing with side conditions introduced in \cite{N:SofModels}. Other methods are also employed when compactness principles are to be obtained globally, i.e.\ on successive regular cardinals. Sometimes the usual forcings can be combined using a product with a suitable support (typically if the consistency strength of the principles is below a weakly compact cardinal or there are larger gaps between the successive cardinals): for instance the tree property at every other regular cardinal below $\aleph_\omega$ (see for instance Friedman and Honzik \cite{FH:tree1}) or successive failures of the approachability property on every regular cardinal in the interval $[\omega_2,\aleph_\omega]$ in Unger \cite{U:ap}. For stronger principles, the situation starts to be more complicated because the principles tend to interact with each other if considered at successive cardinals. For instance the tree property at $\omega_2$ and $\omega_3$ requires a more complicated forcing and is consistency-wise much stronger than two weakly compact cardinals.\footnote{$\TP(\omega_2) + \TP(\omega_3)$ implies the consistency of a Woodin cardinal as was observed in Cummings et al.\ \cite{cummings2025treepropertylongintervals} improving an older lower bound due to Magidor reported in Abraham \cite{ABR:tree}. The construction in Abraham \cite{ABR:tree} uses two supercompact cardinals.} Obtaining the tree property at an infinite interval of regular cardinals is technically demanding, with the current record set in Cummings et al.\ \cite{cummings2025treepropertylongintervals} with the interval $[\omega_2,\aleph_{\omega^2+3}]$ with $\aleph_{\omega^2}$ strong limit.\footnote{It is still open whether the tree property can hold at every regular cardinal $\ge\omega_2$ or whether the tree property at $\aleph_{\omega+1}$ is compatible with the failure of $\SCH$ at a strong limit $\aleph_\omega$.}  It is natural to consider stronger forms of the tree property as well as was  done by Fontanella in \cite{Font:tree} for the principle $\ITP$ below $\aleph_\omega$. A similar questions for the stronger principles $\ISP_\kappa$ below $\aleph_\omega$ is open with the exception of the joint consistency of $\ISP_{\omega_2}$ and $\ISP_{\omega_3}$ which was shown  by Mohammadpour and Veli{\v c}kovi{\'c} in \cite{MR4290492} using a generalization of Neeman's forcing with side conditions. In some cases the task is complicated further by restrictions put on the continuum function: while the tree property on all regular cardinals in the interval $[\omega_2,\aleph_\omega]$ does not put new restrictions on the continuum function as shown  by Stejskalov{\'a} \cite{S:tp}, a similar configuration with the negation of the weak Kurepa Hypothesis necessitates $2^\omega \ge \aleph_{\omega+1}$. 
Finally, global results were considered also for stationary reflection below $\aleph_\omega$ by Jech and Shelah \cite{JS:full} or Chang's Conjecture by Eskew and Hayut \cite{MR3748588}. 

\brm
All the models  we mentioned above are obtained by collapsing a large cardinal (or cardinals) or by iterating along a large cardinal. In both cases, the final model is a limit of some increasing chain of models. In Rinot et al.\ \cite{paper69} it was demonstrated that obtaining compactness in a limit of a decreasing chain of models could be quite fruitful, in particular, with respect to the notions studied by Bagaria and Magidor in \cite{MR3226024}.
\erm

\section{Preservation of compactness} \label{sec:preserve}

On the one hand, by Solovay's theorem (see \cite{LS:m}), all large cardinals $\kappa$ mentioned in this article are preserved by all forcings of size $<\kappa$. On the other hand, all these cardinals are killed by adding $\kappa$-many Cohen subsets of $\omega$ because they destroy the inaccessibility of $\kappa$.  One might ask whether other large cardinal properties of $\kappa$, besides inaccessibility, can be destroyed with such ease. This question can be made more precise by asking whether the combinatorial cores, i.e.\ compactness principles related to these cardinals like the ones we discuss in this artile, can also be destroyed by simple forcings.

As it turns out, this is usually not the case, and in particular Cohen forcing $\Add(\omega,\kappa)$ will preserve $\TP(\kappa)$, $\SR(\kappa)$ and many other compactness principles at a weakly compact cardinal $\kappa$, making $\kappa$ a weakly inaccessible cardinal satisfying these compactness principles in the generic extension $V[\Add(\omega,\kappa)]$.

The investigation of the extent of preservation of compactness principles by forcing notions is an active area of set theory. The results of this line of research are useful for separating consequences of various compactness principles and also consequences of $\PFA$. Additionally, the preservation results indicate that compactness principles at accessible cardinals are stable notions and not accidental by-products of  specific iterations and that they can---perhaps--- be viewed as  viable candidates for  additional axioms.

However,  extensive preservation, or \emph{indestructibility}, can be interpreted also negatively as an indication that compactness principles do not have many consequences, at least not those whose truth can be changed by forcings which preserve the said principles (see Section \ref{sec:sep} for examples).

We will divide the preservation theorems into two groups: \emph{absolute} preservation theorems and \emph{model-specific} theorems. As the names suggest, absolute theorems assert that a forcing $\P$ preserves a given principle over any transitive model of a theory $T$ which extends $\ZFC$, while model-specific theorems apply only over specific forcing extensions satisfying  $T$. For showing independence of various statements, the model-specific preservation is sufficient. However, absolute preservation theorems give more insight into the relationship between a given compactness principle and the theory $T$.

\brm
Existing  preservation theorems are usually formulated for successors of regular cardinals and include all logical compactness principles, but only the weaker forms of compactness for stationary sets like $\SR$ or $\FRP$.  Principles like $\RC$ and $\WRP(\omega_2)$ are typically much easier to destroy: they both imply $2^\omega \le \omega_2$ hence cannot be preserved by Cohen forcing of length $>\omega_2$ (in fact $\RC$ is destroyed by adding a new real, see Section \ref{sec:uni} for more references). At successors of singulars, even logical principles are less understood: for instance, $\aleph_{\omega+1}$ is destroyed by $\Coll(\omega,\omega_1)$ as shown by Hayut and Magidor in \cite{HM:dest}.\footnote{It is still open whether $\TP(\aleph_{\omega+1})$ can be destroyed by a cofinality-preserving small forcing. By Rinot \cite{paper06} the answer is positive for special Aronszajn trees at $\aleph_{\omega_1+1}$: there is a cofinality-preserving small forcing which to a model without special Aronszajn trees at $\aleph_{\omega_1+1}$ adds a special $\aleph_{\omega_1+1}$-Aronszajn tree.}
\erm

\subsection{Absolute preservation} \label{sec:abs}

Before we discuss compactness principles, let us review preservation results for forcing axioms for comparison.  $\MM$ is preserved by all $\omega_2$-directed closed forcings by Larson \cite{MR1782117} and $\PFA$ is preserved by all $\omega_2$-closed forcing by K{\"o}nig and Yoshinobu \cite{200410101}. It follows that over models of $\PFA$, $\ISPii$ and all consequences of $\PFA$ are preserved by $\omega_2$-closed forcings. There is a weaker analogue of this preservation for the tree property at $\kappa^{++}$ by $\kappa^+$-closed and \emph{$\kappa^{++}$-liftable} forcings over the Mitchell model, see next Section \ref{sec:model}, item (\ref{lift}).

However, compactness principles are in addition preserved by forcings with small chain conditions:

\begin{enumerate}
\item \label{a:FRP} Preservation over $\ZFC$.  Fodor-type Refection Principles is preserved by all ccc forcings by Fuchino et al.\ \cite[Theorem 3.4]{MR2610450} (see Section \ref{sec:FRP} for more details on this principle). Also Chang's Conjecture is preserved by all ccc forcing notions. 
\item Preservation over $\ZFC$. Foreman showed in \cite{MR730583} that $\mu^{++}$-saturated ideals over $\mu^+$, $\mu$  regular, are preserved by $\mu^+$-centered forcing notions (in the sense that they generate saturated ideals in the extension).
\item Preservation over $\ZFC$. Gitik and Krueger showed in \cite{GK:a} that the negation of the approachability property at $\mu^{++}$, $\mu$ regular, is preserved by all $\mu^+$-centered forcings.
\item Preservation over $\ZFC$. Krueger essentially showed in \cite{KR:DSS} the preservation of the Disjoint Stationary Sequence property (which implies the negation of the approachability property):
Suppose $\kappa$ is an infinite cardinal and $\seq{s_\alpha}{\alpha \in S}$ is a disjoint stationary sequence on $\kappa^{++}$, with $S \sub \kappa^{++} \cap \cof(\kappa^+)$ stationary. Suppose $\P$ is a forcing notion which preserves $\kappa$, and moreover preserves stationary subsets of both $\kappa^+$ and $\kappa^{++}$. Then $\P$ forces that $\seq{s_\alpha}{\alpha \in S}$ is a disjoint stationary sequence on $\kappa^{++}$. 
\item Preservation over $\ZFC$. Honzik and Stejskalov{\'a} showed in \cite{HS:u} that stationary reflection at $\mu^+$, $\mu$ regular,  is preserved by all $\mu$-cc forcing notions. They further showed in \cite{HS:u} that if $\mu^{<\mu} = \mu$ then the club stationary reflection at $\mu^{++}$ is preserved by Cohen forcing at $\mu$ and Prikry forcing at $\mu$ (provided $\mu$ is measurable). This preservation result has been later extended to all $\mu^+$-linked forcings in Gilton and Stejskalov{\'a} \cite{TS:CSR} (being $\mu^+$-linked is slightly weaker than $\mu^+$-centered).
\item \label{a:1} Preservation over $T = \ZFC + \ISPgen{\mu^{++}} + \mu^{<\mu} = \mu$. Honzik at al. showed in \cite{HLS:gmp} that Cohen forcing at a regular cardinal $\mu$ preserves $\ISPgen{\mu^{++}}$ over all models of $T$.

Note that the same result is open both for $\ITPgen{\mu^{++}}$ and $\TP(\mu^{++})$ (in fact, preservation by single Cohen at $\mu$ is open).
\item \label{a:2} Preservation over $T = \ZFC + \nonwKH(\mu^{+}) + \mu^{<\mu} = \mu$. Lambie-Hanson and Stejskalov{\'a} recently observed that \cite{HLS:gmp} also establishes that Cohen forcing at a regular cardinal $\mu$ preserves $\nonwKH(\mu^+)$ over all models of $T$.

Note that the same result is open for $\nonKHgen(\mu^+)$ (in fact, preservation by single a Cohen real at $\mu$ is open).
\end{enumerate}

\brm
The difficulty of extending (\ref{a:1}) and (\ref{a:2}) to  preservation of $\ITPgen{\mu^{++}}$, $\TP(\mu^{++})$ and $\nonKHgen(\mu^+)$ lies in the fact that these principles refer to thin lists (or trees) while  $\ISPgen{\mu^{++}}$ and $\nonwKH(\mu^+)$ refer to slender lists--a difference which essential for the argument in \cite{HLS:gmp}. See the next section where we discuss that this obstacle can be partially overcome over specific models.
\erm

\brm \label{rm:overPFA}
Lambie-Hanson and Stejskalov{\'a} essentially showed in \cite[Corollary 5.8] {arithmetic_paper}  that arbitrarily long random forcing preserves $\nonwKH(\omega_1)$ over models of $\PFA$. They also observed that this result extends to $\TP(\omega_2)$ over models of $\PFA$.
\erm

\subsection{Model-specific preservation}\label{sec:model}

Suppose $\kappa^{<\kappa}= \kappa$ and $\lambda > \kappa$ is a large cardinal. Mitchell forcing $\M(\kappa,\lambda)$ (in its usual variants) collapses cardinals in the open interval $(\kappa^+,\lambda)$ and forces various compactness principles at $\lambda$, which is equal to  $\kappa^{++}$ in the generic extension. See more details about variants of the Mitchell forcing in Abraham \cite{ABR:tree} or in Honzik and Stejskalov{\'a} \cite{HS:ureg} (the article summarizes the key properties of Mitchell forcing needed for preservation theorems).

Mitchell forcing $\M(\kappa,\lambda)$ has the useful property that it can be written as a two stage iteration of the Cohen forcing $\Add(\kappa,\lambda)$ followed by a forcing $\dot{R}$ which is forced to be $\kappa^+$-distributive. Moreover, there is a projection from $\Add(\kappa,\lambda) \x \T$, where $\T$ is $\kappa^+$-closed. One can show various compactness principles in $V[\M(\kappa,\lambda)]$ using this product and the associated projections. Now, if $\dot{\Q}$ is a $\kappa^+$-cc forcing notion in $V[\Add(\kappa,\lambda)]$ which preserves $\kappa$, then it is possible to analyse the forcing extension $V[\M(\kappa,\lambda) * \dot{\Q}]$ using a variant of the  product  analysis by considering the product $(\Add(\kappa,\lambda) * \dot{\Q}) \x \T$.  This method can be for instance used to extend the result of Jensen and Schlechta in (\ref{m:1})  to Mitchell-style forcings. 

\begin{enumerate}
\item \label{m:1} Jensen and Schlechta showed in \cite{JSkurepa} that $\kappa^+$-cc forcings preserve the negation of the Kurepa Hypothesis $\neg \KHgen(\kappa^+)$ over generic extensions by Levy collapse $\Coll(\kappa^+,<\lambda)$, where $\lambda$ is a Mahlo cardinal. It is known that $\lambda$ being Mahlo is optimal for this result. 

\item \label{q1} Honzik and Stejskalov{\'a} showed in \cite{HS:ind} that the tree property $\TP(\kappa^{++})$ is preserved over a generic extension $V[\M(\kappa,\lambda)]$ by all $\kappa^+$-cc forcings existing in the intermediate forcing extension $V[\Add(\kappa,\lambda)]$: Suppose $\kappa = \kappa^{<\kappa}$ and $\lambda > \kappa$ is weakly compact. Suppose $\Add(\kappa,\lambda)$ forces that $\dot{\Q}$ is a forcing notion which is $\kappa^+$-cc and preserves $\kappa$. Then $\M(\kappa,\lambda) * \dot{\Q}$ forces $\TP(\kappa^{++})$.

Even though this is not written up, Lambie-Hanson and Stejskalov{\'a} observed that using 
\cite[Lemma 32]{ChS:guess}  this preservation result extends to $\ITP_{\kappa^{++}}$. Using a suitable definition of $\M(\kappa,\lambda)$ as in Cummings et al.\ \cite{8fold}, this result extends to  $\ISP_{\kappa^{++}}$ as well.

\item \label{lift} It is known that a supercompact cardinal $\kappa$ is preserved by all $\kappa$-directed closed forcing over a carefully prepared model by a result of Laver (\emph{Laver preparation}). We also mentioned in Section \ref{sec:abs} that $\PFA$ is preserved by all $\omega_2$-closed forcings over any model of $\PFA$. There is an analogue of these preservation theorems for $\TP(\kappa^{++})$ over a variant of the Mitchell model: Honzik and Stejskalov{\'a} defined in \cite[Definition 4.1]{HS:ind} a closure-type property called \emph{$\lambda$-liftability} (sandwiched between $\lambda$-closure and $\lambda$-directed closure)\footnote{This class for instance includes all forcings $\P$  such that any two compatible $p,q$ have the greatest lower bound and any decreasing sequence of length $<\lambda$ has the greatest lower bound (for instance the generalized Sacks forcing at $\lambda$ is not $\lambda$-directed closed, but it is $\lambda$-liftable (see \cite[Footnote 12]{HS:ind}). The definition of \emph{liftability} is similar to the notion of the \emph{complete} forcing notion (for $\lambda = \omega_1$) introduced by Shelah (see \cite[Chapter V]{SHELAHproper} and \cite[Remark 4.2]{HS:ind} for more context).} and showed that if $V[\R(\kappa,\lambda)]$ is a generic extension by Abraham-style Mitchell forcing $\R(\kappa,\lambda)$ for a supercompact $\lambda$ (see \cite[Definition 2.3]{HS:ind}), then $\TP(\lambda)$ is preserved by all $\kappa^+$-closed $\kappa^{++} = \lambda$-liftable forcings $\Q$ in $V[\R(\kappa,\lambda)]$, see \cite[Theorem 4.7]{HS:ind} for more details.

\item \label{q2} The result of Jensen and Schlechta from (\ref{m:1}) has a weaker analogue for the Mitchell model as proved by Honzik and Stejskalov{\'a} in \cite{HS:ureg} for the negation of the weak Kurepa Hypothesis:  Assume $\omega \le \kappa < \lambda$ are cardinals, $\kappa^{<\kappa} = \kappa$ and $\lambda$ is weakly compact.  Suppose $\Add(\kappa,\lambda) * \dot{\Q}$ is productively $\kappa^+$-cc (i.e.\ the product is $\kappa^+$-cc) and preserves $\kappa$. Then $\M(\kappa,\lambda) * \dot{\Q}$ forces $\neg \wKH(\kappa^+)$.
\end{enumerate}

Even if  $\dot{\Q}$ in items (\ref{q1}) and (\ref{q2}) is restricted to the Cohen submodel, there are some questions for which these limited preservation results can be useful. For instance, the generalized Baire space $\kappa^{\kappa}$ is not changed by the $\kappa^+$-distributive quotient forcing $\dot{R}$, and hence by utilizing a $\kappa^+$-cc forcing $\dot{\Q}$ which controls generalized cardinal invariants over $V[\Add(\kappa,\lambda)]$, it is possible to obtain the consistency of compactness principles at $\kappa^{++}$ together with cardinal invariants ensured by $\dot{\Q}$. See \cite{HS:u} and \cite{HS:ureg} for more details and Example 4 in the next section \ref{sec:sep}.

\subsection{Compactness vs.\ forcing axioms}\label{sec:vs}\label{sec:sep}

We have seen in Sections \ref{sec:con} and \ref{sec:FRP} that $\ISPii$ and $\FRP$ imply some of the global consequences of $\PFA$ like the failure of squares $\square(\lambda)$ for every uncountable regular $\lambda$. However,  preservation of these compactness principles at $\omega_2$ under various ccc forcing notions entails that many of the local consequences of $\MM$ are independent over $\ZFCc$. We will illustrate this phenomenon by giving some examples of well-known mathematical problems which are independent from $\ISPii +\FRP$ (and also from large cardinals by Solovay's theorem \cite{LS:m}): the Suslin Hypothesis\footnote{Suslin Hypothesis asserts that every dense linear order without end points which is complete and satisfies the ccc condition must be separable (and hence isomorphic to the reals). It is equivalent to the non-existence of an $\omega_1$-Suslin tree. $\SH$ follows from $\MA$ by Solovay and Tennenbaum \cite{ST:Suslin} and is falsified by $\Diamond_{\omega_1}$ (see Jensen \cite{JENfine}).} $\SH$, Whitehead's Conjecture\footnote{We say that an abelian group $A$ is Whitehead if every surjective homomorphism $f$ from any abelian group $B$ onto $A$ with kernel $\Z$ splits, i.e.\ there exists some homomorphism $f^*: A \to B$ such that $f \circ f^*$ is the identity on $A$. It is known that every free group is Whitehead. Whitehead asked whether the converse holds as well. Stein \cite{Stein} proved that all countable Whitehead groups are free. We write $\mathsf{WC}(\kappa)$ to assert that there exists a non-free Whitehead group of size $\kappa$ (a counterexample to all Whitehead groups being free). The question turned out to be independent from $\ZFC$. By Shelah \cite{Shelah:abelian}, $\MA$ implies $\mathsf{WC}(\kappa)$ for every regular uncountable $\kappa$ (see Eklof \cite[Section 8]{Eklof:W}), while $\Diamond_{\omega_1}(S)$ for every stationary $S$ implies $\neg \WP$ (in $V = L$, $\neg \mathsf{WC}(\kappa)$ for all regular uncountable $\kappa$). See  the Eklof's article \cite{Eklof:W} for a survey of Shelah's construction.} $\WP$, and Kaplansky's Conjecture\footnote{Kaplansky's Conjecture asserts that every algebra homomorphism from $C(X)$, where $X$ is any infinite compact Hausdorff space and $C(X)$ is the Banach algebra of continuous real valued functions, into any other commutative Banach algebra is continuous (``automatic continuity''). $\CH$ implies $\neg \KP$ and $\PFA$ implies $\KP$ (however, $\KP$ is equiconsistent with $\ZFC$ using a ccc partial order). See the book of Dales and Woodin \cite{DW:ind} for more details, Todor{\v c}evi{\'c} \cite [p.\ 87]{T:partition} for more historical details regarding $\PFA$,  and articles \cite{Woodin:KH-large,DUMAS_2024,Aoki:KH} for more context a recent development regarding the compatibility of $\neg \KP$ with large continuum. See also Remark \ref{rm:KC}.} $\KP$ (only independence from $\FRP$ follows by current results, see below for more details). In fact, for the examples we will discuss, preservation under the Cohen forcing at $\omega$ is enough because Cohen forcing often decides mathematical problems in the opposite way than forcing axioms. 

In Example 4, we will use a direct argument using  appropriate Mitchell models to show consistency with $\FRP + \ISPii$ and also with Rado's Conjecture. Since Rado's Conjecture is destroyed by adding a single real, a direct argument is an alternative to an indestructibility argument. Let us denote by $\BA$ the Baumgartner's Axiom.\footnote{A set $A \sub \R$ is called $\dense$ if it has no least and greatest elements and for all $a < b$ in $A$, $A \cap (a,b)$ has size $\omega_1$. $\BA$ is the statement that all $\dense$ sets are order-isomorphic, thus extending Cantor's theorem on the categoricity of the rationals (as a linear order). $\CH$ implies the failure of $\BA$ while $\PFA$ proves $\BA$ by Baumgartner \cite{Baum:PFA} (however, the consistency strength of $\BA$ is just that of $\ZFC$ using a ccc forcing notion \cite{MR317934}).} We will show that in an appropriate Mitchell extension, the following hold:

\beq \ISPii + \FRP + \neg \BA.\eeq

This implies that $\BA$ is independent from $\ZFCc$ because $\MM$ implies $\FRP + \ISPii + \BA$.

In another Mitchell extension, the following hold:

\beq \RC + 2^\omega = \omega_2 + \neg \BA + \neg \WP + \neg \SH.\eeq See Section \ref{sec:uni} where we discuss the theory $\RC + 2^\omega = \omega_2$ from a more general perspective.\footnote{We have chosen examples of well-known mathematical problems which fit our article, but there are many other which can be investigated in a similar way. For examples the problem of the automorphisms of the Calkin algebra, see Farah \cite{MR2776359}, or the existence of a five-element basis for uncountable linear orders, see Moore \cite{MR2199228}.}

\medskip

\emph{Example 1 (absolute preservation): Both $\SH$ and $\WP$ are independent from $\ZFCc$.} One direction follows from the fact that $\ISPii + \FRP$ and also $\SH$ and $\WP$ are consequences of $\MM$. The converse direction follows by considering a Cohen extension  over any model of $\MM$ and using the fact that both compactness principles are preserved by Cohen forcing (see Section \ref{sec:abs}, items (\ref{a:FRP}) and (\ref{a:1})):

\begin{itemize}
\item A single Cohen real is enough to falsify $\SH$ by Shelah's result that it adds an $\omega_1$-Suslin tree.
\item Any number of Cohen reals of uncountable cofinality falsifies $\WP$ by a theorem of Bergfalk et al.\ \cite{bergfalk2024whiteheadsproblemcondensedmathematics}:

\begin{theorem}[\cite{bergfalk2024whiteheadsproblemcondensedmathematics}]\label{th:bls}
Suppose $A$ is a non-free abelian group of size $\omega_1$. Then in $V[\Add(\omega,\omega_1)]$, $A$ is not Whitehead. Moreover, if $\P$ is a ccc forcing in $V$, then $A$ stays non-Whitehead in $V[\Add(\omega,\omega_1) \x \P]$.
\end{theorem}
\end{itemize}

One can observe that $\neg \SH$ and $\nonWP$ are consistent with an arbitrarily large $2^\omega$ of any uncountable cofinality, and by the preservation results for $\FRP$ and $\ISPii$, this can be partially extended to these principles as well:

The consistency of $\SH + \cf(2^\omega) > \omega_1$ can be shown using the standard ccc iteration which specializes all $\omega_1$-Aronszajn trees. Laver \cite{Laver:random} found an alternative argument and showed that a random forcing over a model of $\MA$ preserves $\SH$. Both forcings are ccc and hence by preservation results in Section \ref{sec:abs}, $\SH + \FRP$ is consistent with $2^\omega$ being any cardinal of uncountable cofinality. For $\ISPii$, the situation is more complicated because a general preservation theorem for ccc forcings is missing. A partial result follows from observations in Remark \ref{rm:overPFA}:  the random forcing applied over a model of $\PFA$ yields the consistency of $\SH + \nonwKH(\omega_1) + \TP(\omega_2)$ with any value of $2^\omega$ of uncountable cofinality.

For $\nonWP$, one can notice that for any $\kappa$ of uncountable cofinality, the forcing $\Add(\omega,\kappa)$ forces $\nonWP$: if there was in the generic extension a non-free Whitehead group $G$, then by taking a suitable automorphism of $\Add(\omega,\kappa)$, one can assume that some such group is added by the first $\omega_1$-many Cohen reals. By \cite[Theorem 6.3]{bergfalk2024whiteheadsproblemcondensedmathematics}, the next $\omega_1$-many Cohen reals make $G$ non-Whitehead, and the witness for being non-Whitehead survives through the rest of the Cohen forcing (since it is ccc). By the indestructibility results for $\FRP$ and $\ISPii$ in Section \ref{sec:abs}, it follows that if $\Add(\omega,\kappa)$ is applied over a model of $\FRP+\ISPii$, it forces $\nonWP + \FRP + \ISPii$ with $2^\omega=\kappa$ for any $\kappa \ge \omega_2$ of uncountable cofinality.

\medskip

\emph{Example 2 (absolute preservation): $\KP$ is independent from $\ZFC + \FRP + 2^\omega = \omega_2$.}

\begin{observation}\label{obs:KP}
$\KP$ is independent from $\ZFC + \FRP + 2^\omega = \omega_2$.
\end{observation}

\begin{proof}
$\KP$ is consistent with $\FRP$ because they are both consequences of $\MM$. For the other direction we use the fact that the Levy collapse of a strongly compact  cardinal to $\omega_2$ gives a model of $\FRP + \CH$, and hence also of $\nonKP$. By Woodin \cite{Woodin:KH-large},  Cohen forcing $\Add(\omega,\omega_2)$ forces $\neg \KP$ and by the indestructibility of $\FRP$ by all ccc forcings, this yields a model of $\neg \KP + \FRP + 2^\omega = \omega_2$.
\end{proof}

It seems to be open whether $\ZFC + \ISPii + \neg \KP$ is consistent, though.  A natural alternative to Levy collapse in Observation \ref{obs:KP}  is Mitchell forcing, but it is not even known whether  $V[\Add(\omega,\kappa)]$, where $\kappa$ is strongly compact and $\CH$ holds in $V$, satisfies $\neg \KP$.\footnote{Following Woodin's suggestion about morasses as a strengthening of $\CH$ in \cite{Woodin:KH-large}, Dumas \cite{DUMAS_2024} obtained $\nonKP$ with $2^\omega = \omega_3$ using Cohen forcing $\Add(\omega,\omega_3)$ over a ground model which in addition to $\CH$ contains a simplified morass. He suggests at the end of the paper that with higher morasses, $2^\omega$ can be equal to $\omega_n$ for any $n \ge 3$ with $\nonKP$.  The consistency of $\neg \KP$ with $2^{\omega} > \aleph_\omega$ seems to be open.}

\brm \label{rm:KC}
The argument for the consistency of $\KP$ in \cite{DW:ind} proceeds by constructing a generic extension via a ccc iteration which yields simultaneously $\MA$ and a combinatorial property which implies $\KP$. Todor{\v c}evi{\'c} noticed in \cite [Theorem 8.8]{T:partition} that this combinatorial property already follows from $\PFA$ (see \cite [p.\ 87]{T:partition} for more historical details on this point). It is open whether $\MA$ is necessary for $\KP$; see \cite{Aoki:KH} which constructs a model with $\nonKP$, $\neg \CH$ and a weak fragment of $\MA$.
\erm

\medskip

\emph{Example 3 (Mitchell model): cardinal invariants over $\ZFC + \SRii + \TP(\omega_2) + \neg \wKH(\omega_1) + \DSS(\omega_2)$.}

It is known that $\MA$ implies that most of the cardinal invariants have the maximal value $\omega_2$, in particular the tower number $\mathfrak{t}$ (which is provably below many of the other cardinal invariants) and all cardinal characteristics of the meager and null ideal. It follows that $\ISPii + \FRP$ are consistent with cardinal invariants being equal to $2^\omega = \omega_2$. It begs the question whether they are also consistent with other values of cardinal invariants.

By using preservation theorems for the principles $\SRii$, $\DSS(\omega_2)$,\footnote{Disjoint Stationary Sequence property, see Krueger \cite{KR:DSS} for more details.} $\TP(\omega_2)$ and $\nonwKH(\omega_1)$, Honzik and Stejskalov{\'a} showed in \cite{HS:ureg} that these principles are consistent with diametrically different patterns of cardinal invariants. In particular, they are consistent (for example) with 
\begin{equation} 
\omega_1 = \mathfrak{t} < \mathfrak{u}   < 2^\omega,
\end{equation}

where $\mathfrak{u}$ is the ultrafilter number. It is highly plausible that with more work, one can show this result also for $\ISPii$ because $\nonwKH(\omega_1)$ is as regards its behaviour and consequences close to $\ISPii$.

The proof starts by defining a ccc forcing $\Q$ which controls cardinal invariants in the generic extension $V[\Add(\omega,\kappa) * \dot{\Q}]$.  Then, preservation results for the Mitchell model and a projection analysis based on the forcing equivalence between $\M(\omega,\kappa) *\dot{\Q}$ and $(\Add(\omega,\kappa) * (\dot{\Q} \x \dot{R}))$ entails that the desired compactness principles hold in $V[\M(\omega,\kappa) * \dot{\Q}]$. The pattern of cardinal invariants in the extension $V[\M(\omega,\kappa) * \dot{\Q}]$ is computed using the fact that the quotient $\dot{R}$ does change the Baire space $\omega^\omega$.

\medskip

\emph{Example 4 (A direct argument in the Mitchell model): $\BA$ is independent from $\ZFCc$. Moreover, $\neg \BA$, $\neg \WP$ and $\neg \SH$ are consistent  with $\ZFCcR$.}

\brm $\RC$ is destroyed by adding any new real,\footnote{See Footnote \ref{ft:real} for details.} hence an indestructibility argument cannot be used for its consistency. $\ZFCcR$ proves $\TPii$ (see \cite{MR3600760}) and $\FRP$ (see \cite{rc_frp}). $\ZFCcR$ is an interesting theory which we discuss in some detail in Section \ref{sec:uni}. We do not know whether $\ZFCcR$ is consistent with $\SH$, $\BA$ or $\WP$. The problem is that $\RC$ contradicts forcing axioms, and not many forcings are known which force $\RC$.\footnote{It is natural to try to use other forcings which force various compactness principles, such as the tree properties, and check whether they force $\RC$ as well. To our knowledge, this has not been investigated yet (at least not in a published form).} \erm

By a result of Sierpi{\'n}ski \cite{MR41909}, $\CH$ implies a strong failure of $\BA$: there are $\dense$ sets $A, B$ such that there is no order-preserving embedding from $A$ into $B$ or conversely.\footnote{In fact, one can modify the argument we give in Lemma \ref{lm:Serp} by using a recursive construction indexed by the tree $2^{<\omega_1}$ and show that there is a family of $2^{\omega_1}$-many of $\dense$ sets of reals which are pairwise incomparable under the order-preserving embedding.} We will give a proof sketch of a this claim to make the argument self-contained (and because there seems to be no easily accessible published proof).

\begin{lemma}[Sierpi{\'n}ski]\label{lm:Serp}
Suppose $\CH$ holds, then there are $\dense$ sets $A^0,A^1$ wich are incomparable under the order-preserving embeddings, i.e., $A^0$ cannot be embedded in the order-preserving way into $A^1$, or conversely.
\end{lemma}

\begin{proof}
Since $\R$ is separable,  order-preserving embeddings between $\dense$ sets are determined by countable dense subsets. By $\CH$, it is possible to enumerate all countable order-preserving embeddings with dense domains on the reals as $\seq{f_\alpha}{\alpha<\omega_1}$. 

We construct $A^0$ and $A^1$ recursively in $\omega_1$ steps by a back-and-forth argument, diagonalizing over $\seq{f_\alpha}{\alpha<\omega_1}$. We define on each side of the back-and-forth argument two sequences: sequences of countable sets $\seq{A^i_\alpha}{\alpha<\omega_1}$ increasing under inclusion, with unions $A^0$ and $A^1$, respectively, and sequences of reals $\seq{y^i_\alpha}{\alpha <\omega_1}$ (these will be forbidden from being in $A^i$), for $i < 2$. 

Let $A_0^0,A_0^1$ be arbitrary countable dense subsets of $\R$. Suppose the sequences are constructed for all $\beta < \alpha$ and all $A^0_\beta$ and $A^1_\beta$ are countable and dense. Let $\bar{A}^0_\alpha$ denote the union $\bigcup_{\beta<\alpha}A^0_\beta$, and similarly for $\bar{A}^1_\alpha$. Consider the function $f_\alpha$. 

In the ``back'' direction, we look at $f_\alpha$ as a function from $A^0$ to $A^1$. We can assume  $\bar{A}^0_\alpha \sub \dom{f_\alpha}$, otherwise $f_\alpha$ cannot be an embedding from $A^0$ to $A^1$. Since $\bar{A}_\alpha^{0}$ is dense, its completion $\bar{A}_\alpha^{0,c}$ has size $2^\omega$. Let $f^c_\alpha$ denote the unique extension of $f_\alpha$ to  $\bar{A}_\alpha^{0,c}$. Since $\bar{A}^1_\alpha$ is countable, there is some $x^0_\alpha \not \in \set{y^0_\beta}{\beta < \alpha}$ in $\bar{A}_\alpha^{0,c}$ such that $f^c_\alpha(x^0_\alpha) \not \in \bar{A}^1_\alpha$.  Add $x^0_\alpha$ to $\bar{A}^0_\alpha$ and set $y^1_\alpha = f^c_\alpha(x^0_\alpha)$. If $y^1_\alpha$ is not added in any further stage of the construction to $A^1$, $f_\alpha$ cannot be an embedding from $A^0$ into $A^1$.

In the ``forth'' direction, we look at $g_\alpha = f^{-1}_\alpha$. We can  assume  $\bar{A}^1_\alpha \sub \dom{g_\alpha}$ and proceed as in the previous case, defining $x^1_\alpha$ and $y^0_\alpha$.

Finally, for every pair of reals in $\bar{A}_\alpha^i$, $i<2$, add one new real between them, avoiding the sets $\set{y^i_\beta}{\beta \le \alpha}$, $i<2$ (to make the resulting sets eventually $\dense$), and denote the resulting sets $A^i_\alpha$, $i<2$.

Set $A^i = \bigcup_{\beta<\omega_1}A^i_\beta$. It is easy to see that there cannot be any order-preserving embedding between them.
\end{proof}

Baumgartner mentions in \cite{MR317934} without a proof that uncountably many reals can be added while preserving $\neg \BA$ over a model of $\CH$. An explicit proof, which implies that a strong failure of $\BA$ holds in Cohen and Random extensions adding any number of new reals, is stated in Switzer \cite[Theorem 4.1]{MR4939512} (the proof shows that a certain principle $\mathsf{U}_\kappa$  fails in Cohen and Random extensions, but it can be easily adapted to Lemma \ref{lm:B}).

\begin{lemma}\label{lm:B}
Suppose $\CH$ holds. Then in the Cohen extension $V[\Add(\omega,\kappa)]$ for adding $\kappa$-many new reals, there are $\dense$ sets $A^0,A^1$ wich are incomparable under the order-preserving embeddings.
\end{lemma}

Now we can prove the consistency of $\neg \BA$, actually a strong version of the negation from Lemma \ref{lm:Serp}, with $\ZFCc$:

\begin{observation}\label{obs:BA}
Suppose $\kappa$ is a supercompact cardinal, $\CH$ holds. Then in the standard Mitchell model $V[\M(\omega,\kappa)]$, $\ISPii + \FRP$ and $\neg \BA$ hold.
\end{observation}

\begin{proof}
Let $\M(\omega,\kappa)$ be a forcing notion as in \cite{8fold} which  forces $\ISPii$ and $\FRP$ (this is a standard argument). Let us show that the strong form of the negation of $\BA$ from Lemma \ref{lm:Serp} holds in this model. $\M(\omega,\kappa)$  is equivalent to a forcing $\Add(\omega,\kappa) * \dot{R}$, where $\Add(\omega,\kappa)$ forces that $\dot{R}$ is $\sigma$-distributive. By Lemma \ref{lm:B}, there are pairwise incomparable $\dense$ sets $A^0, A^1$  in $V[\Add(\omega,\kappa)]$. Since $\dot{R}$ is forced to be $\sigma$-distributive, and hence does not add new countable embeddings, they remain incomparable in $V[\M(\omega,\kappa)]$.
\end{proof}

Since $\MM$ proves $\ZFCc + \BA$, $\BA$ is neither proved or refuted from $\ZFCc$.

\brm
Baumgartner's Axiom can be generalized to axioms which deal with $\kappa$-dense subsets of topological spaces $X$, denoted $\mathsf{BA}_\kappa(X)$, with order-preserving embeddings replaced by continuous embeddings.  Thus, in this notation, $\mathsf{BA}_{\omega_1}(\R)$ is equivalent to $\BA$, where $\R$ is endowed with the standard topology. Switzer \cite{MR4939512} considers two natural weakenings of $\mathsf{BA}_\kappa(X)$, denoted $\mathsf{BA}^{-}_\kappa(X)$ and $\mathsf{U}_\kappa(X)$, for arbitrary Polish spaces $X$, and shows that the (strict) weakening $\mathsf{BA}^{-}_\kappa(X)$ retains some of the strong consequences of $\BA$ such as $2^\omega = 2^{\omega_1}$, while $\mathsf{U}_\kappa$ does not. In \cite[Theorem 4.1]{MR4939512} he proves that $\mathsf{U}_\kappa$  fails in Cohen and Random extensions (we mentioned this result already for Lemma \ref{lm:B}). $\BA$ is interesting also from the point of cardinal invariants of the continuum: Todorcevi{\'c} showed in \cite{T:partition} that $\BA$ implies $\mathfrak{b} > \omega_1$ (it is still open whether it implies $\mathfrak{p}>\omega_1$). Since $\mathfrak{b}=\omega_1$ in the Mitchell model, this gives an alternative argument for Observation \ref{obs:BA}.
\erm

Let us now show that $\RC + 2^\omega = \omega_2$ does not prove $\BA,\WP$, or $\SH$.

\begin{theorem}\label{obs:BWS}
Suppose $\CH$ hold, there is an $\omega_1$-Suslin tree, $\kappa$ is strongly compact, and $\M(\omega,\kappa)$ is a Mitchell forcing with sparse collapses\footnote{It is sufficient that collapses, i.e., the conditions on the second coordinate of the Mitchell forcing, are defined only at coordinates with cofinality $\ge \omega_2$. See Cummings et al.\ \cite{8fold} for more details regarding variants of the Mitchell forcing.}. It forces $\ZFCcR$, together with $\neg \BA, \neg \SH$, and $\neg \WP$.
\end{theorem}

\begin{proof}
Let $\M := \M(\omega,\kappa)$ be the Mitchell forcing as in Zhang \cite{MR4094551} who showed that it forces $\ZFCcR$. Let us also denote the restriction of the Cohen forcing $\Add(\omega,\kappa)$ to an interval $I$ on $\kappa$ by $\Add_I$ and the truncation of $\M$ to stage $\alpha$ by $\M_\alpha$.

\emph{$\neg \BA$ holds in this extension.} This is exactly as the proof of Observation \ref{obs:BA}.

\emph{$\neg \SH$ holds in this extension.}  We know there is a projection from $\M$ to a product $\Add(\omega,\kappa) \x \T$ where $\T$ is $\sigma$-closed. Let $T$ be an $\omega_1$-Suslin tree in $V$. It is easy to check that due to its $\sigma$-closure, $\T$ does not add uncountable antichains to $T$, and nor does $\Add(\omega,\kappa)$ over $V[\T]$ because it is $\omega_1$-Knaster there. This implies that $T$ is Suslin in $V[\Add(\omega,\kappa) \x \T]$, and hence also in $V[\M]$ which is its submodel.

\emph{$\neg \WP$ holds in this extension.} Suppose $A$ is a non-free abelian group of size $\omega_1$ in $V[\M]$. We wish to show that $A$ is non-Whitehead (we will write ``non-W'') in $V[\M]$.

Due to $\M$ being $\kappa$-cc, the group $A$ appears at some stage $V[\M_\alpha]$, $\alpha < \kappa$. Let us work in $V[\M_\alpha]$. $A$ is non-free in $V[\M_\alpha]$ due to the downward preservation of this property. By Theorem \ref{th:bls}, the tail of the Cohen forcing $\Add_{[\alpha,\kappa)}$ makes $A$ non-W. It follows that $A$ is a non-W group in $V[\M_\alpha][\Add_{[\alpha,\kappa)}].$ Let us fix  in $V[\M_\alpha][\Add_{[\alpha,\kappa)}]$ a homomorphism \begin{equation}\label{eq:split} f: B \to A \end{equation} which does not split. Both $f$ and $B$ have size $\omega_1$, so we can assume by permuting the generic for $\Add_{[\alpha,\kappa)}$ if necessary that $f,B$ are added by $\Add_{[\alpha,\alpha+\omega_1)}$ over $V[\M_\alpha]$. Let us denote $\alpha +\omega_1$ by $\beta$. Since the Mitchell forcing is sparse, there are no collapses in the interval $[\alpha,\beta)$, hence also the model $V[\M_\beta]$ contains $f: B \to A$ which does not split in $V[\M_\beta]$, i.e.\ there is no homomorphism $f^*: A \to B$ in $V[\M_\beta]$ such that $f \circ f^*$ is the identity on $A$. Let us work over $V[\M_\beta]$, and let $\T$ be  the $\omega_1$-closed term forcing such that $\T \x \Add_{[\beta,\kappa)}$ projects onto the tail of the Mitchell forcing. We will argue that $\T \x \Add_{[\beta,\kappa)}$ does not add a splitting homomorphism to $f$. It follows that there cannot be a splitting homomorphism in $V[\M]$, which finishes the proof.

This is shown using the standard method of working in $V[\M_\beta]$ and diagonalizing over antichains in $\Add_{[\beta,\kappa)}$ and building a decreasing sequence of conditions in $\T$ by recursion on $\omega_1$.  In some detail, suppose for contradiction that $\dot{f^*}$ is forced by $\Add_{[\beta,\kappa)} \x \T$ over $V[\M_\beta]$ to be a splitting homomorphism for $f:B \to A$ which is not in $V[\M_\beta][\Add_{[\beta,\kappa)}]$, and let $\seq{a_\alpha}{\alpha <\omega_1}$ be some enumeration of $A$. Build a decreasing sequence in $\T$, $\seq{t_\alpha}{\alpha < \omega_1}$, and a sequence of antichains $\seq{X_\alpha}{\alpha<\omega_1}$ in $\Add_{[\beta,\kappa)}$ such that if $G$ is any generic for $\Add_{[\beta,\kappa)}$, then in $V[\M_\beta][G]$, there is for each $X_\alpha$ exactly one condition $p_\alpha \in X_\alpha \cap G$ and  $(p_\alpha,t_\alpha)$ determines the value of $\dot{f^*}(a_\alpha) := b_\alpha$ in $B$. It is easy to see that the function in $V[\M_\beta][G]$ which maps $a_\alpha$ to $b_\alpha$ is a splitting homomorphism for $f$ in $V[\M_\beta][G]$, which is a contradiction with (\ref{eq:split}).
\end{proof}

\medskip
\begin{center} *** \end{center}
\medskip

One may consider the principles $\SH, \BA, \WP$  on higher cardinals and ask whether compactness principles start to play some role. 

Suslin Hypothesis at $\omega_2$, $\SHii$,  becomes a genuine compactness principles which is implied by $\TP(\omega_2)$ (and hence also by $\PFA$) in the context of $\neg \CH$. But it can also be considered as a stand-alone principle with $\CH$. However, $\SHii$ (with or without $\CH$) is no longer equivalent to a characterization of a well-known mathematical object (like the reals for $\SH$), and hence its general appeal is smaller.  Still,  $\SHg(\kappa)$ for a regular $\kappa$ is an interesting principle from the set-theoretic perspective  (see Section \ref{sec:Suslin} for some more details on Suslin Hypothesis).

Whitehead's Conjecture can be formulated for abelian groups of size $\omega_2$ (and bigger), but this will not make the problem related to large cardinals or compactness because $\MA$ entails $\WPg(\kappa)$ and $V = L$ entails $\neg \WPg(\kappa)$ for every regular uncountable $\kappa$. It makes sense to generalize the question more extensively, with focus on richer algebraical structures like modules, or with heavier use of the homology context. See for instance a comprehensive two-volume monograph by G{\"o}bel and Trlifaj \cite{MR2985554, MR2985654} which among other things investigates generalizations of the Whitehead's Conjecture to modules using advanced set-theoretic concepts like Shelah's Uniformization Principle (see in particular Chapter 11 of \cite{MR2985554}). However, no large cardinal principles are mentioned in these results. There are also works of Bergfalk, Lambie-Hanson and Hru{\v s}{\'a}k \cite{MR4275058, MR4568267} on simultaneous vanishing of higher derived limits in the homology algebra. It is of some interest that the first article \cite{MR4275058} proved the required result using a large cardinal hypothesis while the second article \cite{MR4568267} reproved the result just from the consistency of $\ZFC$. 

Whitehead's Conjecture has been generalized in yet another direction which we surveyed in some details in Sections \ref{sec:ab} and \ref{sec:delta}. Since every countable Whitehead group is free and all subgroups of a Whitehead group are Whitehead, all Whitehead groups of size $\omega_1$ are \emph{almost-free}, i.e., all subgroups of strictly smaller sizes are free. This leads to a notion of compactness which postulates that almost-free groups should be free. As we saw, this concept necessarily involves large cardinals, and this form of compactness is provable false below $\aleph_{\omega^2}$.

The principle $\BA$ can be generalized to $\BAii$, referring now to $\omega_2$-dense subsets of the reals (defined in the obvious way) in the context of $2^\omega > \omega_2$.\footnote{The principle $\BAg(2^\omega)$ is always false, by an argument as in Lemma \ref{lm:Serp}.} The statement of $\BAii$ retains the original appeal of $\BA$, with $\BA + \BAii$ entailing a notion of categoricity for $\omega_1$ and $\omega_2$-dense suborders of the reals. The consistency of this principle has not been settled yet, but Moore and Todor{\v c}evi{\'c} showed in \cite{MR3696077} that $\BA + \BAii + \textsf{MA}(\omega_2)$ is consistent modulo a combinatorial principle for $\omega_2$ denoted (**). The consistency of (**) has been open since then, though.\footnote{Despite an early optimism that (**) may be consistent from large cardinals, there has been no real progress so far (see a brief remark in Guzman and Todor{\v c}evi{\'c} \cite{MR4819970}).}

\section{Unifications} \label{sec:eval}\label{sec:uni}

In light of the different nature of logical and mathematical compactness principles, as regards their compatibility or incompatibility with non-reflecting stationary sets and with instances of $\GCH$, it is natural to look for unifying principles which would have both types of principles as their consequences. A convincing unification of compactness principles would be a good starting point for proposing compactness principles as new axioms, along the lines of Foreman \cite{MR1648052} and  Feferman et al.\ \cite{4}.\footnote{For comparison, note that the adoption of choice principles was facilitated by the existence of a uniform unification provided by $\AC$. This unification was moreover unique: many choice principles considered initially were soon proved to be equivalent ($\AC$, the Well-ordering Principle, Zorn's lemma, etc.). See Footnote \ref{ft:WO} for more comments on this point.}

We discuss such unifications from two perspectives. 

First, in Section \ref{sec:expl}, we discuss unifications from the set-theoretic perspective which looks for general reflection-type principles which imply many of the compactness principles (and forcing axioms as well) in a uniform way. We will focus on \emph{Laver generic large cardinal axioms, LgLCAs}, which generalize and extend the K{\"o}nig's \emph{Game Reflection Principle}, $\GRP^+$, introduced in \cite{MR2052885}, which is equivalent to generic supercompactness of $\omega_2$ for $\sigma$-closed forcings. They require a deeper understanding of set-theoretic concepts, and thus may not be immediately appealing to all mathematicians. For some, however, they may provide a structural and uniform explicatory reason for the naturalness of purely mathematical compactness principles, such as those we discussed in this article.\footnote{On a more philosophical note, it is a matter of subjective preferences to decide whether LgLCAs make the mathematical compactness principles which they imply more ``natural'' on account of being consequences of LgLCAs, or, rather, LgLCAs are seen as ``natural'' precisely because they have natural mathematical consequences.}

Then, in Section \ref{sec:prac}, we discuss unifications from the perspective of general mathematics, formulated in terms of concepts which do not require a deeper understanding of set theory. We will focus on two principles which are the strongest of those discussed in this article, Rado's Conjecture and Martin's Maximum, and yet incompatible with each other.

\subsection{A set-theoretic perspective}\label{sec:expl}

Generic elementary embeddings can be used to formulate various reflection-type principles which--depending on the parameter for a class of posets $\mathcal P$---have wide-ranging (but sometimes incompatible) consequences. Let us state a general form of the definition, following Fuchino and Rodrigues \cite{ReflPrinciples}:

\begin{definition}[]\label{def:gen}
Let $\mathcal P$ be a class of forcing notions and let (*) be a variable for a large cardinal property, such as supercompact, superhuge, etc. We say that $\kappa$ is \emph{Laver-generically (*) for $\mathcal P$} if for any $\lambda \ge \kappa$ and any $\P \in \mathcal P$, there is a $\Q \in \mathcal P$ such that $\P$ is regularly embeddable into $\Q$ and for any generic $V$-generic filter $H$ for $\Q$ there are $M, j \sub V[H]$ such that 
\bce[(i)]
\item $M$ is an inner model of $V[H]$,
\item $j: V \to M$ is an elementary embedding definable in $V[H]$, with critical point $\kappa$ and $j(\kappa)>\lambda$,
\item $\P, H \in M$, and
\item $M$ is closed under sequences, as prescribed by (*).\footnote{For example, if $\kappa$ generically supercompact, then we require $j"\lambda \in M$, if $\kappa$ is generically superhuge, we require $j"j(\kappa) \in M$, etc.}
\ece
We write \emph{LgLCAs} \emph{(Laver-generic Large Cardinal Axioms)} to denote statements claiming the existence of a Laver-generically large cardinal for some $\mathcal P$.
\end{definition}

As it turns out, small regular cardinals such as $\omega_2$ can be Laver-generically large for various classes of forcing notions $\mathcal P$. Let us state several compactness-type consequences of LgLCAs for $\omega_2$.

\begin{itemize}

\item K{\"o}nig's \emph{Game Reflection Principle}, $\GRP^+$, is equivalent to $\omega_2$ being Laver-generically supercompact for $\sigma$-closed forcings (see K{\"o}nig \cite[Theorem 17]{MR2052885}). By \cite[Proposition 22]{MR2052885}, $\GRP^+$ implies $\RC$. By Fuchino et al.\ \cite[Lemma 4.2]{MR4198354}, it implies $\CH$.\footnote{Which is not desirable from the perspective of unification as we discussed in Section \ref{sec:prac}.}

\item If $\omega_2$ is Laver-generically supercompact for stationary preserving forcings, then $\MM^{++}$ holds (see \cite[Theorem 5.7]{MR4240755}), and hence $2^\omega = \omega_2$.

\item A strengthening of Definition \ref{def:gen}, \emph{Super-$C^{\infty}$-LgLCAs}, was introduced by Fuchino and Usuba in \cite{MR4927442} (see also \cite{F:latest}). It is used in the definition of an ultimate generic large cardinal principle, dubbed the \emph{Laver Generic Maximum} in \cite{MR4927442}. It implies $\MM^{++}$ and many other reflection-type principles (see the list in \cite[Section 7]{MR4927442}).

\end{itemize}

From the set-theoretic perspective, the majority of compactness principles at $\omega_2 = 2^\omega$ discussed in this article can thus be viewed trough the lenses of Laver-generic largeness as a consequence of $\omega_2$ being a genuine large cardinal of the given type in a definable submodel of $V$. This provides a uniform explication for the compactness properties true at $\omega_2$, at least for principles derivable from generic largeness, and captures explicitly the fact that compactness at small cardinals is often ensured by collapsing a large cardinal (see Section \ref{sec:canonical} on standard models). However, there are also some limitations and additional considerations:

\begin{itemize}

\item There is one notable exception to the uniform derivability of compactness principles from Laver-generic large cardinals:  LgLCAs do not imply $\RC$ together with $2^\omega = \omega_2$ (essentially because $\RC$ contradicts $\MA$). This sets the theory $\RC + 2^\omega = \omega_2$ appart from not only forcing axioms, but also from generic largeness (we discuss this theory in the next section).

\item LgLCAs depend on the parameter $\mathcal P$. This makes the choice of a specific axiom LgLCA rather non-canonical, especially because the variations of $\mathcal P$ yield incompatible consequences. We saw above that $\mathcal P$ for $\sigma$-closed forcings yields $\CH$ while $\mathcal P$ for stationary preserving or proper forcings yields $2^\omega = \omega_2$. With some other classes of $\mathcal P$, the continuum can be arbitrarily large (see for instance \cite{ReflPrinciples} for more details). It is not a priori clear which $\mathcal P$ is the ``right one'' (if there is one): it is possible to assign intuitive plausibility to  LgLCAs with different $\mathcal P$'s based on their consequences, but it may defeat the purpose of having a uniform explicatory principle in the first place.

\item On the positive side, though, the non-canonicity of LgLCAs can be viewed---because of the parameter $\mathcal P$---as conceptually useful generalizations of forcing axioms which can be formulated for bigger cardinals than $\omega_2$.
\end{itemize} 

The perceived drawback of non-canonicity mentioned above corresponds to the narrow perspective in this section which focuses on unifying compactness principles. From the more general perspective, LgLCAs can be interpreted as providing a framework which solves the Continuum Hypothesis (among other things) along the lines of the set-theoretic multiverse (see for instance \cite{MR2970696}, \cite{MR2366963}, \cite{MR3728994}, \cite{MR3400617}): LgLCAs  imply in a well-defined sense that continuum is either $\omega_1$, $\omega_2$ or a weakly inaccessible cardinal (on the level of weakly Mahlo cardinals), which is a fascinating trichotomy (see Fuchino and Rodrigues \cite[Section 6]{ReflPrinciples} and Eskew \cite{MR4092254} for more details and references for generic large cardinals and their potential for becoming recognized axioms).

\subsection{A mathematical perspective}\label{sec:prac}

From the narrower perspective, the adoption of axioms formulated in terms of combinatorial concepts not specific to set theory or logic appears to be easier. For example, it is well-known that the adaption of the Axiom of Choice was accelerated by existence of combinatorial restatements such as Zorn's Lemma which avoid the set-theoretic notions of well-orders and arbitrary choice functions, and can be directly applied to mathematical structures.\footnote{\label{ft:WO} The notion of a well-ordered set (a partial ordered in which all non-empty subsets have the least element) was first considered by Cantor for the purpose of defining  infinite cardinals. Zermelo showed in 1904 that the statement that every set can be well-ordered is equivalent to the fact that there is a choice function on every set. Zorn's lemma (formulated by Zorn in 1935 \cite{MR1563165}) postulates the existence of maximal elements in  partial orders $P$ in which all chains have upper bounds---a familiar concept, not requiring deeper knowledge of set theory. See Moore \cite{Moore:AC} for more historical details regarding the adoption of $\AC$.} To take a more modern example, forcing axioms can be stated as purely combinatorial statements, without mentioning consistency of theories, transitive models of set theory and other logical concepts which  appear naturally in set-theoretic arguments. By historical analogies with $\AC$ mentioned on the previous lines, this might make them more palatable (and useful) for a general mathematical community.

A well-known axiom, formulated in combinatorial terms, which unifies almost all compactness principles for $2^\omega = \omega_2$ discussed in this article (with the exception of $\RC$), is Martin's Maximum $\MM$ in various variants (such as $\MM^{++}$): since it applies not only to proper forcings but also to stationary preserving forcings, it implies various forms of stationary reflection such as $\RP$ or $\FRP$, along with $\ISPii$ (which is already implied by $\PFA$). It is in a well-defined sense the strongest possible forcing axiom\footnote{See a survey by Viale \cite{MR4713473}, in particular Proposition 4.2 and Theorem 4.4 there.} with numerous consequences in mathematics (see Section \ref{sec:sep} for some examples).  It implies $\AD^{L(\R)}$ and is therefore appealing also from the perspective of the Axiom of Determinacy (see Maddy's section in \cite{4} for more details, and also the articles \cite{Maddy:bI, Maddy:bII}).

However,  $\MM$ stubbornly resists generalizations to larger cardinals, hence it is natural to look for alternatives:

\begin{itemize}

\item Rado's Conjecture is a natural mathematical statement which is incompatible with $\MA$. Though it is formulated as a compactness principle for a certain class of graphs (and hence its scope looks a priori rather limited), it has a surprisingly wide range of consequences. Moreover, unlike $\MM$, it is easier to generalize to higher cardinals since it is formulated only in terms of subgraphs and cardinalities.

\item The deeper reason why generalizations of $\MM$ to higher cardinals appear to be hard to find is that the structure of stationary subsets of $[\kappa]^\theta$ for an uncountable $\theta$ is much more complex, in comparison with $[\kappa]^\omega$. This lack of uniformity suggests that forcing axioms for $\omega_2$, related as they are to $[\kappa]^\omega$, may be an exception, not a rule. This makes $\MM$ an isolated principle rather than an instance of a more general structure (see also Remark \ref{rm:MM}), and it may be seen as lowering its explicatory strength.
\end{itemize}

As we mentioned in the previous Section, both $\RC$ and $\MM$ are consequences of Laver-generic large cardinals for specific classes $\mathcal P$. However, we also noted that $\GRP^+$, which implies $\RC$, also implies $\CH$, and hence the negation of the logical compactness principles at $\omega_2$ such as the tree property or the failure of the approachability. 

Hence for the purposes of this section---a discussion of unifications of mathematical and logical compactness---it is worth considering the  theory: $$T^+ :=_{\mx{df}} \ZFC + \RC + 2^{\omega} = \omega_2,$$ which is a genuine and powerful alternative to $\ZFC + \MM$ (and not a consequence of Laver-generic largeness). By results of Torres-P{\'e}rez and Wu \cite{MR3600760}, the strong form of Chang's conjecture with $\neg \CH$ implies the strong tree property at $\omega_2$, $\TP_{\omega_2}$ in our notation. Thus $T^+$ proves $\TP_{\omega_2}$, unifying some logical and mathematical principles.\footnote{However, it is known that $T^+$ does not prove $\ITP_{\omega_2}$ by Zhang \cite[Theorem 2.2]{MR4094551}. The reason is that $\RC$ is consistent just from strong compactness and thus principles related to supercompactness appear to be outside its reach.} The theory $T^+$ is also  much stronger in terms of the (lack) of indestructibility results we discussed in this article: $T^+$ is destroyed by adding just one new real and hence obtaining independence of mathematical statements from $T^+$ is harder.\footnote{\label{ft:real} This follows from \cite[Theorem 6.4]{Tdich} which shows that under $\RC$ transitive models computing correctly $\omega_2$ must contain all the reals. Note that $\MM$ is destroyed by adding a Cohen real, but it seems to be open whether \emph{any} new real destroys $\MM$.} For instance, the methods of proof in Examples 1, 2, and 3 in the previous Section \ref{sec:sep} do not apply to $T^+$. However, it does not mean that $\RC$  decides these problems: we checked by a direct argument  in Example 4 in Section \ref{sec:sep} that $T^+$ does not prove $\BA$, $\WP$, or $\SH$, and very likely (though it is open) it does not prove their negations either. This suggest that---unsurprisingly, considering the forcing content of $\MM$---$\ZFC + \MM$ does have more consequences in mathematics than $T^+$ does.

While there are (at least) two alternatives for $\omega_2$ regarding unifications, there appears to be no well-established candidate for cardinals above $\omega_2$. Still, from the two theories for $\omega_2$, $\ZFC + \MM$ and $T^+$, $T^+$ is the one which has---arguably---more potential to be generalized to higher cardinals. Nothing prevents a straightforward generalization of cardinality concepts which appear in $\RC$\footnote{\label{ft:RC3}In this setting, one can define a three-parameter version of Rado's Conjecture, $\RC(\kappa,\lambda,\mu)$, which asserts that every tree of height $\kappa^+$ which is not special and has size at most $\lambda$ has a subtree of size $<\mu$ which is not special. In this notation, $\RC$ denotes $(\forall \lambda) \RC(\omega,\lambda,\omega_2)$, and $\RC(\lambda)$ denotes $\RC(\omega,\lambda,\lambda)$ (see Theorem \ref{th:RC}). By considering different $\kappa$ and $\mu$, one obtains variations of $\RC$ with different properties. For instance, Switzer (and perhaps others) observed that $(\forall \lambda)\RC(\omega,\lambda,2^\omega)$ is consistent with $2^\omega $ larger than $\omega_2$ by adding supercompact many Cohen reals ($2^\omega$ is equal to $\kappa$, where $\kappa$ is a supercompact cardinal in the ground model). In the context of unifications, it is worth mentioning that it is apparently open whether $\RC(\omega,\lambda,\omega_2)$ and $\RC(\omega_1,\lambda,\omega_3)$ can hold simultaneously. It is possible that the simultaneous Rado Conjecture might have strictly larger consistency strength and stronger consequences than the individual principles (compare with the tree properties at $\omega_2$ and $\omega_3$ in Abraham's paper \cite{ABR:tree}).} or in Laver-generic large cardinal axioms,
as is for instance considered in Fuchino et al.\ \cite{MR4198354} in the context of infinitary logics, while it is known that there are provable restrictions for forcing axioms above $\omega_2$.\footnote{See for instance Shelah \cite{Shelah:FA} and Todor{\v c}evi{\'c} and Xiong \cite{todorčević2020inconsistentforcingaxiomomega2}. However, it is also possible that the by moving to  versions of $\RC$ for higher cardinals, some limitations along these lines will appear for this principle as well (for instance connected to the structure of $[\kappa]^\theta$ for uncountable $\theta$, as we discussed above).} 

Although $T^+$ has potential for generalization, the main research focus remains on $\omega_2$, where the most interesting applications and problems are found. Nonetheless, proving interesting statements from a generalized version of $T^+$ would strengthen its standing as a specific instance of a global pattern.


\providecommand{\bysame}{\leavevmode\hbox to3em{\hrulefill}\thinspace}
\providecommand{\MR}{\relax\ifhmode\unskip\space\fi MR }
\providecommand{\MRhref}[2]{%
  \href{http://www.ams.org/mathscinet-getitem?mr=#1}{#2}
}
\providecommand{\href}[2]{#2}

\end{document}